\documentclass[11pt]{article}%
\usepackage{hyperref}
\usepackage{amsmath}
\usepackage{amsfonts}
\usepackage{amssymb}
\usepackage{graphicx}%
\usepackage[nottoc]{tocbibind}
\usepackage[titletoc,title]{appendix}

\providecommand{\U}[1]{\protect\rule{.1in}{.1in}}
\newtheorem{theorem}{Theorem}

\newtheorem{condition}[theorem]{Condition}

\newtheorem{definition}[theorem]{Definition}

\newtheorem{lemma}[theorem]{Lemma}

\newtheorem{remark}[theorem]{Remark}

\newenvironment{proof}[1][Proof]{\noindent\textbf{#1.} }{\ \rule{0.5em}{0.5em}}

\begin{document}

\title{On constraint preservation and strong hyperbolicity}
\author{J. Fernando Abalos\\FAMAF-CIEM, Universidad Nacional de C\'{o}rdoba, \\Ciudad Universitaria, 5000, C\'{o}rdoba, Argentina}
\date{{November 2021}}
\maketitle

\begin{abstract}
We use partial differential equations (PDEs) to describe physical systems. In
general, these equations include evolution and constraint equations. One
method used to find solutions to these equations is the Free-evolution
approach, which consists in obtaining the solutions of the entire system by
solving only the evolution equations. Certainly, this is valid only when the
chosen initial data satisfies the constraints and the constraints are
preserved in the evolution. In this paper, we establish the sufficient
conditions required for the PDEs of the system to guarantee the constraint
preservation. This is achieved by considering quasi-linear first-order PDEs,
assuming the sufficient condition and deriving strongly hyperbolic first-order
partial differential evolution equations for the constraints. We show that, in
general, these constraint evolution equations correspond to a family of
equations parametrized by a set of free parameters. We also explain how these
parameters fix the propagation velocities of the constraints.

As application examples of this framework, we study the constraint
conservation of the Maxwell electrodynamics and the wave equations in
arbitrary space-times. We conclude that the constraint evolution equations are
unique in the Maxwell case and a family in the wave equation case. 

\newpage

\end{abstract}
{\small\tableofcontents}

\section{Introduction}

Physical systems are described by sets of partial differential equations
(PDEs). In general, these equations can be expressed as first-order partial
derivatives, leading us to focus on the study of first-order quasi-linear
PDEs. These equations (or systems) usually include gauge freedoms and
differential constraints (see \cite{geroch1996partial},
\cite{Hilditch:2013ila}). Namely, degrees of freedom (variables) that are not
determined by the equations and degrees of freedom that are restricted to a
lower-dimensional subspace, respectively. In this paper we do not consider
gauge freedoms, we assume that they have already been fixed if the system
includes them, and instead we concentrate on the study of the constraints.
Furthermore, in our development, we do not assume the presence of a background
metric, so the relativistic systems are included in our analysis but we do not
restrict exclusively to them.

To study the above mentioned physical systems, the space-time $M$ is foliated
as $M=\left[  0,T\right]  \times\underset{0\leq t\leq T}{\cup}\Sigma_{t}$ and
the PDEs are divided into two subsets, evolution and constraints equations.
The evolution equations determine how the variables change along the different
hypersurfaces $\Sigma_{t}$, with $0<t\leq T$, while the constraint equations
restrict the allowed values of the variables on each $\Sigma_{t}$. In order
for these two subsets to have predictive power, the associated Cauchy problem
should be well-posed (see \cite{gustafsson1995time}, \cite{kreiss2004initial},
\cite{sarbach2012continuum}). This means that given an initial data over
$\Sigma_{t}$, the associated solutions of the evolution equations exist, are
unique, continuous with respect to the initial data and satisfy the
constraints. Although, we emphasize that the evolution equations describing a
physical system are not unique, since they can be modified by adding
constraint terms to them. Such freedom can lead to well- or ill-posed
evolution equations \cite{kreiss2004initial}, \cite{abalos2017necessary}. In
this paper, we restrict our attention to one class of equations within the
well-posed ones, the strongly hyperbolic (SH). In this class, the recipe for
finding strongly hyperbolic evolution equations has been given in
\cite{Abalos:2018uwg}.

Several methods are used to find numerical and analytical solutions to the
PDEs. Among the most widely used is the free-evolution approach. In this
approach, one begins by verifying that when constraints are initially
satisfied they will remain preserved during the evolution, a property commonly
referred to as constraint preservation (or conservation). Then, one solves the
evolution equations for initial data which satisfy the constraints; and due to
the constraints preservation, this method automatically yields solutions of
the complete system. Some relevant numerical implementations of this approach
can be founded in reviews \cite{Frauendiener:2011zz, Lehner:2014asa,
sarbach2012continuum, shibata2015numerical, Palenzuela:2020tga} and references therein.

The present paper addresses the study of the constraint preservation from a
PDEs point of view, by establishing sufficient conditions that the equations
of a physical system have to satisfy to guarantee the constraint preservation.
The standard method to verify this conservation is: 1) deriving a set of
evolution equations for the constraints, 2) establishing that zero constraints
are a solution of this system and 3) checking that this system is strongly
hyperbolic. When these steps are satisfied, they analytically guarantee that
the vanishing solution is unique and therefore that the constraints are
preserved. Moreover, the strong hyperbolicity is used to find which boundary
conditions preserve the constraints (see \cite{Calabrese:2002xy},
\cite{Tarfulea:2013qjq}, \cite{sarbach2012continuum} \cite{Calabrese:2002xy}).

What do we refer to by "a set of evolution equations for the constraints"?. It
is a set of partial differential evolution equations whose variables are the
constraints. This set of equations is called subsidiary system (SS). \ Most
well-known physical systems have quasi-linear first-order partial differential
subsidiary systems. Some examples are: Maxwell \cite{Calabrese2004ARF}, and
Non-linear \cite{PhysRevD.92.084024} Electrodynamics, Einstein
\cite{Frittelli:1996nj}, \cite{Shinkai:2008yb}, Einstein-Christoffel
\cite{Calabrese:2002xy}\ ,BSSN \cite{Yoneda:2002kg}, \cite{Beyer2004OnTW}, ADM
\cite{Kidder:2001tz}, \cite{Yoneda:2001iy}, $n+1$ ADM \cite{Shinkai:2003vu},
$f\left(  R\right)  $-Gravity \cite{Paschalidis:2011ww},
\cite{Mongwane:2016qtz}, Bimetric Relativity \cite{Kocic:2018yvr} theories,
etc. The process for deriving the SS and verifying its strong hyperbolicity
is\ conducted separately for each physical theory and usually involves very
cumbersome calculations.To assess this problem, we present a theory that
simplifies and automatizes the process.

In \cite{reula2004strongly}, Reula studies the constant coefficient case,
assumes the existence of a first-order SS and explains how the characteristic
structure of this system connects with the characteristic structure of the
evolution equations of the system. Here, we focus on the quasi-linear case,
show which conditions guarantee the existence of a SS and explain how the
Reula result arises naturally from our results. Indeed, we show that the
principal symbol of the SS is a simple projection of some tensorial objects
$C_{A}^{a\Gamma}$ called here Geroch fields (see \cite{geroch1996partial}). We
also show that the resulting SSs are not unique. In fact, they are families of
equations with free parameters letting to choose the propagation velocities of
the constraints.

Finally, we restrict ourselves to the constant-coefficient case to simplify
the discussion and give sufficient conditions that the original set of PDEs
has to satisfy for having SH evolution equations and an associated SH
subsidiary system. To reach this result we give a detailed description of the
characteristic structure of both subsystems and explain how they should be
chosen to make them SH.

The constraints studied here are called first-class constraints in their
Hamiltonian version, see for example \cite{Henneaux:1992ig},
\cite{Giulini:2015qha}, \cite{Hilditch:2013ila}. We highlight the work of
Hilditch and Richter \cite{Hilditch:2013ila}, where a similar problem to the
one presented here is studied with a different approach.

There are widely used methods that introduce extra variables to the system and
avoid dealing with constraints. Some of the most popular ones are the
$\lambda-$systems \cite{brodbeck1999einstein}, the divergence cleaning
\cite{munz2000three}, \cite{dedner2002hyperbolic} for electrodynamics, the Z4
systems\ \cite{Bona:2002fq, Bona:2002ft, Bona:2003fj, Bona:2004yp,
Bona:2004ky, Alic:2011gg, Gundlach:2005eh, Bernuzzi:2009ex} and the modified
harmonic gauge \cite{Kovacs:2020ywu} for Einstein equations. However, for
these systems to work properly, the constraints conservation of the original
system has to be satisfied. Therefore, this aspect of the problem is crucial.

The outline of the paper is the following. In Section \ref{Seccion_Setting_1},
we introduce the Geroch fields, used to obtain the expressions for the
constraints; the reductions, that select the evolution equations; and the
integrability conditions of the system, which allow to obtain first-order
partial derivatives SSs. In section \ref{Teorema_1}, we introduce the
expressions of the SSs in the first main theorem and the proof to this
theorem. In sections \ref{Sec_Const_coeff_1} and \ref{teorema_2}, we present
the well-posed concepts and our second main theorem about the strong
hyperbolicity of the SS, respectively. This section include the subsection
\ref{Proof_Theorem_coef_const_2}, where we provide the proof of the second
main theorem. This proof is splited in many subsubsections which include the
analysis of the Kronecker decomposition of the principal symbol of the system
and of the SS (including the constraints of the constraints). In subsection
\ref{teo_3}, we discuss how to suppress one condition of the second main
theorem. In section \ref{Examples}, we present two examples of application of
the developed formalism: the Maxwell electrodynamics and the wave equation. In
section \ref{Conclusions}, we briefly discuss the results and make some
comments on future work. In appendix \ref{Ap_lemmas}, we introduce some lemmas
that are used to find the Kronecker decomposition of a pencil. They are used
in the proof of Theorem \ref{teorema_2}. Finally, in appendix
\ref{App_coordenadas}, we discuss the introduction of the lapse and shift
variables in the $n+1$ foliations. This is done in the simple system
$\nabla_{a}q^{b}=0$, in cases with and without a metric.

\section{Setting\label{Seccion_Setting_1}}

Following Geroch's notation, let $b\overset{\pi}{\rightarrow}M$ a fiber bundle
over a space-time $M$ with $\dim M=n+1$. We consider the cross-section
$\phi^{\alpha}:M\rightarrow b$, which defines the physics fields and satisfies
the equation of motion
\begin{equation}
E^{A}:=\mathfrak{N}_{~\alpha}^{Aa}\left(  x,\phi\right)  \nabla_{a}%
\phi^{\alpha}-J^{A}\left(  x,\phi\right)  =0. \label{eq_sys_1}%
\end{equation}
Here the Greek indices $\alpha,\beta,\gamma\,$represent field indices, lower
letters $a,b,c,...$ represent space-time indices and capital letters
$A,B,C,...$ represent multi-tensorial indices on the fiber space of equations
call $\Psi_{L}$. We are considering $\dim\left(  A\right)  =e$ and
$\dim\left(  \alpha\right)  =u$ such that $e\geq u$. The tensor fields
$\mathfrak{N}_{~\alpha}^{Aa}$ and $J^{A}$ are the principal symbol and the
source term respectively, they do not depend on derivatives of $\phi$. \ The
derivative $\nabla_{a}$ represent a partial derivative or a Levi Civitta
connection when the system has \ a metric.

Solving $E^{A}=0$, with a given initial condition for unknown fields
$\phi^{\alpha}$, is called the initial value problem of the system
(\ref{eq_sys_1}). As we commented in the introduction, this can be done
numerically or analytically using the free-evolution approach, i.e.,
separating $E^{A}=0$ into two sets of equations, the evolutions and the
constraints, finding the Subsidiary System (SS), showing the constraint
conservation and using the evolution equations to find the solutions.
Following this path, we present in the first two sections, and from a PDE
perspective, a theory that introduces the sufficient conditions by which a
system such as (\ref{eq_sys_1}) has a SS of first-order partial derivatives.
We begin, this section, introducing the main tools of the article, defining
the evolution, the constraints and the integrability conditions associated to
(\ref{eq_sys_1}).

\subsection{$n+1$ foliation\label{n+1_decomposition_sec}}

We introduce a foliation of $M$ given by the level surfaces of a function
$t:M\rightarrow%
\mathbb{R}
$, and call $\Sigma_{t_{0}}=\left\{  p\in M\mid t\left(  p\right)
=t_{0}\right\}  $ to these hypersurfaces. We also introduce coordinates
$x^{a}=\left(  t,x^{i}\right)  ,$ $i=1,..,n,$ with $x^{i}$ adapted to the
$\Sigma_{t_{0}}$'s. This is a local $n+1$ foliation of $M$ given by
$[0,T]\times\Sigma_{t_{0}}$ with $T\in%
\mathbb{R}
.$ We call \textit{time coordinate} to $t$ and \textit{spatial coordinates} to
$x^{i}$; these names are just names to differentiate one coordinate from the
others since we are not assuming the presence of a background metric here. We
also consider the vector $t^{a}:=\left(  \partial_{t}\right)  ^{a}\dot
{=}\left(  1,0,...,0\right)  $, and co-vector $n_{a}:=\nabla_{a}t\dot
{=}\left(  1,0,...,0\right)  $, where the dot in the equal sign means "in this
coordinates the explicit expressions are". They satisfy the condition
$t^{a}n_{a}=1$, which allows to definite the projector
\begin{equation}
\eta_{b}^{a}:=\delta_{b}^{a}-t^{a}n_{b}\dot{=}\left[
\begin{array}
[c]{cccc}%
0 & 0 & 0 & 0\\
0 & 1 & 0 & 0\\
0 & 0 & ... & 0\\
0 & 0 & 0 & 1
\end{array}
\right]  . \label{Eq_eta_1}%
\end{equation}
It has the properties $\eta_{b}^{a}t^{b}=0=\eta_{b}^{a}n_{a}$ and $\eta
_{b}^{a}\eta_{c}^{b}=\eta_{c}^{a}$. We use them to rewrite (\ref{eq_sys_1}),
as follows%
\begin{align}
E^{A}  &  =\mathfrak{N}_{~\alpha}^{Ac}\left(  n_{c}t^{b}+\eta_{c}^{b}\right)
\nabla_{b}\phi^{\alpha}-J^{A},\label{eq_split_E_1}\\
&  =\mathfrak{N}_{~\alpha}^{Ac}n_{c}\nabla_{t}\phi^{\alpha}+\mathfrak{N}%
_{~\alpha}^{Ac}\eta_{c}^{b}\nabla_{b}\phi^{\alpha}-J^{A}=0
\label{eq_split_E_2}%
\end{align}
In this equation the derivatives are splitting in time and spatial derivatives
since the term $\eta_{c}^{b}\nabla_{b}$ has no derivatives in the $t^{c}$
direction ($t^{c}\eta_{c}^{b}\nabla_{b}=0$). Notice that if we use the
coordinates $\left(  t,x^{i}\right)  $ the last expression can be rewritten
as
\[
E^{A}\dot{=}\mathfrak{N}_{~\alpha}^{A0}\nabla_{t}\phi^{\alpha}+\mathfrak{N}%
_{~\alpha}^{Ai}\nabla_{i}\phi^{\alpha}-J^{A}=0,
\]
where $i=1,...,n$ and
\[
\mathfrak{N}_{~\alpha}^{A0}:=\mathfrak{N}_{~\alpha}^{Ac}n_{c}.
\]
In these expressions the temporal and spatial derivatives are explicit.
Without loss of generality, we will assume along this work that we are in
these coordinates. \ Nevertheless, to simplify the notation we will suppress
the dot in the symbol $\dot{=}$ and the use of $\left(  t,x^{i}\right)  $ will
be understood from the context.

\subsection{Geroch fields and constraints}

We introduce now the main objects of this article, the Geroch fields
$C_{A}^{a}\left(  x,\phi\right)  $, which are used to define the constraints
as we explain below (see \cite{geroch1996partial} too). These Geroch fields
are defined by the equation
\begin{equation}
C_{A}^{(a}\mathfrak{N}_{~\alpha}^{\left\vert A\right\vert b)}=0.
\label{Eq_C_K_nueva_1}%
\end{equation}
\bigskip At each point $\kappa=\left(  x,\phi\right)  \in b$, they define a
vector space\footnote{Notice that given $X_{A}^{a}$ and $Y_{A}^{a}$ satisfying
(\ref{Eq_C_K_nueva_1}), any linear combination $\alpha X_{A}^{a}+\beta
Y_{A}^{a}$ with $\alpha.\beta\in%
\mathbb{R}
$ also satisfies (\ref{Eq_C_K_nueva_1}).}.

We introduce now a condition about these objects that we will explain after
the definition of the constraints. For every open $U$ of $M$, we assume that
it is possible to choose%
\begin{equation}
c:=e-u \label{Eq_c_e_u_1}%
\end{equation}
Geroch fields $C_{A}^{\Gamma a}$, denoted by the index $\Gamma$, and such that
for all $x\in U$ and all $\phi$, the components
\[
C_{A}^{\Gamma0}:=C_{A}^{\Gamma a}n_{a}%
\]
are linearly independent. This field $C_{A}^{\Gamma0}$ can be thought of as a
linear transformation from the index $A$ to the index $\Gamma$, so the above
condition is equivalent to:

\begin{condition}
\label{cod_C0_rank_max_b}For each $\kappa=\left(  x,\phi\right)  \in b$,
\begin{equation}
C_{A}^{\Gamma0}\text{ has maximal rank\footnote{There is no $X_{\Gamma}$ such
that $X_{\Gamma}n_{a}C_{A}^{\Gamma a}=0$ to except of the trivial one
$X_{\Gamma}=0$}.} \label{cod_C0_rank_max}%
\end{equation}

\end{condition}

These Geroch fields $C_{A}^{\Gamma a}$ satisfy the equation
\begin{equation}
C_{A}^{\Gamma(a}\mathfrak{N}_{~\alpha}^{\left\vert A\right\vert b)}=0,
\label{eq_CK_1}%
\end{equation}
and define the \textit{constraints} of the system as follows
\[
\psi^{\Gamma}:=n_{a}C_{A}^{\Gamma a}E^{A}.
\]

These $\psi^{\Gamma}$ are called the constraints equations since they have not
time derivatives, they only have spatial derivatives as follow.
\begin{align}
\psi^{\Gamma}  &  =n_{a}C_{A}^{\Gamma a}E^{A}\nonumber\\
&  =n_{a}C_{A}^{\Gamma a}\mathfrak{N}_{~\alpha}^{Ac}n_{c}\nabla_{t}%
\phi^{\alpha}+n_{a}C_{A}^{\Gamma a}\mathfrak{N}_{~\alpha}^{Ac}\eta_{c}%
^{b}\nabla_{b}\phi^{\alpha}-n_{a}C_{A}^{\Gamma a}J^{A}\nonumber\\
&  =n_{a}C_{A}^{\Gamma a}\mathfrak{N}_{~\alpha}^{Ac}\eta_{c}^{b}\nabla_{b}%
\phi^{\alpha}-n_{a}C_{A}^{\Gamma a}J^{A}\label{eq_def_constr_1}\\
&  =C_{A}^{\Gamma0}\mathfrak{N}_{~\alpha}^{Ai}\nabla_{i}\phi^{\alpha}%
-C_{A}^{\Gamma0}J^{A}=0.\nonumber
\end{align}
Where equation (\ref{eq_split_E_2}) was used in the first line and $n_{a}%
C_{A}^{\Gamma a}\mathfrak{N}_{~\alpha}^{Ac}n_{c}=0$ (follows directly from
(\ref{eq_CK_1})) in the fourth line. This means that they can be calculated
only with the information of $\phi^{\alpha}$ pullbacked to the hypersurfaces
$\Sigma_{t}$.

Notice that these constraints depend on the hypersurfaces considered through
$n_{a}$. In general and when we have a metric, these hypersurfaces are chosen
as spatial, however, there is no restriction on $n_{a}$ to be temporal so by
changing the choice of $n_{a}$ we could obtain constraints associated with any hypersurface.

The motivation of the introduction of condition \ref{cod_C0_rank_max} is that
it implies that the $\psi^{\Gamma}$ are algebraically independents and the
number of them is $c$. This is a natural condition in physical examples as
General Relativity. Recalling that $c$ is the difference between the number of
equations and the fields (see eq. (\ref{Eq_c_e_u_1})), we observe that there
remain $u$ equations in $E^{A}$, these will be the evolution equations.

On the other hand, we notice that assumption (\ref{cod_C0_rank_max}) is a
non-covariant request, however, we will explain how to recover the covariance
in subsection \ref{Sec_Vin_M}. It will be with the introduction of another
kind of Geroch fields.

\begin{remark}
A set of PDE's (as (\ref{eq_sys_1})) that includes constraints might not admit
Geroch fields. In this case, the system should accept another kind of fields
$C_{A}^{a_{1}a_{2}...a_{m}}$ used to define the constraints. Such systems are
not considered here and will be included in future work. However, this does
not represent a significant loss of generality since, as Geroch showed in
\cite{geroch1996partial}, most of the classical physical systems have only
Geroch fields.
\end{remark}

\subsection{Non algebraic constraints}

We also assume, in this work, that there is no $X_{A}$ such that%
\begin{equation}
X_{A}\mathfrak{N}_{~\alpha}^{Aa}=0. \label{eq_X_K_1}%
\end{equation}
It is equivalent to claim that the system (\ref{eq_sys_1}) has no algebraic
constraints since if such $X_{A}$ exists, then the expression
\begin{equation}
X_{A}E^{A}:=-X_{A}J^{A}=0 \label{eq_alg_con_1}%
\end{equation}
would be an algebraic constraint. In the cases in which the system has an
$X_{A}$ satisfying (\ref{eq_X_K_1}), we redefine the system to obtain a new
system without algebraic constraints. This is done by suppressing some
equations from (\ref{eq_sys_1}) and adding the derivatives of
(\ref{eq_alg_con_1}) to the system. We end-up with a new system of first order
in derivatives with more equations, equivalent to the original system, but
without the algebraic constraint. We assume that this new system only admits
Geroch fields.

\subsection{Evolution equations\label{Evo}}

We have introduced the constraints, now we introduce the evolution equations.
We begin with an additional assumption,

\begin{condition}
\label{cod_K0_rank_max_b} For each $\kappa\in b$,
\begin{equation}
\mathfrak{N}_{~\beta}^{A0}\text{ has maximal rank\footnote{This means there is
not $\delta\phi^{\alpha}\neq0$ such that $\mathfrak{N}_{~\beta}^{A0}\delta
\phi^{\alpha}=0.$}.} \label{cod_K0_rank_max}%
\end{equation}

\end{condition}

This condition guarantees that we can define evolution equations for each of
the unknown variables \ $\phi^{\alpha}$ as follow.

First, we introduce the \textit{reduction} tensor $h_{A}^{\alpha}$ such that
\begin{equation}
h_{A}^{\alpha}\mathfrak{N}_{~\beta}^{A0}=\delta_{\beta}^{\alpha},
\label{eq_hK0_1}%
\end{equation}
where $\delta_{\beta}^{\alpha}$ is the identity map. This reduction is always
possible to find when condition (\ref{cod_K0_rank_max}) holds.

Now, with the aid of the reduction $h_{A}^{\alpha}$, we choose the
\textit{evolution equations }%
\begin{equation}
e^{\alpha}:=h_{A}^{\alpha}E^{A}=t^{a}\nabla_{a}\phi^{\alpha}+h_{A}^{\alpha
}\mathfrak{N}_{~\beta}^{Ac}\eta_{c}^{b}\nabla_{b}\phi^{\beta}-h_{A}^{\alpha
}J^{A}=0, \label{eq_evol_1}%
\end{equation}
where expressions (\ref{eq_split_E_2}) and (\ref{eq_hK0_1}) were used. These
equations are called evolution equations since they include derivatives in the
$t^{a}$ direction for each $\phi^{\alpha}$, telling how the fields\ $\phi
^{\alpha}$ evolve outside of the hypersurfaces $\Sigma_{t}$ In general, when
the system has a background metric, $t^{a}$ is chosen temporal.

These evolution equations are a linear combination of the eq. (\ref{eq_sys_1})
and, of course, they are not unique. We can define other reductions $\tilde
{h}_{A}^{\alpha}$ as follow%
\[
\tilde{h}_{A}^{\alpha}=h_{A}^{\alpha}+p_{A}^{\alpha},
\]
where $p_{A}^{\alpha}$ satisfies
\[
p_{A}^{\alpha}\mathfrak{N}_{~\beta}^{A0}=0,
\]
and produce another set of evolutions equations
\begin{equation}
\tilde{e}^{\alpha}:=\tilde{h}_{A}^{\alpha}E^{A}=t^{a}\nabla_{a}\phi^{\alpha
}+\left(  h_{A}^{\alpha}\mathfrak{N}_{~\beta}^{Ac}+p_{A}^{\alpha}%
\mathfrak{N}_{~\beta}^{Ac}\right)  \eta_{c}^{b}\nabla_{b}\phi^{\beta}-\left(
h_{A}^{\alpha}+p_{A}^{\alpha}\right)  J^{A}=0. \label{Eq_evo_2}%
\end{equation}
In general, this freedom is used to obtain well-posed evolution equations (see
\cite{Abalos:2018uwg}).

\subsection{Evolution and constraints equations}

By requiring that system (\ref{eq_sys_1}) satisfies conditions
(\ref{cod_C0_rank_max}) and (\ref{cod_K0_rank_max}), we have decomposed
(\ref{eq_sys_1}) into two subsets of equations, the evolution \ ($e^{\alpha}$)
and constraint equations ($\psi^{\Delta}$) as follow
\begin{equation}
\left[
\begin{array}
[c]{c}%
e^{\alpha}\\
\psi^{\Delta}%
\end{array}
\right]  :=\left[
\begin{array}
[c]{c}%
h_{A}^{\alpha}\\
C_{A}^{\Delta0}%
\end{array}
\right]  E^{A}. \label{Eq_e_y_vin_Mat_E_1}%
\end{equation}
Now, we will show that for each $\kappa\in b$, the matrix $\left[
\begin{array}
[c]{c}%
h_{B}^{\alpha}\\
C_{B}^{\Delta0}%
\end{array}
\right]  $ is invertible. This means that $e^{\alpha}$ and $\psi^{\Delta}$ are
linearly equivalent to the whole system.

\begin{lemma}
For each $\kappa\in b$, the matrix $\left[
\begin{array}
[c]{c}%
h_{B}^{\alpha}\\
C_{B}^{\Delta0}%
\end{array}
\right]  $ is invertible, with inverse $\left[
\begin{array}
[c]{cc}%
\mathfrak{N}_{~\beta}^{B0} & h_{~\Gamma}^{B}%
\end{array}
\right]  $, where%
\[%
\begin{array}
[c]{ccc}%
C_{B}^{\Delta0}h_{~\Gamma}^{B}=\delta_{\Gamma}^{\Delta}, &  & h_{B}^{\alpha
}h_{~\Gamma}^{B}=0
\end{array}
\]
and%
\begin{equation}
\left[
\begin{array}
[c]{cc}%
\mathfrak{N}_{~\alpha}^{A0} & h_{~\Delta}^{A}%
\end{array}
\right]  \left[
\begin{array}
[c]{c}%
h_{B}^{\alpha}\\
C_{B}^{\Delta0}%
\end{array}
\right]  =\mathfrak{N}_{~\alpha}^{A0}h_{B}^{\alpha}+h_{~\Delta}^{A}%
C_{B}^{\Delta0}=\delta_{B}^{A} \label{eq_Kh_hC_1}%
\end{equation}

\end{lemma}

\begin{proof}
To show that $\left[
\begin{array}
[c]{c}%
h_{B}^{\alpha}\\
C_{B}^{\Delta0}%
\end{array}
\right]  $ is invertible, we will assume that it is not, and conclude an
absurd. So, if $\left[
\begin{array}
[c]{c}%
h_{B}^{\alpha}\\
C_{B}^{\Delta0}%
\end{array}
\right]  $ is not invertible, there exists $\left[
\begin{array}
[c]{cc}%
X_{\alpha} & Y_{\Delta}%
\end{array}
\right]  \neq0$ such that
\begin{equation}
X_{\alpha}h_{B}^{\alpha}+Y_{\Delta}C_{B}^{\Delta0}=0. \label{eq_ld_h_C0}%
\end{equation}

\sloppy
Multiplying this expression by $\mathfrak{N}_{~\beta}^{B0}$ and using the
equations (\ref{eq_hK0_1}) and  $C_{A}^{\Gamma0}\mathfrak{N}_{~\alpha}^{A0}=0$
we conclude
\[
X_{\alpha}=0.
\]
This means that $Y_{\Delta}C_{B}^{\Delta0}=0$, but since $C_{B}^{\Delta0}$ has
only trivial kernel in the $\Delta$ index, it follows $Y_{\Delta}=0$. So, the
unique solution of eq. (\ref{eq_ld_h_C0}) is $\left[
\begin{array}
[c]{cc}%
X_{\alpha} & Y_{\Delta}%
\end{array}
\right]  =0$, which is a contradiction. We conclude that $\left[
\begin{array}
[c]{c}%
h_{B}^{\alpha}\\
C_{B}^{\Delta0}%
\end{array}
\right]  $ is invertible.

If $\left[
\begin{array}
[c]{cc}%
\mathfrak{N}_{~\beta}^{B0} & h_{~\Gamma}^{B}%
\end{array}
\right]  $ is the inverse, the product
\[
\left[
\begin{array}
[c]{c}%
h_{B}^{\alpha}\\
C_{B}^{\Delta0}%
\end{array}
\right]  \left[
\begin{array}
[c]{cc}%
\mathfrak{N}_{~\beta}^{B0} & h_{~\Gamma}^{B}%
\end{array}
\right]  =\left[
\begin{array}
[c]{cc}%
h_{B}^{\alpha}\mathfrak{N}_{~\beta}^{B0} & h_{B}^{\alpha}h_{~\Gamma}^{B}\\
C_{B}^{\Delta0}\mathfrak{N}_{~\beta}^{B0} & C_{B}^{\Delta0}h_{~\Gamma}^{B}%
\end{array}
\right]  ,
\]
should be equal to the identity matrix, then $C_{B}^{\Delta0}h_{~\Gamma}%
^{B}=\delta_{\Gamma}^{\Delta}$, $h_{B}^{\alpha}h_{~\Gamma}^{B}=0.$

On the other hand, commuting the matrices we obtain the following equation%
\begin{align*}
\delta_{B}^{A}  &  =\left[
\begin{array}
[c]{cc}%
\mathfrak{N}_{~\alpha}^{A0} & h_{~\Delta}^{A}%
\end{array}
\right]  \left[
\begin{array}
[c]{c}%
h_{B}^{\alpha}\\
C_{B}^{\Delta0}%
\end{array}
\right]  ,\\
&  =\mathfrak{N}_{~\alpha}^{A0}h_{B}^{\alpha}+h_{~\Delta}^{A}C_{B}^{\Delta0}.
\end{align*}

\end{proof}

\subsection{Non unicity in the Geroch fields\label{Sec_Vin_M}}

Frequently, the constraints have certain differential relationships between
them on the hypersurfaces $\Sigma_{t}$. The canonical example is when a
second-order derivative system is reduced to first-order by defining the first
derivatives as new variables. This approach reduces the order of the
derivatives and at the same time introduces extra constraints to the system.
These extra constraints also satisfy certain differential relationships
between them, their cross-derivatives should vanish. The way we parameterize
this, and any other, differential relationship between the constraints is
through another class of Geroch fields that we call $M_{A}^{\tilde{\Delta}%
a}\left(  x,\phi\right)  $, where the $\tilde{\Delta}$ index numerates them.

For each $\kappa=\left(  x,\phi\right)  \in b$, these new extra fields satisfy
the equation (\ref{eq_CK_1}), which means
\begin{equation}
M_{A}^{\tilde{\Delta}(a}\mathfrak{N}_{~\alpha}^{\left\vert A\right\vert b)}=0,
\label{eq_M_K_1}%
\end{equation}
and also satisfy the following condition
\begin{equation}
M_{A}^{\tilde{\Delta}a}\mathfrak{N}_{~\alpha}^{A0}=0. \label{eq_M_K0_1}%
\end{equation}
Notice that, for each $\kappa,$ the equations (\ref{eq_M_K_1}) and
(\ref{eq_M_K0_1}) define a vector subspace within the Geroch field space.

To simplify the analysis, we add an assumption on the $M_{A}^{\tilde{\Delta}%
a}$ which also defines the $\tilde{\Delta}$ index.

\begin{condition}
\label{As_M_1_b}In every open $U\subset M$ and for all $\kappa=\left(
x,\phi\right)  \in b$ with $x\in U$, there exist a finite number of fields
$M_{A}^{\tilde{\Delta}a}$ indexed by $\tilde{\Delta}$ such that any field
satisfying (\ref{eq_M_K_1}) and (\ref{eq_M_K0_1}) can be expanded as a linear
combination of these $M_{A}^{\tilde{\Delta}a}$. This means that any field
satisfying (\ref{eq_M_K_1}) and (\ref{eq_M_K0_1}) is a linear combination of
the $M_{A}^{\tilde{\Delta}a}$.
\end{condition}

This assumption is usually satisfied in physical systems in such a way that
the pair $\left(  \Gamma,\tilde{\Delta}\right)  $ define a tensor index. Thus,
the covariance of these Geroch tensors is recovered by considering the pair
$C_{A}^{\Gamma a}$ and $M_{A}^{\tilde{\Delta}a}$ together. We will comment
more about this at the end of this subsection.

Another result that arises from this assumption is:

\begin{lemma}
If the conditions \ref{cod_C0_rank_max}, \ref{cod_K0_rank_max} and
\ref{As_M_1_b} hold, then in every open $U\subset M$ and for all
$\kappa=\left(  x,\phi\right)  \in b$ with $x\in U$, every Geroch field
$X_{A}^{a}$ (satisfying the eq. (\ref{Eq_C_K_nueva_1})) is a linear
combination of $C_{A}^{\Gamma a}$ and $M_{A}^{\tilde{\Delta}a}$, i.e. there
exist always $S_{\Delta}$ and $S_{\tilde{\Delta}}$ such that
\[
X_{A}^{a}=S_{\Delta}C_{A}^{\Delta a}+S_{\tilde{\Delta}}M_{A}^{\tilde{\Delta}%
a}.
\]

\end{lemma}

\begin{proof}
Let $X_{A}^{a}$ be any Geroch field, we have to show that there exist
$S_{\Delta}$ and $S_{\tilde{\Delta}}$ such that $X_{A}^{a}=S_{\Delta}%
C_{A}^{\Delta a}+S_{\tilde{\Delta}}M_{A}^{\tilde{\Delta}a}.$ We propose
$S_{\Delta}=X_{B}^{0}h_{~\Delta}^{B}$ and point out that to conclude the
proof, we only need to show that
\begin{equation}
X_{A}^{a}-\left(  X_{B}^{0}h_{~\Delta}^{B}\right)  C_{A}^{\Delta a}
\label{eq_proy_a_M_1}%
\end{equation}
satisfies (\ref{eq_M_K_1}) and (\ref{eq_M_K0_1}). Since, if this is true, by
the condition \ref{As_M_1_b} the field (\ref{eq_proy_a_M_1}) should be
expanded by $M_{A}^{\tilde{\Delta}a}$. Showing (\ref{eq_M_K_1}) is trivial
from the definitions of Geroch fields, so it only remains to demonstrate
(\ref{eq_M_K0_1}). Multiplying (\ref{eq_proy_a_M_1}) by $\mathfrak{N}%
_{~\alpha}^{A0},$ we obtain
\begin{align*}
\left(  X_{A}^{a}-\left(  X_{B}^{0}h_{~\Delta}^{B}\right)  C_{A}^{\Delta
a}\right)  \mathfrak{N}_{~\alpha}^{A0}  &  =-\left(  X_{A}^{0}-\left(
X_{B}^{0}h_{~\Delta}^{B}\right)  C_{A}^{\Delta0}\right)  \mathfrak{N}%
_{~\alpha}^{Aa},\\
&  =-X_{B}^{0}\left(  \delta_{A}^{B}-h_{~\Delta}^{B}C_{A}^{\Delta0}\right)
\mathfrak{N}_{~\alpha}^{Aa},\\
&  =-X_{B}^{0}\mathfrak{N}_{~\alpha}^{B0}h_{A}^{\alpha}\mathfrak{N}_{~\alpha
}^{Aa},\\
&  =0.
\end{align*}
We used $X_{A}^{a}\mathfrak{N}_{~\alpha}^{A0}=-X_{A}^{0}\mathfrak{N}_{~\alpha
}^{Aa}$ and $C_{A}^{\Delta a}\mathfrak{N}_{~\alpha}^{A0}=-C_{A}^{\Delta
0}\mathfrak{N}_{~\alpha}^{Aa}$ in the first equality, (\ref{eq_Kh_hC_1}) (only
valid assuming \ref{cod_C0_rank_max}, \ref{cod_K0_rank_max}) in the second
one, (\ref{eq_Kh_hC_1}) in the third one and $X_{B}^{0}\mathfrak{N}_{~\alpha
}^{B0}=0$ in the last one. This concludes the proof.
\end{proof}

We said that the $M_{A}^{\tilde{\Delta}a}$ parameterize the differential
relationships between the constraints, their explicit expressions can be found
in subsection (\ref{teorema_ec_ev_vin_1_a}), eq. (\ref{Eq_vinc_de_los_vinc_1}%
). It is important to mention that to reach those expressions we need the
integrability condition (\ref{eq_int_LE_2}) which appears in the next
subsection. From now on we will refer to (\ref{Eq_vinc_de_los_vinc_1})\ as the
\textit{constraints of the constraints.}

Using the identity (\ref{eq_Kh_hC_1}), $M_{A}^{\tilde{\Delta}a}$ can be
rewritten as%
\begin{align}
M_{B}^{\tilde{\Delta}a}  &  =M_{A}^{\tilde{\Delta}a}\left[
\begin{array}
[c]{cc}%
\mathfrak{N}_{~\alpha}^{A0} & h_{~\Gamma}^{A}%
\end{array}
\right]  \left[
\begin{array}
[c]{c}%
h_{B}^{\alpha}\\
C_{B}^{\Gamma0}%
\end{array}
\right] \nonumber\\
&  =M_{A}^{\tilde{\Delta}a}h_{~\Gamma}^{A}C_{B}^{\Gamma0}, \label{Eq_M_h_C0_1}%
\end{align}
where the equation (\ref{eq_M_K0_1}) has been used. So, by redefining
\begin{equation}
M_{\Gamma}^{\tilde{\Delta}a}:=M_{A}^{\tilde{\Delta}a}h_{~\Gamma}^{A}
\label{Eq_M_n_1}%
\end{equation}
and using the eq. (\ref{eq_CK_1}) we derive an equation for $M_{\Gamma
}^{\tilde{\Delta}a}$%

\begin{equation}
M_{\Gamma}^{\tilde{\Delta}(a}C_{A}^{\left\vert \Gamma0\right\vert
}\mathfrak{N}_{~\alpha}^{\left\vert A\right\vert b)}=0. \label{eq_M_C0_K_1}%
\end{equation}
There is an isomorphism between $M_{\Gamma}^{\tilde{\Delta}a}$ and
$M_{B}^{\tilde{\Delta}a}$, in the sense that one of them completely define the
other. For each $M_{\Gamma}^{\tilde{\Delta}a}$ satisfying the eq.
(\ref{eq_M_C0_K_1}), $M_{B}^{\tilde{\Delta}a}$ can be defined by the equation
(\ref{Eq_M_h_C0_1}), and for each $M_{B}^{\tilde{\Delta}a}$ satisfying
(\ref{eq_M_K_1}) and (\ref{eq_M_K0_1}),\ \ $M_{\Gamma}^{\tilde{\Delta}a}$
\ can be defined by $M_{A}^{\tilde{\Delta}a}h_{~\Gamma}^{A}$. In addition,
considering the following expression,
\[
M_{A}^{\tilde{\Delta}a}\left[
\begin{array}
[c]{cc}%
\mathfrak{N}_{~\alpha}^{A0} & h_{~\Gamma}^{A}%
\end{array}
\right]  =\left[
\begin{array}
[c]{cc}%
0 & M_{\Gamma}^{\tilde{\Delta}a}%
\end{array}
\right]  ,
\]
where we have used that $M_{A}^{\tilde{\Delta}a}\mathfrak{N}_{~\alpha}^{A0}%
=0$; and recalling that $\left[
\begin{array}
[c]{cc}%
\mathfrak{N}_{~\alpha}^{A0} & h_{~\Gamma}^{A}%
\end{array}
\right]  $ is invertible, the last expression means that any set of
$M_{\Gamma}^{\tilde{\Delta}a}$'s is linearly independent if and only if the
associated set of $M_{B}^{\tilde{\Delta}a}$'s is linearly independent too.
From now on, we will regard $M_{B}^{\tilde{\Delta}a}$ and $M_{\Gamma}%
^{\tilde{\Delta}a}$ as equivalent to simplify the discussion.

The $M_{\Gamma}^{\tilde{\Delta}0}$ has another interesting property. By
employing eq. (\ref{eq_M_C0_K_1}), we reach to the following lemma.

\begin{lemma}
\label{Lema_M0} $M_{\Gamma}^{\tilde{\Delta}0}=0$ (or $M_{A}^{\tilde{\Delta}%
0}=0$).
\end{lemma}

\begin{proof}
Considering the $0,b$ components of eq. (\ref{eq_M_C0_K_1}) and the equation
$C_{A}^{\Gamma0}\mathfrak{N}_{~\alpha}^{A0}=0$, we conclude
\[
M_{\Gamma}^{\tilde{\Delta}0}C_{A}^{\Gamma0}\mathfrak{N}_{~\alpha}%
^{Ab}=-M_{\Gamma}^{\tilde{\Delta}b}C_{A}^{\Gamma0}\mathfrak{N}_{~\alpha}%
^{A0}=0.
\]
Since we are working with systems that have not algebraic constraints, there
is not $X_{A}$ such that eq. (\ref{eq_X_K_1}) holds, so it should be valid
\[
M_{\Gamma}^{\tilde{\Delta}0}C_{A}^{\Gamma0}=0.
\]
On the other hand, the $C_{A}^{\Gamma0}$ has only trivial kernel in the
$\Gamma$ index, thus we finally conclude $M_{\Gamma}^{\tilde{\Delta}0}=0$.
\end{proof}

Notice that this result allows us to redefine the Geroch's fields
$C_{A}^{\Gamma a}$ in the following way%
\[
\tilde{C}_{A}^{\Gamma a}=C_{A}^{\Gamma a}+N_{\tilde{\Delta}}^{\Gamma}%
M_{A}^{\tilde{\Delta}a},
\]
leaving the $C_{A}^{\Gamma0}$ component unchanged and keeping the same
expressions for the constraints eq. (\ref{eq_def_constr_1}). Here, the
$N_{\tilde{\Delta}}^{\Gamma}$ fields can be freely chosen and the reader can
see that the Geroch fields have significant freedom in their definition. This
freedom will also appear in the evolution equations of the constraints.

Finally, we conclude the discussion of the covariance in the Geroch fields. So
far, it has been requested that $C_{A}^{\Delta0}$ has maximal rank and has
shown that $M_{A}^{\tilde{\Gamma}0}=0$, but these are not covariant conditions
since the $n_{a}$ direction is privileged among others in these two objects.
In general, one can recover the covariance, thinking that we have split the
real tensors (the ones which transform good under change of coordinates) into
two parts, the $C_{A}^{\Delta a}$ and $M_{B}^{\tilde{\Delta}a}$. We will not
pay attention to this non-covariant splitting in the future and we will call
fields or tensors to $C_{A}^{\Delta a}$ and $M_{B}^{\tilde{\Delta}a}$, even if
they do not transform as real ones.

We have said nothing about how many derivatives admit the objects considering
until here, so, we will assume that \ $\mathfrak{N}_{~\alpha}^{Ac}$, $J^{A}$,
$C_{A}^{\Delta a}$, $M_{B}^{\tilde{\Delta}a}$, $h_{B}^{\alpha}$ and
$h_{~\Gamma}^{A}$ admit at least one derivative in any direction for the next
conditions hold.

\subsection{Integrability conditions}

At this point, the Geroch fields $C_{A}^{\Gamma a}$ have been defined but only
the $C_{A}^{\Gamma0}$ component has been used to define the constraints. We
introduce now a set of \textit{off-shell identities} which include all the
components of $C_{A}^{\Gamma a}$ and $M_{\Gamma}^{\tilde{\Delta}d}$, they are%

\begin{equation}
\nabla_{d}\left(  C_{A}^{\Gamma d}E^{A}\right)  =L_{1A}^{\Gamma}\left(
x,\phi,\nabla\phi\right)  E^{A}\left(  x,\phi,\nabla\phi\right)  ,
\label{eq_int_LE_1}%
\end{equation}%
\begin{equation}
\nabla_{d}\left(  M_{\Gamma}^{\tilde{\Delta}d}C_{A}^{\Gamma0}E^{A}\right)
=L_{2A}^{\tilde{\Delta}}\left(  x,\phi,\nabla\phi\right)  E^{A}\left(
x,\phi,\nabla\phi\right)  . \label{eq_int_LE_2}%
\end{equation}
We call \textit{integrability conditions }to these identities and we assume
they hold for the rest of this work. The motivation behind their inclusion
will be explained in the following paragraphs.

We begin by noticing, the integrability conditions specify that the
divergences of certain combinations of the equations of the system are
proportional to the system. These proportionality factors may depend on
$\left(  x,\phi,\nabla\phi\right)  $. \ In the case of the Einstein equations,
the eq. (\ref{eq_int_LE_1}) follows from the conservation law of the Einstein
tensor $\nabla_{a}G_{~b}^{a}=0$ (or from the 2nd Bianchi identity). This is
obtained by the Noether theorem as a conserved quantity from the Lagrangian
coordinate transformation invariance (see \cite{Wald:1984rg}). The eq.
(\ref{eq_int_LE_2}) appears when the Einstein equations are cast into
first-order form. In the case of the Maxwell equations (or Yang-Mills
equations), eq. (\ref{eq_int_LE_1}) follows from the charge conservation law,
this is obtained by the Noether theorem and associated to the Lagrangian gauge
invariance symmetry (see \cite{Brading:2002lzh}).This system does not have an
identity as (\ref{eq_int_LE_2}). In the case of the wave equations, the
identities (\ref{eq_int_LE_1}) and (\ref{eq_int_LE_2}) appear when the system
is reduced into first order in derivative adding extra constraints. The study
of the constraints propagation of the Maxwell and wave equations are discussed
in section \ref{Examples}, we there give explicit expressions for these identities.

It is likely that at least the identity (\ref{eq_int_LE_1}) always comes from
the Noether theorem (in there off-shell version) as a conserved quantity and
associated to some "gauge" symmetry of the system. As we will show in the next
section, these identities are needed to show the constraints conservation
since they are the evolutions equations of the constraints. Therefore, in
general, we beleive that the constraints conservation will be associated to
some symmetry of the system.

From the PDE point of view, it has sense that the Geroch fields appear in the
expressions (\ref{eq_int_LE_1}) and (\ref{eq_int_LE_2}). Considering only the
left-hand side of (\ref{eq_int_LE_1}) and replacing $E^{A}$ by its definition
(eq. (\ref{eq_sys_1})), we find that%

\begin{align}
&  \nabla_{d}\left(  C_{A}^{\Gamma d}E^{A}\right) \nonumber\\
&  =\nabla_{d}\left(  C_{A}^{\Gamma d}\left(  \mathfrak{N}_{~\alpha}%
^{Aa}\nabla_{a}\phi^{\alpha}-J^{A}\right)  \right)  ,\\
&  =\nabla_{d}\left(  C_{A}^{\Gamma d}\mathfrak{N}_{~\alpha}^{Aa}\right)
\nabla_{a}\phi^{\alpha}+\left(  C_{A}^{\Gamma\lbrack d}\mathfrak{N}_{~\alpha
}^{\left\vert A\right\vert a]}\right)  \nabla_{d}\nabla_{a}\phi^{\alpha
}-\nabla_{d}\left(  C_{A}^{\Gamma d}J^{A}\right)  ,\\
&  =\nabla_{d}\left(  C_{A}^{\Gamma d}\mathfrak{N}_{~\alpha}^{Aa}\right)
\nabla_{a}\phi^{\alpha}+\frac{1}{2}\left(  C_{A}^{\Gamma\lbrack d}%
\mathfrak{N}_{~\alpha}^{\left\vert A\right\vert a]}\right)  R_{~\beta
da}^{\alpha}\phi^{\beta}-\nabla_{d}\left(  C_{A}^{\Gamma d}J^{A}\right)  .
\label{eq_N_CE_1}%
\end{align}
Where we have used the chain rule, equation (\ref{eq_CK_1}) and the following
definition of the curvature tensor
\[
\nabla_{\lbrack d}\nabla_{a]}\phi^{\alpha}=\frac{1}{2}R_{~\beta da}^{\alpha
}\phi^{\beta}.
\]

\sloppy
We note that the resulting expression (\ref{eq_N_CE_1}) does not include
second derivatives of the fields $\phi^{\beta}$, so it is reasonable that it
can be factored as a product of expressions in first derivatives of
$\phi^{\beta}$, i.e. as  $L_{1A}^{\Gamma}\left(  x,\phi,\nabla\phi\right)
E^{A}\left(  x,\phi,\nabla\phi\right)  $. The case for (\ref{eq_int_LE_2}) is
analogous, this is
\begin{align*}
&  \nabla_{d}\left(  M_{\Gamma}^{\tilde{\Delta}d}C_{A}^{\Gamma0}E^{A}\right)
,\\
&  =\nabla_{d}\left(  M_{\Gamma}^{\tilde{\Delta}d}C_{A}^{\Gamma0}\left(
\mathfrak{N}_{~\alpha}^{Aa}\nabla_{a}\phi^{\alpha}-J^{A}\right)  \right)  ,\\
&  =\nabla_{d}\left(  M_{\Gamma}^{\tilde{\Delta}d}C_{A}^{\Gamma0}%
\mathfrak{N}_{~\alpha}^{Aa}\right)  \nabla_{a}\phi^{\alpha}+\frac{1}{2}\left(
M_{\Gamma}^{\tilde{\Delta}[d}C_{A}^{\left\vert \Gamma0\right\vert
}\mathfrak{N}_{~\alpha}^{\left\vert A\right\vert a]}\right)  R_{~\beta
da}^{\alpha}\phi^{\beta},\\
&  -\nabla_{d}\left(  M_{\Gamma}^{\tilde{\Delta}[d}C_{A}^{\left\vert
\Gamma0\right\vert }J^{A}\right)  ,\\
&  =\nabla_{i}\left(  M_{\Gamma}^{\tilde{\Delta}i}C_{A}^{\Gamma0}%
\mathfrak{N}_{~\alpha}^{Aj}\right)  \nabla_{j}\phi^{\alpha}+\frac{1}{2}\left(
M_{\Gamma}^{\tilde{\Delta}[i}C_{A}^{\left\vert \Gamma0\right\vert
}\mathfrak{N}_{~\alpha}^{\left\vert A\right\vert j]}\right)  R_{~\beta
ij}^{\alpha}\phi^{\beta},\\
&  -\nabla_{i}\left(  M_{\Gamma}^{\tilde{\Delta}[i}C_{A}^{\left\vert
\Gamma0\right\vert }J^{A}\right)  ,
\end{align*}
with $i,j=1,...,n$ and where the eq. $M_{\Gamma}^{\tilde{\Delta}0}=0$ and
$C_{A}^{\Gamma0}\mathfrak{N}_{~\alpha}^{A0}=0$ have been used. \ We obtain
again a set of first derivative equations for the fields $\phi^{\beta}$.

We shall give a final justification of why equations (\ref{eq_int_LE_1}) and
(\ref{eq_int_LE_2}) make sense in the next subsection.

\section{First main Theorem: Subsidiary System (SS) \label{Teorema_1}}

As we have shown the set of equations $E^{B}$ can be decomposed into evolution
$e^{\alpha}\left(  \phi\right)  =h_{B}^{\alpha}E^{B}=0$ and constraints
$\psi^{\Delta}\left(  \phi\right)  =C_{B}^{\Delta0}E^{B}=0$ equations. One
mechanism to find solutions for these equations is the free-evolution
approach. That is, we give an initial data $\left.  \phi\right\vert
_{\Sigma_{0}}=\phi_{0}$ satisfying the constraints $\left.  \psi^{\Gamma
}\left(  \phi_{0}\right)  \right\vert _{\Sigma_{0}}=0$, and use the evolution
equations $e^{\alpha}\left(  \phi\right)  =0$ to find the solutions of the
system over the future hypersurfaces $\Sigma_{t}$ with $t>0$. However, these
found solutions may not satisfy the constraints for $t>0$, and therefore may
not be solutions of the complete system. For this reason, we need equations
that tell us how these constraints evolve for $t>0$ when $e^{\alpha}\left(
\phi\right)  =0$. With these equations, we can determine if the constraint are
preserved or not. The standard mechanism to obtain these constraint evolution
equations is: take a time derivative of $\psi^{\Gamma}=n_{a}C_{A}^{\Gamma
a}\mathfrak{N}_{~\alpha}^{Ac}\eta_{c}^{b}\nabla_{b}\phi^{\alpha}-J^{A}$, which
translates into time derivatives of $\phi^{\alpha}$, plus other terms; use the
evolution equations $e^{\alpha}\left(  \phi\right)  =0$ to eliminate
$\partial_{t}\phi^{\alpha}$ from these expressions and (hopefully and after
many calculations) rearrange the resulting terms into expressions that only
include spatial derivatives of $\psi^{\Gamma}$ and lower order terms
proportional to $\psi^{\Gamma}$. As we shall see in the proof of the following
theorem, counting with the identities (\ref{eq_int_LE_1}) and
(\ref{eq_int_LE_2}) is equivalent to the mentioned process (since they include
the time derivatives of $\psi^{\Gamma}$). Actually, if these conditions are
not satisfied, the system has extra constraints that should be added to the
system and whose conservation has to be studied.

The following theorem, called \textit{Subsidiary System,} shows the explicit
form of the subsidiary system (SS) in the quasi-linear case.

\begin{theorem}
\label{teorema_ec_ev_vin_1}Considere the system (\ref{eq_sys_1}) such that the
following conditions hold:

(i) The principal symbol satisfies assumption \ref{cod_K0_rank_max_b}.

(ii) All constraints of the systems come from Geroch fields $C_{A}^{\Gamma a}$
as in eq. (\ref{eq_def_constr_1}) and these $C_{A}^{\Gamma a}$ satisfy
assumption \ref{cod_C0_rank_max_b}.

(iii) The system can admit (or not) extra Geroch fields $M_{A}^{\tilde{\Delta
}a}$ satisfying assumption \ref{As_M_1_b}.

(iv) The integrability conditions (\ref{eq_int_LE_1}) and (\ref{eq_int_LE_2})
are satisfied.

Then, the following off-shell identity is satisfied%
\begin{align}
&  \nabla_{0}\psi^{\Gamma}+\left(  C_{A}^{\Gamma i}h_{~\Delta}^{A}%
+N_{\tilde{\Delta}}^{\Gamma}M_{\Delta}^{\tilde{\Delta}i}\right)  \nabla
_{i}\psi^{\Delta}\nonumber\\
&  =\left(  L_{1A}^{\Gamma}\mathfrak{N}_{~\alpha}^{A0}-\nabla_{d}\left(
C_{A}^{\Gamma d}\mathfrak{N}_{~\alpha}^{A0}\right)  \right)  e^{\alpha}%
-C_{A}^{\Gamma d}\mathfrak{N}_{~\alpha}^{A0}\nabla_{d}e^{\alpha}\nonumber\\
&  +\left(  L_{1A}^{\Gamma}h_{~\Delta}^{A}-\nabla_{d}\left(  C_{A}^{\Gamma
d}h_{~\Delta}^{A}\right)  +N_{\tilde{\Delta}}^{\Gamma}\left(  L_{2A}%
^{\tilde{\Delta}}h_{~\Delta}^{A}-\nabla_{i}\left(  M_{\Delta}^{\tilde{\Delta
}i}\right)  \right)  \right)  \psi^{\Delta}
\label{Eq_ide_ev_off_shell_const_1_a}%
\end{align}
Here the tensors $N_{\tilde{\Delta}}^{\Gamma}$ can be freely chosen. If the
system does not admit Geroch fields $M_{A}^{\tilde{\Delta}a}$, then
$N_{\tilde{\Delta}}^{\Gamma}=0$, $L_{2A}^{\tilde{\Delta}}=0$. In this case,
the SS is unique.

In the on-shell case i.e. $e^{\alpha}\left(  \phi\right)  =0$, the above
identity reduces to the subsidiary system,%
\begin{align}
&  \nabla_{0}\psi^{\Gamma}+\left(  C_{A}^{\Gamma i}h_{~\Delta}^{A}%
+N_{\tilde{\Delta}}^{\Gamma}M_{\Delta}^{\tilde{\Delta}i}\right)  \nabla
_{i}\psi^{\Delta}\nonumber\\
&  =\left(  L_{1A}^{\Gamma}h_{~\Delta}^{A}-\nabla_{d}\left(  C_{A}^{\Gamma
d}h_{~\Delta}^{A}\right)  +N_{\tilde{\Delta}}^{\Gamma}\left(  L_{2A}%
^{\tilde{\Delta}}h_{~\Delta}^{A}-\nabla_{i}\left(  M_{\Delta}^{\tilde{\Delta
}i}\right)  \right)  \right)  \psi^{\Delta} \label{eq_ev_on_shell_const_1_b}%
\end{align}
where $i=1,...,n.$
\end{theorem}

We notice that if $\left\vert \psi^{\Gamma}\right.  _{\Sigma_{0}}=0,$ then
$\psi^{\Gamma}=0$ is a solution of equation (\ref{eq_ev_on_shell_const_1_b}).
If in addition, the equation (\ref{eq_ev_on_shell_const_1_b}) is well-posed,
the solution $\psi^{\Gamma}=0$ is unique and continue in the initial data.

It is important to emphasize that the set of constraint evolution equations is
not unique when the $M_{\Delta}^{\tilde{\Delta}i}$ fields exist. \ In other
words, constraints $\psi^{\Gamma}$ are solutions of a family of differential
equations (\ref{eq_ev_on_shell_const_1_b}). This family is parametrized by the
field $N_{\tilde{\Delta}}^{\Gamma}$ (which can be freely chosen), producing
the non-uniqueness. An example of this class of systems is the wave equation,
see section\ \ref{S_wave_equation}.

On the other hand, the well-posedness of the SS depends on the explicit form
of its principal symbol
\begin{equation}
B_{~\Delta}^{\Gamma i}:=C_{A}^{\Gamma i}h_{~\Delta}^{A}+N_{\tilde{\Delta}%
}^{\Gamma}M_{\Delta}^{\tilde{\Delta}i}. \label{Eq_simb_vinc_1}%
\end{equation}
Since, this symbol includes the field $h_{~\Delta}^{A}$, which is completely
determined by the choice of $h_{B}^{\alpha}$ (see eq. (\ref{eq_Kh_hC_1}%
))\footnote{Notice that the reverse is true too, $h_{B}^{\alpha}$ is
completely defined by choice of $h_{~\Delta}^{A}$.}, and the free field
$N_{\tilde{\Delta}}^{\Gamma}$, the well-posedness of the subsidiary systems is
given by $h_{B}^{\alpha}$ and $N_{\tilde{\Delta}}^{\Gamma}$. We focus on this
problem in section \ref{teorema_2}, where we provide a theorem on how to
apropriately chose $h_{B}^{\alpha}$ and $N_{\tilde{\Delta}}^{\Gamma}$ to
obtain a well-posed SS.

Finally, when there are no $M_{\Delta}^{\tilde{\Delta}i}$ fields present in
the system, the evolution equations of the constraints are unique, so the
analysis of well-posedness is highly simplified. The Maxwell equations are an
example of this case, see section \ref{S_Maxwell_eq}.

We will discuss the hyperbolicity of the SS in the next sections, focusing on
the constant-coefficient systems where we can give a closed answer.

\subsection{Invariance of the choice of $e^{\alpha}$ in the subsidiary system
\label{invariance_e}}

Given a particular physical theory, different reductions $h_{B}^{\alpha}$ and
$\tilde{h}_{B}^{\alpha}$ lead to different choices of the evolution equations
$e^{\alpha}$ and $\tilde{e}^{\alpha}$ (see subsection \ref{Evo}).\ Therefore,
to study the conservation of the constraints, the SS should be calculated for
each choice of $e^{\alpha}$ or $\tilde{e}^{\alpha}$. In addition, it might
happen that the SS does not exist\footnote{Meaning that it is not possible to
write a set of PDE's where the variables of the PDE's are the constraints.}
for a particular choice of the evolution equations. \ In our scheme, this
dilemma is solved in a quite elegant way. As we explained before, each
reduction $h_{B}^{\alpha}$ ($\tilde{h}_{B}^{\alpha}$) introduces a unique
field $h_{~\Delta}^{A}$ ($\tilde{h}_{~\Delta}^{A}$) given by the equation
(\ref{eq_Kh_hC_1}). \ This field $h_{~\Delta}^{A}$ ($\tilde{h}_{~\Delta}^{A}$)
appear in the equations of the SS (\ref{eq_ev_on_shell_const_1_b}), showing
that the SS always exists and how the information of the reductions
$h_{B}^{\alpha}$ ($\tilde{h}_{B}^{\alpha}$) is propagated to this equation.
The reason for this simple answer is that equation
(\ref{Eq_ide_ev_off_shell_const_1_a}) follows directly from the integrability
conditions (\ref{eq_int_LE_1}) and (\ref{eq_int_LE_2}), which do not depend on
any reduction $h_{B}^{\alpha}$. As it is shown in the proof of the theorem,
these reductions appear only as a trick to transform $C_{A}^{\Gamma a}E^{A}$
into $\left[
\begin{array}
[c]{cc}%
C_{A}^{\Gamma d}\mathfrak{N}_{~\alpha}^{A0} & C_{A}^{\Gamma d}h_{~\Delta}^{A}%
\end{array}
\right]  \left[
\begin{array}
[c]{c}%
e^{\alpha}\\
\psi^{\Delta}%
\end{array}
\right]  $, where the explicit form of $h_{B}^{\alpha}$ plays no role in this
proof, so that, eq. (\ref{Eq_ide_ev_off_shell_const_1_a}) can be obtained for
any $h_{B}^{\alpha}$.

It may happen that in some physical theories the integrability conditions
(\ref{eq_int_LE_1}) and (\ref{eq_int_LE_2}) depend on a specific reduction
$h_{B}^{\alpha}$ and consequently on its corresponding evolution equations
$e^{\alpha}$. In this class of systems, the eqs. (\ref{eq_int_LE_1}) and
(\ref{eq_int_LE_2}) could include 2-order or higher partial derivatives of the
evolution equations. \ This should not modify equation
(\ref{eq_ev_on_shell_const_1_b}) since $e^{\alpha}=0$ in the on-shell case,
but this system would be forced to evolve only with $e^{\alpha}=0$ since any
other choice of the evolution equations $\tilde{e}^{\alpha}$ would not produce
a SS as (\ref{eq_ev_on_shell_const_1_b}). Therefore, the constraint
conservation would not be guaranteed.

\subsection{Proof of Subsidiary System theorem \label{teorema_ec_ev_vin_1_a}}

The proof consists of showing how to go from the sum of (\ref{eq_int_LE_1})
and (\ref{eq_int_LE_2}),%

\begin{equation}
\nabla_{d}\left(  C_{A}^{\Gamma d}E^{A}\right)  +N_{\tilde{\Delta}}^{\Gamma
}\nabla_{d}\left(  M_{\Gamma}^{\tilde{\Delta}d}C_{A}^{\Gamma0}E^{A}\right)
=L_{1A}^{\Gamma}E^{A}+N_{\tilde{\Delta}}^{\Gamma}L_{2A}^{\tilde{\Delta}}E^{A},
\label{Eq_com_conds_int_1}%
\end{equation}
to (\ref{Eq_ide_ev_off_shell_const_1_a}), where $N_{\tilde{\Delta}}^{\Gamma}$
can be freely chosen.

Using (\ref{Eq_e_y_vin_Mat_E_1}) (valid from the assumptions
\ref{cod_K0_rank_max_b} and \ref{cod_C0_rank_max_b}) and (\ref{eq_Kh_hC_1})
the system $E^{A}$ can be written as
\begin{align}
E^{A}  &  =\left[
\begin{array}
[c]{cc}%
\mathfrak{N}_{~\alpha}^{A0} & h_{~\Delta}^{A}%
\end{array}
\right]  \left[
\begin{array}
[c]{c}%
h_{B}^{\alpha}\\
C_{B}^{\Delta0}%
\end{array}
\right]  E^{B},\nonumber\\
&  =\left[
\begin{array}
[c]{cc}%
\mathfrak{N}_{~\alpha}^{A0} & h_{~\Delta}^{A}%
\end{array}
\right]  \left[
\begin{array}
[c]{c}%
e^{\alpha}\\
\psi^{\Delta}%
\end{array}
\right]  . \label{Eq_E_e_phi_1}%
\end{align}
By replacing the latter equation in each of the terms of
(\ref{Eq_com_conds_int_1}), we obtain%
\begin{align}
&  \nabla_{d}\left(  C_{A}^{\Gamma d}E^{A}\right) \nonumber\\
&  =\nabla_{d}\left(  \left[
\begin{array}
[c]{cc}%
C_{A}^{\Gamma d}\mathfrak{N}_{~\alpha}^{A0} & C_{A}^{\Gamma d}h_{~\Delta}^{A}%
\end{array}
\right]  \left[
\begin{array}
[c]{c}%
e^{\alpha}\\
\psi^{\Delta}%
\end{array}
\right]  \right)  ,\nonumber\\
&  =\nabla_{d}\left(  C_{A}^{\Gamma d}\mathfrak{N}_{~\alpha}^{A0}\right)
e^{\alpha}+\nabla_{d}\left(  C_{A}^{\Gamma d}h_{~\Delta}^{A}\right)
\psi^{\Delta}+C_{A}^{\Gamma d}\mathfrak{N}_{~\alpha}^{A0}\nabla_{d}e^{\alpha
}+C_{A}^{\Gamma d}h_{~\Delta}^{A}\nabla_{d}\psi^{\Delta},\nonumber\\
&  =\nabla_{d}\left(  C_{A}^{\Gamma d}\mathfrak{N}_{~\alpha}^{A0}\right)
e^{\alpha}+\nabla_{d}\left(  C_{A}^{\Gamma d}h_{~\Delta}^{A}\right)
\psi^{\Delta}+C_{A}^{\Gamma d}\mathfrak{N}_{~\alpha}^{A0}\nabla_{d}e^{\alpha
}\label{Eq_nabla_C_E__e_phi_1}\\
&  +\nabla_{0}\psi^{\Gamma}+C_{A}^{\Gamma i}h_{~\Delta}^{A}\nabla_{d}%
\psi^{\Delta}.\nonumber
\end{align}
Where we used $C_{A}^{\Gamma0}h_{~\Delta}^{A}=\delta_{\Delta}^{\Gamma}$ and
$i=1,...,n$ in the last equation. On the other hand,%

\begin{align}
L_{1A}^{\Gamma}E^{A}  &  =L_{1A}^{\Gamma}\left[
\begin{array}
[c]{cc}%
\mathfrak{N}_{~\alpha}^{A0} & h_{~\Delta}^{A}%
\end{array}
\right]  \left[
\begin{array}
[c]{c}%
e^{\alpha}\\
\psi^{\Delta}%
\end{array}
\right]  ,\nonumber\\
&  =L_{1A}^{\Gamma}\mathfrak{N}_{~\alpha}^{A0}e^{\alpha}+L_{1A}^{\Gamma
}h_{~\Delta}^{A}\psi^{\Delta}. \label{Eq_L_e_phi_1}%
\end{align}
We conclude that (\ref{eq_int_LE_1}) can be written as
\begin{align*}
\nabla_{d}\left(  C_{A}^{\Gamma d}E^{A}\right)   &  =L_{1A}^{\Gamma}E^{A}\\
\nabla_{d}\left(  C_{A}^{\Gamma d}\mathfrak{N}_{~\alpha}^{A0}\right)
e^{\alpha}+\nabla_{d}\left(  C_{A}^{\Gamma d}h_{~\Delta}^{A}\right)
\psi^{\Delta}+  & \\
+C_{A}^{\Gamma d}\mathfrak{N}_{~\alpha}^{A0}\nabla_{d}e^{\alpha}+\nabla
_{0}\psi^{\Gamma}+C_{A}^{\Gamma i}h_{~\Delta}^{A}\nabla_{d}\psi^{\Delta}  &
=L_{1A}^{\Gamma}\mathfrak{N}_{~\alpha}^{A0}e^{\alpha}+L_{1A}^{\Gamma
}h_{~\Delta}^{A}\psi^{\Delta}.
\end{align*}
Considering now the divergences which involve $M_{\Gamma}^{\tilde{\Delta}d}$,
we obtain:
\begin{align}
\nabla_{d}\left(  M_{\Gamma}^{\tilde{\Delta}d}C_{A}^{\Gamma0}E^{A}\right)   &
=\nabla_{d}\left(  M_{\Gamma}^{\tilde{\Delta}d}\psi^{\Gamma}\right)
,\nonumber\\
&  =\nabla_{d}\left(  M_{\Gamma}^{\tilde{\Delta}d}\right)  \psi^{\Gamma
}+M_{\Gamma}^{\tilde{\Delta}d}\nabla_{d}\left(  \psi^{\Gamma}\right)
,\nonumber\\
&  =\nabla_{i}\left(  M_{\Gamma}^{\tilde{\Delta}i}\right)  \psi^{\Gamma
}+M_{\Gamma}^{\tilde{\Delta}i}\nabla_{i}\left(  \psi^{\Gamma}\right)  .
\label{Eq_vinc_vinc_1}%
\end{align}
Where in the first line we used the definition of $\psi^{\Gamma}=C_{A}%
^{\Gamma0}E^{A}$ and in the third line that $M_{\Gamma}^{\tilde{\Delta}0}$
vanishes (see Lemma \ref{Lema_M0}). The expression for $L_{2A}^{\tilde{\Delta
}}E^{A}$ is
\[
L_{2A}^{\tilde{\Delta}}E^{A}=L_{2A}^{\tilde{\Delta}}\mathfrak{N}_{~\alpha
}^{A0}e^{\alpha}+L_{2A}^{\tilde{\Delta}}h_{~\Delta}^{A}\psi^{\Delta},
\]
so the expression (\ref{eq_int_LE_2}) can be rewritten as
\begin{align*}
\nabla_{d}\left(  M_{\Gamma}^{\tilde{\Delta}d}C_{A}^{\Gamma0}E^{A}\right)   &
=L_{2A}^{\tilde{\Delta}}E^{A}\\
\nabla_{i}\left(  M_{\Gamma}^{\tilde{\Delta}i}\right)  \psi^{\Gamma}%
+M_{\Gamma}^{\tilde{\Delta}i}\nabla_{i}\left(  \psi^{\Gamma}\right)   &
=L_{2A}^{\tilde{\Delta}}\mathfrak{N}_{~\alpha}^{A0}e^{\alpha}+L_{2A}%
^{\tilde{\Delta}}h_{~\Delta}^{A}\psi^{\Delta}\text{.}%
\end{align*}
Recalling that this expression is an identity that holds for all $\phi
^{\alpha}$, it should not be possible that the right hand side of this
equation to contains time derivatives (on $e^{\alpha}$) while the left hand
side does not. This means that,
\[
L_{2A}^{\tilde{\Delta}}\mathfrak{N}_{~\alpha}^{A0}=0,
\]
hence
\begin{equation}
\nabla_{i}\left(  M_{\Gamma}^{\tilde{\Delta}i}\right)  \psi^{\Gamma}%
+M_{\Gamma}^{\tilde{\Delta}i}\nabla_{i}\left(  \psi^{\Gamma}\right)
=L_{2A}^{\tilde{\Delta}}h_{~\Delta}^{A}\psi^{\Delta}.
\label{Eq_vinc_de_los_vinc_1}%
\end{equation}
Finally, replacing these results in eq. (\ref{Eq_com_conds_int_1}) gives
(\ref{Eq_ide_ev_off_shell_const_1_a}), which concludes the proof.

As a final comment, in section \ref{Sec_Vin_M} we said that the fields
$M_{\Gamma}^{\tilde{\Delta}d}$ parameterize the constraints of the constraint.
The explicit expression for these differential relationships between the
constraints is given by the last equation (\ref{Eq_vinc_de_los_vinc_1}).

\section{Constant coefficient and Strong Hyperbolicity (SH)
\label{Sec_Const_coeff_1}}

In this section we introduce a brief summary of paper \cite{Abalos:2018uwg},
we present the definitions and the main result of that work about strong
hyperbolicity (SH) for systems with constraints. This main result gives the
necessary and sufficient conditions under which the system
\begin{equation}
E^{A}:=\mathfrak{N}_{~\alpha}^{Aa}\partial_{a}\phi^{\alpha},
\label{Eq_sis_1_coef_cte_1}%
\end{equation}
has a strong hyperbolic set of evolution equations.

As in \cite{Abalos:2018uwg}, we do not consider the quasi-linear case
(\ref{eq_sys_1}) and focus on the constant coefficient case, where
$\mathfrak{N}_{~\alpha}^{Aa}$ does not depend on $x,\phi$ (i.e. $\nabla
_{c}\mathfrak{N}_{~\alpha}^{Aa}=0$). This simplification leads to a closed
pseudo-differential theorem about SH for the evolution equations (theorem
\ref{Theor_FyL_2}) and will allow us, in the next section, to derive a theorem
about the strong hyperbolicity of the SS (theorem \ref{Theorem_coef_const_2}).
Note also that the covariant derivatives $\nabla_{a}$ have been changed to
partial derivatives $\partial_{a}$ and that the lower order terms have been
suppressed because the SH is not affected by them.

As before, the initial value problem\ consists on solving
\begin{equation}
E^{A}\left(  \phi\right)  =0\text{ with initial data }\left.  \phi^{\alpha
}\right\vert _{\Sigma_{0}}=f^{\alpha}\left(  x\right)  .
\label{Eq_ini_val_E_f_1}%
\end{equation}
To discuss its SH, we convert the problem to its Fourier space and present a
pseudo-differential analysis. Applying to (\ref{Eq_ini_val_E_f_1}) a Fourier
transformation on the spatial variables $x^{i}$, with $i=1,...,n$,\ we obtain
\begin{equation}
\tilde{E}^{A}:=\mathfrak{N}_{~\alpha}^{A0}\partial_{t}\tilde{\phi}^{\alpha
}+i\mathfrak{N}_{~\alpha}^{Ai}k_{i}\tilde{\phi}^{\alpha}=0,
\label{Eq_sis_1_pseudo_1}%
\end{equation}
with
\[
\left.  \tilde{\phi}^{\alpha}\right\vert _{\Sigma_{0}}=\tilde{f}^{\alpha
}\left(  x\right)  .
\]
Where $k_{a}$ is the \textit{wave vector} such that $k_{a}t^{a}=0$ (i.e.
$k_{0}=0$) and $k_{a}\eta_{b}^{a}=k_{b}$ (i.e. $k_{a}=\left(  0,k_{i}\right)
$). \ 

Now we introduce the reduction $h_{~A}^{\beta}\left(  k_{i}\right)  $
(satisfying eq. (\ref{eq_hK0_1})) which may depend on the wave vector $k_{a}$
or not. \ Applying $h_{~A}^{\beta}\left(  k_{i}\right)  $\ to the equation
(\ref{Eq_sis_1_pseudo_1}), we obtain a set of evolution equations for
$\tilde{\phi}^{\beta}$
\begin{equation}
\tilde{e}^{\beta}=\partial_{t}\tilde{\phi}^{\beta}+ih_{~A}^{\beta}\left(
k_{i}\right)  \mathfrak{N}_{~\alpha}^{Ai}k_{i}\tilde{\phi}^{\alpha}=0.
\label{Eq_evol_pseudo_1}%
\end{equation}
As before, the reduction aim is to combine the constraint and time derivative
equations, to produce systems of evolution equations for each field
$\tilde{\phi}^{\beta}$. The main difference is that now these evolution
equations are pseudo-differential and their solutions must be anti-transformed
to obtain solutions of the original system.

The set of evolution equations has to be well-posed to be predictive. This is
a property that depends on the choice of $h_{~A}^{\beta}$ since different
reductions may lead to ill-posed or well-posed systems. \ In particular, we
are considering a sub-class within the well-posed equations, namely the
strongly hyperbolic ones. This leads us to introduce the following definition.

\begin{definition}
\label{def_hyp_fuerte_1}Consider $n_{a}=\nabla_{a}t$ such that the assumption
\ref{cod_K0_rank_max_b} holds, we say that the system
(\ref{Eq_sis_1_coef_cte_1}) is strongly hyperbolic if there exists at least
one reduction $h_{~A}^{\beta}\left(  k\right)  $ satisfying (\ref{eq_hK0_1}),
such that for all $k_{i}$, with $\left\vert k\right\vert =1$, the principal
symbol of the evolution equations $A_{~\alpha}^{\beta i}k_{i}:=h_{~A}^{\beta
}\mathfrak{N}_{~\alpha}^{Ai}k_{i}$ is uniformly diagonalizable with real
eigenvalues. Namely, for all $k_{i}$ with $\left\vert k\right\vert =1,$ there
exists $\left(  T\left(  k\right)  \right)  _{~\gamma}^{\beta}$ such that
$h_{~A}^{\beta}\mathfrak{N}_{~\alpha}^{Ai}k_{i}=$ $\left(  T\left(  k\right)
\right)  _{~\gamma}^{\beta}\Lambda_{~\theta}^{\gamma}\left(  T^{-1}\left(
k\right)  \right)  _{~\alpha}^{\theta}$ with $\Lambda_{~\theta}^{\gamma}$
diagonal and real; and the diagonalization is uniform, which means that there
exists a constant $C>0$ such that
\begin{equation}
\left\vert T\left(  k\right)  \right\vert +\left\vert \left(  T\left(
k\right)  \right)  ^{-1}\right\vert <C. \label{Eq_T_T_C_1}%
\end{equation}

\end{definition}

The norms $\left\vert \cdot\right\vert $, used in $\left\vert k\right\vert =1$
and in the eq. (\ref{Eq_T_T_C_1}), can be any $k_{a}$ independent, positive
definite norms. We will assume for the following definitions and theorems that
the wave vector $k_{i}$ is \textit{normalized to} $\left\vert k\right\vert =1$.

When the reductions $h_{~A}^{\beta}$ satisfy the definition, we call them
\textit{hyperbolizations.} Note that when the system has no constraints, the
reduction $h_{~A}^{\beta}$ is unique and defined by (\ref{eq_hK0_1}).

In the literature, the above definition is presented as a theorem and the
original definition of strong hyperbolicity is another one, but since this
section is only an introduction to the topic, we condense the discussion and
present it as a definition. For more details about the theory we suggest
\cite{sarbach2012continuum} and the reference therein.

We introduce now a set of definitions to conclude with the definition of
canonical angles and after it, the theorem \ref{Theor_FyL_2}.

We call $\Phi$ and $\Psi$ to the vector fibers that include the vectors
$\delta\phi^{\alpha}$ and $X_{A}$ respectively. These spaces contain the right
and left kernel subspaces of the principal symbol $\mathfrak{N}_{~\alpha}%
^{Aa}w_{a}$ (for a given $w_{a}$), whose elements satisfy
\begin{align*}
\mathfrak{N}_{~\alpha}^{Aa}w_{a}\delta\phi^{\alpha}  &  =0,\\
X_{A}\mathfrak{N}_{~\alpha}^{Aa}w_{a}  &  =0.
\end{align*}
respectively. We will use this notation for any other operator, a vector will
belong to the right (left) kernel when it contracts with the down (up) index
operator and the result vanishes.

We introduce the set of planes $S_{n_{a}}^{%
\mathbb{C}
}=\{l_{a}\left(  \lambda\right)  =-\lambda n_{a}+k_{a}=\left(  -\lambda
,k_{i}\right)  .$ for all $k_{a}$ not proportional to $n_{a}$, with
$\left\vert k\right\vert =1$ and $\lambda\in%
\mathbb{C}
$ $\}$. They are complex planes for each fixed $k_{a}$ and they reduce to
lines when $\lambda\in%
\mathbb{R}
$. We call \ $S_{n_{a}}$ to the set of these lines.

We consider the left and right kernel of the principal symbol on these planes
$\mathfrak{N}_{~\alpha}^{Aa}l_{a}\left(  \lambda\right)  $ with \ $l_{a}%
\left(  \lambda\right)  \in S_{n_{a}}^{%
\mathbb{C}
}$. \ For each $k_{a}$, there exist certain values of $\lambda$ called
generalized eigenvalues $\lambda_{i}\left(  k\right)  $ with $i\in D_{\left(
k\right)  }$, $D_{\left(  k\right)  }:=\left\{  1,2,...,q_{\left(  k\right)
}\right\}  $ and $\lambda_{1}\left(  k\right)  <...<\lambda_{q_{\left(
k\right)  }}\left(  k\right)  $ such that $\mathfrak{N}_{~\alpha}^{Aa}%
l_{a}\left(  \lambda_{i}\left(  k\right)  \right)  $ has non-trivial right and
left kernel. We call $\Phi_{R}^{\lambda_{i}\left(  k\right)  }$ and $\Psi
_{L}^{\lambda_{i}\left(  k\right)  }$ to these subspaces of $\Phi$ and $\Psi$
respectively. Notice that the explicit form of these generalized eigenvalues
and the $q_{\left(  k\right)  }$ number of them is $k_{i}-$dependent. We call
$d_{\lambda_{i}\left(  k\right)  }$ to the geometric multiplicity of
$\lambda_{i}\left(  k\right)  $, thus $\dim\Phi_{R}^{\lambda_{i}\left(
k\right)  }=d_{\lambda_{i}\left(  k\right)  }$.

The left kernel behaves differently since its dimension is larger, $\dim
\Psi_{L}^{\lambda_{i}\left(  k\right)  }=d_{\lambda_{i}\left(  k\right)  }+c.$
This difference is due to the fact that for all $k_{i}$ and all $\lambda$, the
dimension of the left kernel of $\mathfrak{N}_{~\alpha}^{Aa}l_{a}\left(
\lambda\right)  $ is $c$ \footnote{Recall that $\mathfrak{N}_{~\alpha}%
^{Aa}l_{a}\left(  \lambda\right)  $ is an $e\times u$ matrix and $e=u+c$, this
means that $\mathfrak{N}_{~\alpha}^{Aa}l_{a}\left(  \lambda\right)  $ always
has a left kernel of dimension equal or greater than $c$.
\par
{}}, except when $\lambda=\lambda_{i}\left(  k\right)  $, where the left
kernel associated to the generalized eigenvalues increases the dimension to
$d_{\lambda_{i}\left(  k\right)  }+c$. A better explanation will be given in
the proof of theorem SH of the SS.

A necessary condition for the well-posedness of the system is that the
generalized eigenvalues be real. When the system satisfies this condition we
call it hyperbolic,

\begin{definition}
The system (\ref{Eq_sis_1_coef_cte_1}) is called hyperbolic, if there exists a
co-vector $n_{a}=\nabla_{a}t$ such that

(a) $\mathfrak{N}_{~\alpha}^{A0}:=\mathfrak{N}_{~\alpha}^{Aa}n_{a}$ has only
trivial right kernel.

(b) For each $l_{a}\left(  \lambda\right)  \in S_{n_{a}}^{%
\mathbb{C}
}$, if $\mathfrak{N}_{~\alpha}^{Aa}l_{a}\left(  \lambda\right)  $ has
non-trivial right kernel, then $\lambda\in%
\mathbb{R}
.$
\end{definition}

Notice that condition (a) is condition \ref{cod_K0_rank_max}.

Now, we introduce over $\Phi$ a positive definite Hermitian form
$G^{\alpha\beta}$. This allows us to define the vector subspace $\Phi
_{L}^{\lambda_{i}\left(  k\right)  }$, as the subspace obtained by projecting
$\Psi_{L}^{\lambda_{i}\left(  k\right)  }$ with $\mathfrak{N}_{~\alpha}%
^{A0}G^{\alpha\beta}$. Since $\Phi_{L}^{\lambda_{i}\left(  k\right)  }$ and
$\Phi_{R}^{\lambda_{i}\left(  k\right)  }$ are vector subspaces of $\Phi$ we
can introduce the canonical angles $\theta_{j}^{\lambda_{i}\left(  k\right)
}$ between them (see \ \cite{afriat1957orthogonal},
\cite{taslaman2014principal} for an introduction to the topic), where the
index $j$ runs from $1$ to $d_{\lambda_{i}\left(  k\right)  }$. These angles
measure the "separation distance" of these vector subspaces. As explained
below, the strong hyperbolicity theorem indicates that these distances should
be bounded for all $k_{i}.$

\begin{theorem}
\label{Theor_FyL_2} \cite{Abalos:2018uwg} \ The constant-coefficient system
(\ref{Eq_sis_1_coef_cte_1}) is strongly hyperbolic (admits at least one
hyperbolization) if and only if it is hyperbolic for some direction
$n_{a}=\nabla_{a}t$ and, for all $i\in D_{\left(  k\right)  }$ and all
normalized $k_{a}$ non-proportional to $n_{a}$, there is a constant maximum
angle $\vartheta<\frac{\pi}{2}$ between the canonical angles of $\Phi
_{L}^{\lambda_{i}\left(  k\right)  }$ and $\Phi_{R}^{\lambda_{i}\left(
k\right)  }$.

This last condition is equivalent to: if $\theta_{j}^{\lambda_{i}\left(
k\right)  }$ are the canonical angles between $\Phi_{L}^{\lambda_{i}\left(
k\right)  }$ and $\Phi_{R}^{\lambda_{i}\left(  k\right)  }$, then there exists
$\vartheta<\frac{\pi}{2}$ such that%
\begin{equation}
\cos\theta_{j}^{\lambda_{i}\left(  k\right)  }\geq\cos\vartheta>0
\label{Eq_cota_ang_can_1}%
\end{equation}
for all normalized $k_{a}$ non-proportional to $n_{a}$, with $i\in D_{\left(
k\right)  }$ and .$j=1,...,d_{\lambda_{i}\left(  k\right)  }$.
\end{theorem}

From the proof of the theorem, it follows how to build all the
hyperbolizations that the system admits, pseudo-differential or not (those
that do not depend on $k_{i}$). However, for simplicity, the hyperbolization
used in the proof is the one where the eigenvalues, adding by $h_{~A}^{\beta}$
to $A_{~\alpha}^{\beta i}k_{i}=h_{~A}^{\beta}\mathfrak{N}_{~\alpha}^{Ai}k_{i}%
$, are simple, different from each other and different from the generalized
eigenvalues for all $k_{i}$. \ When condition (\ref{Eq_cota_ang_can_1}) is
satisfied and this hyperbolization is chosen, it trivially guarantees the
definition \ref{def_hyp_fuerte_1} for $A_{~\alpha}^{\beta i}k_{i}$. We will
use this particular hyperbolization to prove our second main theorem.

\section{Second main theorem: Strong hyperbolicity of the SS \label{teorema_2}%
}

In this section, we continue considering the constant-coefficient case
(\ref{Eq_sis_1_coef_cte_1}), where $\mathfrak{N}_{~\alpha}^{Ab}$ is constant.
This means that $C_{A}^{\Gamma a}$ and $M_{\Delta}^{\tilde{\Delta}i}$ are
constants too since they are defined by the equations $C_{A}^{(a}%
\mathfrak{N}_{~\alpha}^{\left\vert A\right\vert b)}=0$ and $M_{A}%
^{\tilde{\Delta}(a}\mathfrak{N}_{~\alpha}^{\left\vert A\right\vert b)}=0$.
This simplification helps us to present the below closed theorem with simple hypotheses.

Theorem \ref{Theor_FyL_2} says nothing about the preservation or
non-preservation of the constraints during the evolution. For answering this
question, we assume valid all the hypotheses of section
\ref{Seccion_Setting_1} and show the sufficient conditions for the strong
hyperbolicity of the evolution equations of the constraints (eq.
(\ref{eq_ev_on_shell_const_1_b})). As already mentioned, these equations
(\ref{eq_ev_on_shell_const_1_b}) may not be unique since $N_{\tilde{\Delta}%
}^{\Gamma}$ can be freely chosen, i.e. $N_{\tilde{\Delta}}^{\Gamma}$ plays the
role of a reduction. So, we say that the SS is strongly hyperbolic when it is
possible to choose at least one reduction $N_{\tilde{\Delta}}^{\Gamma}$ (or
hyperbolization) such that the set of subsidiary equations
(\ref{eq_ev_on_shell_const_1_b}) is strongly hyperbolic. Following definition
\ref{def_hyp_fuerte_1} and considering eq. (\ref{eq_ev_on_shell_const_1_b}),
this is equivalent to requiring that there exists a reduction $N_{\tilde
{\Delta}}^{\Gamma}\left(  k\right)  $ such that the principal symbol of the
subsidiary system
\[
B_{~\Delta}^{\Gamma i}k_{i}:=C_{A}^{\Gamma i}h_{~\Delta}^{A}k_{i}%
+N_{\tilde{\Delta}}^{\Gamma}M_{\Delta}^{\tilde{\Delta}i}k_{i}%
\]
is uniformly diagonalizable with real eigenvalues. The following second main
theorem, called \textit{strong hyperbolicity of the subsidiary system},
explains under which conditions exists such hyperbolization.

\begin{theorem}
\label{Theorem_coef_const_2}Consider the system of constant coefficients
(\ref{Eq_sis_1_coef_cte_1}). This system admits a hyperbolization
$h_{~A}^{\beta}\left(  k\right)  $ and has\ at least one strongly hyperbolic
subsidiary system associated to the (strongly hyperbolic) evolution equations
$h_{~A}^{\beta}\left(  k\right)  \tilde{E}^{A}=0$ if the following conditions hold:

i) The system satisfies the hypotheses of Theorem \ref{Theor_FyL_2}.

ii) All the constraints of the systems come from Geroch fields $C_{A}^{\Gamma
a}$ as in the eq. (\ref{eq_def_constr_1}) and these $C_{A}^{\Gamma a}$ satisfy
the assumption \ref{cod_C0_rank_max_b}.

iii) The system can admit (or not) extra Geroch fields $M_{A}^{\tilde{\Delta
}a}$ which satisfy assumption \ref{As_M_1_b}.

iv) The integrability conditions (\ref{eq_int_LE_1}) and (\ref{eq_int_LE_2}).
are satisfied.

v) For each $k_{i}$ with $\left\vert k\right\vert =1$, the fields
$M_{A}^{\tilde{\Delta}i}k_{i}$ span the left kernel of $C_{A}^{\Gamma
0}\mathfrak{N}_{~\alpha}^{Aj}k_{j}.$
\end{theorem}

Notice that ii), iii) and iv) are the same conditions as in the subsidiary
system theorem, and that we have added one extra condition, the number v).
This has to be included to guarantee the existence of the hyperbolization
$N_{\tilde{\Delta}}^{\Gamma}\left(  k\right)  $, as we will show in the proof
of the theorem. However, this is not a condition that physical systems
necessarily satisfy. For example, there could exist $X_{\Gamma}\left(
k\right)  $ defined only for a particular direction of $k_{i}$ and such that
$X_{\Gamma}\left(  k\right)  C_{A}^{\Gamma0}\mathfrak{N}_{~\alpha}^{Aj}%
k_{j}=0$, so $X_{\Gamma}\left(  k\right)  $ could not be expanded by the
$M_{A}^{\tilde{\Delta}i}k_{i}$ since the latter and any linear combination of
them are defined for all $k_{i}$. We will explain in subsection \ref{teo_3},
after the proof of the theorem, how to deal with this class of cases. These
results will only be valid in the pseudo-differential version, so we may not
be able to extrapolate them so directly to the non-pseudo-differential case.

The proof of this theorem is given in subsection
\ref{Proof_Theorem_coef_const_2}. This proof uses the fact that the
diagonalization bases of the principal symbol of the evolution equations are
related in a particular way to the diagonalization bases of the principal
symbol of the subsidiary system. This relationship was found by Reula
\cite{reula2004strongly}, assuming that there are subsidiary equations for the
constraints and that they are first-order in derivatives. The theorem
presented here completes these ideas since Reula's assumption is obtained as a
result when conditions (ii), (iii), (iv) are satisfied. On the other hand,
these conditions are the hypotheses of the SS theorem for the quasi-linear
case. Therefore, this relationship between the bases is also valid in the
quasi-linear case (\ref{eq_sys_1}), when the reductions $h_{~A}^{\beta}$ and
$N_{\tilde{\Delta}}^{\Gamma}$ cannot depend on the wave vector $k_{a}$. The
problem in these non-pseudo differential cases is that the propagation
velocities of the subsidiary system cannot be freely chosen. This complicates
a possible proof of a general theorem. Nevertheless, the steps of the proof
presented here can be adapted to each particular theory to conclude similar
results. Following all these ideas and to avoid a more complex discussion, we
present here a closed theorem with simple hypotheses for the
constant-coefficient case.

\subsection{Proof of theorem SH of the SS \label{Proof_Theorem_coef_const_2}}

We show the theorem assuming that the system admits non-trivial $M_{\Delta
}^{\tilde{\Delta}i}k_{i}$. The proof for the case where the system does not
admit $M_{\Delta}^{\tilde{\Delta}i}k_{i}$ will be trivial from the previous
case. We comment on this at the end.

By (i) and theorem \ref{Theor_FyL_2}, we know that there exists a family of
hyperbolizations $h_{B}^{\alpha}\left(  k\right)  $ of the system
(\ref{Eq_sis_1_coef_cte_1}). From this family, a particular hyperbolization
was used in \cite{Abalos:2018uwg} to prove theorem \ref{Theor_FyL_2}, we call
it \textit{hyperbolization 1}. We shall comment and use about it in the following.

The idea of the proof of our theorem is to show that if we choose the
\textit{hyperbolization 1} $h_{B}^{\alpha}\left(  k\right)  $, then there
exists $N_{\tilde{\Delta}}^{\Gamma}\left(  k\right)  $ such that the principal
symbol of the subsidiary system $B_{~\Delta}^{\Gamma i}k_{i}:=C_{A}^{\Gamma
i}h_{~\Delta}^{A}k_{i}+N_{\tilde{\Delta}}^{\Gamma}M_{\Delta}^{\tilde{\Delta}%
i}k_{i}$ satisfies the definition \ref{def_hyp_fuerte_1} and therefore the
subsidiary equations are strongly hyperbolic. Recall that $h_{B}^{\alpha
}\left(  k\right)  $ defines the evolution equations $\tilde{e}^{\alpha
}\left(  \tilde{\phi}\right)  =h_{B}^{\alpha}\left(  k\right)  \tilde{E}%
^{B}=0$ whose principal symbol is $\ A_{~\alpha}^{\beta i}k_{i}=h_{~A}^{\beta
}\mathfrak{N}_{~\alpha}^{Ai}k_{i}$.

The characteristic structure of $\ A_{~\alpha}^{\beta i}k_{i}$: the
eigenvalues of $A_{~\alpha}^{\beta i}k_{i}$ define the propagation velocities
of the system \ (see \cite{sarbach2012continuum}). These eigenvalues are
divided into two groups. Those we call the "\textit{physics}", which are
associated to the evolution of the physically relevant fields and do not
change no matter the chosen reduction $h_{~A}^{\beta}$ ; and the rest, which
we call the "\textit{constraints 1}", that depend on the chosen reduction. As
shown in \cite{Abalos:2018uwg}, the "constraints 1" can be freely chosen using
specific reductions. Particularly, by choosing the "hyperbolization 1", the
matrix $A_{~\alpha}^{\beta i}k_{i}$ becomes uniformly diagonalizable with real
eigenvalues. This hyperbolization satisfies that for each $k_{i}$ the
"constraints 1" are simple, non-degenerate, different from each other and
different from the "physics".

The characteristic structure of $\ B_{~\Delta}^{\Gamma i}k_{i}$: we will show
that the eigenvalues of $B_{~\Delta}^{\Gamma i}k_{i}$ are also divided into
two groups. The "constraints 1" (inherited from $A_{~\alpha}^{\beta i}k_{i}$),
which remain unchanged by any choice of $N_{\tilde{\Delta}}^{\Gamma}$; and the
other group which we call the "constraints 2", that depend on $N_{\tilde
{\Delta}}^{\Gamma}$.

Following these ideas, we will show, for each normalized $k_{i}$, that:

a) The set of "constrains 1" are all the generalized eigenvalues of the pencil%
\begin{equation}
\left[
\begin{array}
[c]{c}%
-\delta_{\Delta}^{\Gamma}\lambda+C_{A}^{\Gamma j}h_{~\Delta}^{A}k_{j}\\
M_{\Delta}^{\tilde{\Delta}j}k_{j}%
\end{array}
\right]  \label{Eq_pencil_cons_1}%
\end{equation}

b) In the Kronecker decomposition (see \cite{gantmacher1992theory} and
\cite{gantmakher1998theory} for its definition) of this pencil, all its Jordan
blocks are of dimension 1 for all $k_{i}$. \ In other words, for each $k_{i}$,
these generalized eigenvalues are simple, non-degenerated and different from
each other. \ 

This implies that (\ref{Eq_pencil_cons_1}) satisfies the condition
(\ref{Eq_cota_ang_can_1}) for canonical angles. So we can use theorem
\ref{Theor_FyL_2} applied to the following pseudo-differential equations
\begin{equation}
\left[
\begin{array}
[c]{c}%
\delta_{\Delta}^{\Gamma}\partial_{t}+iC_{A}^{\Gamma j}h_{~\Delta}^{A}k_{j}\\
M_{\Delta}^{\tilde{\Delta}j}k_{j}%
\end{array}
\right]  \psi^{\Delta}=0. \label{Eq_pencil_ cons_2}%
\end{equation}
\ Thus, using the thesis of this theorem, we conclude that there exists (a
hyperbolization) $N_{\tilde{\Delta}}^{\Gamma}\left(  k\right)  $ such that
$B_{~\Delta}^{\Gamma i}k_{i}=C_{A}^{\Gamma i}h_{~\Delta}^{A}k_{i}%
+N_{\tilde{\Delta}}^{\Gamma}M_{\Delta}^{\tilde{\Delta}i}k_{i}$ satisfies the
definition \ref{def_hyp_fuerte_1}. \ Here, $N_{\tilde{\Delta}}^{\Gamma}\left(
k\right)  $ is chosen as the "hyperbolization 1" for (\ref{Eq_pencil_ cons_2})
and such that the eigenvalues "constraints 2" are simple, non-degenerate,
different from each other and different from "constraints 1".

This discussion says that we should show (a) and (b) from (i), (ii), (iii),
(iv) and (v) to complete the proof of our theorem. \ Although, we should also
justify the use of the pseudo-differential reductions $h_{B}^{\alpha}\left(
k\right)  $, $h_{~\Delta}^{A}\left(  k\right)  $ and $N_{\tilde{\Delta}%
}^{\Gamma}\left(  k\right)  $.

\subsubsection{Subsidiary system in Fourier's form\label{Subsi_Fourier}}

We first reproduce the calculations of subsection \ref{teorema_ec_ev_vin_1_a}
until we arrive at the evolution equations for the constraints in their
Fourier version. We have to repeat these steps since the quasi-linear system
includes non-linearities that complicate the Fourier analysis and these
systems do not admit pseudo-differential reductions. The use of these
reductions is only valid in the constant coefficients case and after applying
the Fourier transform to the system.

Recalling that $C_{A}^{\Gamma a}$ and $M_{\Gamma}^{\tilde{\Delta}a}$ are
constants, it is easy to prove the identity
\[
\left(  n_{a}\partial_{t}+ik_{a}\right)  \left[
\begin{array}
[c]{c}%
C_{A}^{\Gamma a}\\
M_{\Delta}^{\tilde{\Delta}a}%
\end{array}
\right]  \mathfrak{N}_{~\alpha}^{Ab}\left(  n_{b}\partial_{t}+ik_{b}\right)
=0.
\]
Multiplying this expression by $\tilde{\phi}^{\alpha}$, we obtain
\begin{equation}
\left[
\begin{array}
[c]{c}%
C_{A}^{\Gamma0}\partial_{t}+iC_{A}^{\Gamma j}k_{j}\\
iM_{\Delta}^{\tilde{\Delta}j}k_{j}C_{A}^{\Delta0}%
\end{array}
\right]  \tilde{E}^{A}=0, \label{eq_id_1}%
\end{equation}
where the expressions $C_{A}^{\Gamma0}=C_{A}^{\Gamma a}n_{a}$, $k_{a}=\left(
0,k_{i}\right)  $ and $M_{\Delta}^{\tilde{\Delta}0}=0$ have been used.

We introduce the identity $\delta_{B}^{A}=\left[
\begin{array}
[c]{cc}%
\mathfrak{N}_{~\alpha}^{A0} & h_{~\Delta}^{A}\left(  k\right)
\end{array}
\right]  \left[
\begin{array}
[c]{c}%
h_{B}^{\alpha}\left(  k\right) \\
C_{B}^{\Delta0}%
\end{array}
\right]  $ in the above equation and separate the terms as follow:%
\begin{align}
&  \left[
\begin{array}
[c]{c}%
C_{A}^{\Gamma0}\partial_{t}+iC_{A}^{\Gamma j}k_{j}\\
iM_{A}^{\tilde{\Delta}j}k_{j}%
\end{array}
\right]  \left[
\begin{array}
[c]{cc}%
\mathfrak{N}_{~\alpha}^{A0} & h_{~\Delta}^{A}\left(  k\right)
\end{array}
\right] \nonumber\\
&  =\left[
\begin{array}
[c]{cc}%
-iC_{A}^{\Gamma0}\mathfrak{N}_{~\alpha}^{Aj}k_{j} & \delta_{\Delta}^{\Gamma
}\partial_{t}+iC_{A}^{\Gamma j}k_{j}h_{~\Delta}^{A}\\
0 & iM_{\Delta}^{\tilde{\Delta}j}k_{j}%
\end{array}
\right]  , \label{eq_id_3}%
\end{align}
where we used that $M_{\Delta}^{\tilde{\Delta}0}=0$, $C_{B}^{\Gamma
0}h_{~\Delta}^{A}=\delta_{\Delta}^{\Gamma}$, $-C_{A}^{\Gamma0}\mathfrak{N}%
_{~\alpha}^{Aj}=C_{A}^{\Gamma i}\mathfrak{N}_{~\alpha}^{A0}$. And%

\begin{align}
\left[
\begin{array}
[c]{c}%
\tilde{e}^{\beta}\\
\tilde{\psi}^{\Delta}%
\end{array}
\right]   &  :=\left[
\begin{array}
[c]{c}%
h_{B}^{\alpha}\left(  k\right) \\
C_{B}^{\Delta0}%
\end{array}
\right]  \tilde{E}^{B}\label{eq_id_4}\\
&  =\left[
\begin{array}
[c]{c}%
\partial_{t}\tilde{\phi}^{\beta}+ih_{~A}^{\beta}\left(  k_{i}\right)
\mathfrak{N}_{~\alpha}^{Ai}k_{i}\tilde{\phi}^{\alpha}\\
iC_{B}^{\Delta0}\mathfrak{N}_{~\alpha}^{Ai}k_{i}\tilde{\phi}^{\alpha}%
\end{array}
\right]  , \label{eq_id_5}%
\end{align}
where we used that $h_{~A}^{\beta}\mathfrak{N}_{~\alpha}^{A0}=\delta_{\alpha
}^{\beta}$ and $C_{B}^{\Delta0}\mathfrak{N}_{~\alpha}^{A0}=0$. So, by
contracting eq. (\ref{eq_id_3}) with (\ref{eq_id_4}), the identity
(\ref{eq_id_1}) is rewritten as
\begin{equation}
\left[
\begin{array}
[c]{cc}%
-iC_{A}^{\Gamma0}\mathfrak{N}_{~\alpha}^{Aj}k_{j} & \delta_{\Delta}^{\Gamma
}\partial_{t}+iC_{A}^{\Gamma j}k_{j}h_{~\Delta}^{A}\\
0 & iM_{\Delta}^{\tilde{\Delta}j}k_{j}%
\end{array}
\right]  \left[
\begin{array}
[c]{c}%
\tilde{e}^{\beta}\\
\tilde{\psi}^{\Delta}%
\end{array}
\right]  =0. \label{eq_id_6}%
\end{equation}

In the on-shell case, when the evolution equations are satisfied $\tilde
{e}^{\beta}=0,$ we obtain the system (\ref{Eq_pencil_ cons_2}), that is,
\[
\left[
\begin{array}
[c]{c}%
\delta_{\Delta}^{\Gamma}\partial_{t}+iC_{A}^{\Gamma j}k_{j}h_{~\Delta}^{A}\\
iM_{\Delta}^{\tilde{\Delta}j}k_{j}%
\end{array}
\right]  \tilde{\psi}^{\Delta}=0.
\]
The constraints $\tilde{\psi}^{\Delta}$ satisfy all these pseudo-differential
equations, then the final expression for the subsidiary system \ (on-shell
case $\tilde{e}^{\alpha}\left(  \tilde{\phi}\right)  =0$) is
\begin{equation}
\partial_{0}\psi^{\Gamma}+\left(  C_{A}^{\Gamma i}h_{~\Delta}^{A}%
+N_{\tilde{\Delta}}^{\Gamma}M_{\Delta}^{\tilde{\Delta}i}\right)  k_{i}%
\psi^{\Delta}=0. \label{Eq_sist_B_1}%
\end{equation}
Notice that the principal symbol of the subsidiary system is the same as in
the quasi-linear case (see eq. (\ref{Eq_simb_vinc_1})), the difference here is
that $h_{~\Delta}^{A}\left(  k\right)  $ and $N_{\tilde{\Delta}}^{\Gamma
}\left(  k\right)  $ can depend on the wave vector.

We also note that by contracting the eq. (\ref{eq_id_3}) with (\ref{eq_id_5})
the following identity is obtained%
\begin{equation}
\left[
\begin{array}
[c]{cc}%
-iC_{A}^{\Gamma0}\mathfrak{N}_{~\alpha}^{Aj}k_{j} & \delta_{\Delta}^{\Gamma
}\partial_{t}+iC_{A}^{\Gamma j}k_{j}h_{~\Delta}^{A}\\
0 & iM_{\Delta}^{\tilde{\Delta}j}k_{j}%
\end{array}
\right]  \left[
\begin{array}
[c]{c}%
\partial_{t}\tilde{\phi}^{\beta}+ih_{~A}^{\beta}\left(  k_{i}\right)
\mathfrak{N}_{~\alpha}^{Ai}k_{i}\tilde{\phi}^{\alpha}\\
iC_{B}^{\Delta0}\mathfrak{N}_{~\alpha}^{Ai}k_{i}\tilde{\phi}^{\alpha}%
\end{array}
\right]  =0. \label{Eq_left_ker_K_1}%
\end{equation}
By condition v), we are assuming that $M_{\Delta}^{\tilde{\Delta}j}k_{j}$
expands the entire left kernel of $C_{B}^{\Delta0}\mathfrak{N}_{~\alpha}%
^{Ai}k_{i}$, otherwise, the above system would be modified as follows
\begin{equation}
\left[
\begin{array}
[c]{cc}%
-iC_{A}^{\Gamma0}\mathfrak{N}_{~\alpha}^{Aj}k_{j} & \delta_{\Delta}^{\Gamma
}\partial_{t}+iC_{A}^{\Gamma j}k_{j}h_{~\Delta}^{A}\\
0 & iM_{\Delta}^{\tilde{\Delta}j}k_{j}\\
0 & iX_{\Delta}^{\check{s}}\left(  k\right)
\end{array}
\right]  \left[
\begin{array}
[c]{c}%
\tilde{e}^{\beta}\\
\tilde{\psi}^{\Delta}%
\end{array}
\right]  =0. \label{Eq_sin_v_1}%
\end{equation}
Where the vectors $X_{\Delta}^{\check{s}}\left(  k\right)  $ satisfy
$X_{\Delta}^{\check{s}}\left(  k\right)  C_{B}^{\Delta0}\mathfrak{N}_{~\alpha
}^{Ai}k_{i}=0$ for some $k_{i}$ directions and cannot be obtained from linear
combinations of $M_{\Delta}^{\tilde{\Delta}j}k_{j}$. The index $\check{s}$
numbers these vectors such that, for each $k_{i}$, $span\left\langle
X_{\Delta}^{\check{s}},M_{\Delta}^{\tilde{\Delta}j}k_{j}\right\rangle
=left\_\ker\left(  C_{B}^{\Delta0}\mathfrak{N}_{~\alpha}^{Ai}k_{i}\right)  $.
We will return to this discussion after the proof of this theorem and explain
how to obtain the strong hyperbolicity of the subsidiary system by suppressing
condition v).

\subsubsection{Relationship between the principal symbols}

As a second step, we study the relationship between the principal symbols of
the evolution equations and the subsidiary system. Every expression found in
this subsubsection is also valid in the quasi-linear case.

For each normalized $k_{a}=\left(  0,k_{i}\right)  $, consider the lines
$l_{a}\left(  \lambda\right)  =-n_{a}\lambda+k_{a}\in S_{n_{a}}$. Using the
equations (\ref{Eq_C_K_nueva_1}) and (\ref{eq_M_K_1}) we can conclude the
following identities
\begin{align}
0  &  =l_{a}\left(  \lambda\right)  M_{\Gamma}^{\tilde{\Delta}a}C_{A}%
^{\Gamma0}\mathfrak{N}_{~\alpha}^{Ab}l_{b}\left(  \lambda\right)
,\label{Eq_M_C0_K_1}\\
&  =k_{i}M_{\Gamma}^{\tilde{\Delta}i}\left(  C_{A}^{\Gamma0}\mathfrak{N}%
_{~\alpha}^{Aj}k_{j}\right)  , \label{Eq_M_C0_K_2}%
\end{align}
and%
\begin{align*}
0  &  =l_{a}\left(  \lambda\right)  C_{A}^{\Gamma a}\left[
\begin{array}
[c]{cc}%
\mathfrak{N}_{~\alpha}^{A0} & h_{~\Delta}^{A}%
\end{array}
\right]  \left[
\begin{array}
[c]{c}%
h_{B}^{\alpha}\\
C_{B}^{\Delta0}%
\end{array}
\right]  \mathfrak{N}_{~\beta}^{Bb}l_{b}\left(  \lambda\right)  ,\\
&  =-\left(  l_{a}C_{A}^{\Gamma0}\mathfrak{N}_{~\alpha}^{Aa}\right)  \left(
h_{B}^{\alpha}\mathfrak{N}_{~\beta}^{Bb}l_{b}\right)  +\left(  l_{a}%
C_{A}^{\Gamma a}h_{~\Delta}^{A}\right)  \left(  C_{B}^{\Delta0}\mathfrak{N}%
_{~\beta}^{Bb}l_{b}\right)  ,\\
&  =-\left(  k_{j}C_{A}^{\Gamma0}\mathfrak{N}_{~\alpha}^{Aj}\right)  \left(
-\lambda\delta_{\beta}^{\alpha}+h_{~A}^{\alpha}\mathfrak{N}_{~\beta}^{Ai}%
k_{i}\right)  +\left(  -\lambda\delta_{\Delta}^{\Gamma}+C_{A}^{\Gamma
j}h_{~\Delta}^{A}k_{j}\right)  \left(  C_{B}^{\Delta0}\mathfrak{N}_{~\beta
}^{Bj}k_{j}\right)  .
\end{align*}

Notice that these equations are exactly the rows of (\ref{Eq_left_ker_K_1}) by
replacing $\partial_{t}$ by $-i\lambda$ and dividing by $i$.\ We conclude the
following equation
\begin{equation}
\left(  k_{j}C_{A}^{\Gamma0}\mathfrak{N}_{~\alpha}^{Aj}\right)  \left(
-\lambda\delta_{\beta}^{\alpha}+h_{~A}^{\alpha}\mathfrak{N}_{~\beta}^{Ai}%
k_{i}\right)  =\left(  -\lambda\delta_{\Delta}^{\Gamma}+C_{A}^{\Gamma
j}h_{~\Delta}^{A}k_{j}\right)  \left(  C_{B}^{\Delta0}\mathfrak{N}_{~\beta
}^{Bj}k_{j}\right)  . \label{Eq_simbs_1}%
\end{equation}

This equation was found in \cite{reula2004strongly}, assuming that the
subsidiary system was first order in derivatives. The latter equation is not
necessarily valid unless the system has the structure associated with the
Geroch fields presented here, i.e. hypotheses ii), iii) and iv) of the theorem.

From the eq. (\ref{Eq_M_C0_K_2}) we know that $k_{i}M_{\Gamma}^{\tilde{\Delta
}i}$ belongs to the left-hand kernel of $C_{A}^{\Gamma0}\mathfrak{N}_{~\alpha
}^{Aj}k_{j}$. Therefore, considering%
\[
B_{~\Delta}^{\Gamma j}k_{j}:=\left(  C_{A}^{\Gamma j}h_{~\Delta}^{A}%
+N_{\tilde{\Delta}}^{\Gamma}M_{\Delta}^{\tilde{\Delta}j}\right)  k_{j}%
\]
with free $N_{\tilde{\Delta}}^{\Gamma}$, the above equation can be rewritten
as
\begin{equation}
\left(  k_{j}C_{A}^{\Gamma0}\mathfrak{N}_{~\alpha}^{Aj}\right)  \left(
-\lambda\delta_{\beta}^{\alpha}+A_{~\beta}^{\alpha i}k_{i}\right)  =\left(
-\delta_{\Delta}^{\Gamma}\lambda+B_{~\Delta}^{\Gamma j}k_{j}\right)  \left(
C_{B}^{\Delta0}\mathfrak{N}_{~\beta}^{Bi}k_{i}\right)  .
\label{Eq_rel_simbolos_1}%
\end{equation}

This last expression shows how the principal symbols of the evolution
equations for $\phi^{\alpha}$ are related to the principal symbol of the
subsidiary system.

\subsubsection{Left Kernel of $\mathfrak{N}_{~\alpha}^{Ab}l_{b}\left(
\lambda\right)  $}

In this subsubsection, we shall choose a basis for the left kernel of
$\mathfrak{N}_{~\alpha}^{Ab}l_{b}\left(  \lambda\right)  $. This basis will
allow us to find the Kronecker structure of the pencil $\mathfrak{N}_{~\alpha
}^{Ab}l_{b}\left(  \lambda\right)  $ in the following subsubsection.

From the previous subsubsections, we know that
\[
0=\left(  -\lambda n_{a}+k_{a}\right)  \left[
\begin{array}
[c]{c}%
C_{A}^{\Gamma a}\\
M_{\Delta}^{\tilde{\Delta}a}%
\end{array}
\right]  \mathfrak{N}_{~\alpha}^{Ab}\left(  -\lambda n_{b}+k_{b}\right)  .
\]
As before, using equation (\ref{eq_Kh_hC_1}) this expression can be rewritten
as
\[
0=\left[
\begin{array}
[c]{cc}%
-C_{A}^{\Gamma0}\mathfrak{N}_{~\alpha}^{Aj}k_{j} & -\lambda\delta_{\Delta
}^{\Gamma}+\left(  C_{A}^{\Gamma j}h_{~\Delta}^{A}\right)  k_{j}\\
0 & M_{\Delta}^{\tilde{\Delta}j}k_{j}%
\end{array}
\right]  \left[
\begin{array}
[c]{c}%
-\lambda\delta_{\alpha}^{\beta}+h_{~A}^{\beta}\mathfrak{N}_{~\alpha}^{Ai}%
k_{i}\\
C_{B}^{\Delta0}\mathfrak{N}_{~\alpha}^{Ai}k_{i}%
\end{array}
\right]  .
\]
Therefore, the basis we are looking for will result from the choice of a
subset of vectors of
\begin{equation}
\left[
\begin{array}
[c]{cc}%
-C_{A}^{\Gamma0}\mathfrak{N}_{~\alpha}^{Aj}k_{j} & -\lambda\delta_{\Delta
}^{\Gamma}+\left(  C_{A}^{\Gamma j}h_{~\Delta}^{A}\right)  k_{j}\\
0 & M_{\Delta}^{\tilde{\Delta}j}k_{j}%
\end{array}
\right]  . \label{left_ker_1}%
\end{equation}

Notice that, for each $k_{j}$, the pencil
\begin{align}
\lambda I_{\alpha}^{A}+K_{\alpha}^{A}  &  :=\left[
\begin{array}
[c]{c}%
h_{B}^{\beta}\\
C_{B}^{\Delta0}%
\end{array}
\right]  \mathfrak{N}_{~\alpha}^{Bb}l_{b}\left(  \lambda\right)  =\left[
\begin{array}
[c]{c}%
-\lambda\delta_{\alpha}^{\beta}+h_{~A}^{\beta}\mathfrak{N}_{~\alpha}^{Ai}%
k_{i}\\
C_{B}^{\Delta0}\mathfrak{N}_{~\alpha}^{Ai}k_{i}%
\end{array}
\right] \label{eq_pencil_h_C_K_1}\\
&  =\lambda\left[
\begin{array}
[c]{c}%
-\delta_{\beta}^{\alpha}\\
0
\end{array}
\right]  +\left[
\begin{array}
[c]{c}%
h_{~A}^{\beta}\mathfrak{N}_{~\alpha}^{Ai}k_{i}\\
C_{B}^{\Delta0}\mathfrak{N}_{~\alpha}^{Ai}k_{i}%
\end{array}
\right] \nonumber
\end{align}
has the form of the pencil (\ref{pencil_1}) in appendix \ref{Ap_lemmas}. So,
by making use of the lemma \ref{Lemma_1}) (in this appendix) and noticing that
\ $\lambda I_{\alpha}^{A}+K_{\alpha}^{A}$ and $\mathfrak{N}_{~\alpha}%
^{Bb}l_{b}\left(  \lambda\right)  $ are related by an invertible matrix
independent of $\lambda$, we conclude that
\[
\dim\left(  left\_\ker\left(  \mathfrak{N}_{~\alpha}^{Bb}l_{b}\left(
\lambda\right)  \right)  \right)  =c
\]
for any $\lambda$ different from the generalized eigenvalues $\lambda
_{i}\left(  k\right)  $. This means that we have to choose $c$ linearly
independent vectors of (\ref{left_ker_1}) as the left kernel basis of
$\mathfrak{N}_{~\alpha}^{Bb}l_{b}\left(  \lambda\right)  $.

First, since $C_{A}^{\Gamma0}\mathfrak{N}_{~\alpha}^{Ai}k_{i}$ plays a very
important role in the rest of the proof, we introduce some definitions
associated with this operator. For each $k_{i}$, we call
\begin{equation}
d\left(  k\right)  :=\dim\left(  right\_\ker\left(  C_{A}^{\Gamma
0}\mathfrak{N}_{~\alpha}^{Ai}k_{i}\right)  \right)  \label{Eq_d_1}%
\end{equation}
to the dimension of the right kernel of $C_{A}^{\Gamma0}\mathfrak{N}_{~\alpha
}^{Ai}k_{i}$,
\begin{equation}
r\left(  k\right)  :=rank\left(  C_{A}^{\Gamma0}\mathfrak{N}_{~\alpha}%
^{Ai}k_{i}\right)  \label{Eq_r_1}%
\end{equation}
to its rank and
\begin{equation}
s\left(  k\right)  =:\dim\left(  left\_\ker\left(  C_{A}^{\Gamma0}%
\mathfrak{N}_{~\alpha}^{Ai}k_{i}\right)  \right)  \label{Eq_s_1}%
\end{equation}
to the dimension of its left kernel. \ We also recall that, by the
rank-nullity theorem,
\begin{align}
u  &  =r\left(  k\right)  +d\left(  k\right)  ,\label{Eq_u_r_d_1}\\
c  &  =r\left(  k\right)  +s\left(  k\right)  . \label{Eq_c_r_s_2}%
\end{align}

By hypothesis v), for each $k_{i}$, the vectors $M_{\Delta}^{\tilde{\Delta}%
j}k_{j}$ expand the left kernel of $C_{A}^{\Gamma0}\mathfrak{N}_{~\alpha}%
^{Ai}k_{i}$ (whose dimension is $s\left(  k\right)  $). Therefore, we choose
the following $s\left(  k\right)  $ linearly independent vectors
\[
m_{B}^{zj}k_{j}:=h_{\tilde{\Delta}}^{r}\left[
\begin{array}
[c]{cc}%
0 & M_{\Delta}^{\tilde{\Delta}j}k_{j}%
\end{array}
\right]  \left[
\begin{array}
[c]{c}%
h_{B}^{\alpha}\\
C_{B}^{\Delta0}%
\end{array}
\right]
\]
as part of our basis. These are obtained from the rows of (\ref{left_ker_1}),
where $z=1,...,s\left(  k\right)  $ and $h_{\tilde{\Delta}}^{r}$ is the
projector representing our choice. \ These vectors belong to the left kernel
of $\mathfrak{N}_{~\alpha}^{Bb}l_{b}\left(  \lambda\right)  $ and have the
property that they do not depend on $\lambda$. Notice also that by the form of
(\ref{eq_pencil_h_C_K_1}) any other vector of the left kernel of
$\mathfrak{N}_{~\alpha}^{Bb}l_{b}\left(  \lambda\right)  $, linearly
independent from the $m_{B}^{zj}k_{j}$, will depend on $\lambda$.

On the other hand, since for each $k_{i}$, the rank of $C_{A}^{\Gamma
0}\mathfrak{N}_{~\alpha}^{Ai}k_{i}$ is $r\left(  k\right)  $ and since
$C_{A}^{\Gamma0}\mathfrak{N}_{~\alpha}^{Ai}k_{i}$ appears explicitly in the
first rows of (\ref{left_ker_1}), we complete our basis with the following
linearly independent $r\left(  k\right)  $ vectors%
\[
c_{B}^{wa}l_{a}\left(  \lambda\right)  :=h_{\Gamma}^{w}\left[
\begin{array}
[c]{cc}%
-C_{A}^{\Gamma0}\mathfrak{N}_{~\alpha}^{Aj}k_{j} & -\lambda\delta_{\Delta
}^{\Gamma}+\left(  C_{A}^{\Gamma j}h_{~\Delta}^{A}\right)  k_{j}%
\end{array}
\right]  \left[
\begin{array}
[c]{c}%
h_{B}^{\alpha}\\
C_{B}^{\Delta0}%
\end{array}
\right]  .
\]
Where $w=1,...,r\left(  k\right)  ,$ and $h_{\Gamma}^{w}$ is the projector
associated with our choice. It only remains to show that and $c_{B}^{wa}%
l_{a}\left(  \lambda\right)  $ are linearly independent between them. This is
shown by recalling that $h_{\Gamma}^{w}\left(  -C_{A}^{\Gamma0}\mathfrak{N}%
_{~\alpha}^{Aj}k_{j}\right)  $ has no left kernel (in the index $w$) and thus
the only $\left[
\begin{array}
[c]{cc}%
X_{w} & X_{r}%
\end{array}
\right]  $ which cancels the first $u$ columns of
\[
\left[
\begin{array}
[c]{cc}%
h_{\Gamma}^{w}\left(  -C_{A}^{\Gamma0}\mathfrak{N}_{~\alpha}^{Aj}k_{j}\right)
& h_{\Gamma}^{w}\left(  -\lambda\delta_{\Delta}^{\Gamma}+\left(  C_{A}^{\Gamma
j}h_{~\Delta}^{A}\right)  k_{j}\right) \\
0 & h_{\tilde{\Delta}}^{r}\left(  M_{\Delta}^{\tilde{\Delta}j}k_{j}\right)
\end{array}
\right]  ,
\]
is the trivial ones.

We conclude by (\ref{Eq_c_r_s_2}) that our chosen basis $\left\{  m_{B}%
^{zj}k_{j},c_{B}^{wa}l_{a}\left(  \lambda\right)  \right\}  $ has $c$ vectors,
where the $m_{B}^{zj}k_{j}$ are $s\left(  k\right)  $ vectors that do not
depend on $\lambda$ and the $c_{B}^{wa}l_{a}\left(  \lambda\right)  $ are
$r\left(  k\right)  $ vectors that depend linearly on $\lambda$. This is a
base of the left kernel of $\mathfrak{N}_{~\alpha}^{Ab}l_{b}\left(
\lambda\right)  $ valid for $\lambda$ different to the generalized eigenvalues.

\subsubsection{Kronecker decomposition of $\mathfrak{N}_{~\beta}^{Bb}%
l_{b}\left(  \lambda\right)  $ \label{Kronecker_decomposition}}

The aim of this subsubsection is to give the Kronecker decomposition of the
principal symbol (pencil matrix)%
\begin{equation}
\mathfrak{N}_{~\beta}^{Bb}l_{b}\left(  \lambda\right)  =-\lambda
\mathfrak{N}_{~\beta}^{B0}+\mathfrak{N}_{~\beta}^{Bi}k_{i}.
\label{Eq_pencil_K_l_1}%
\end{equation}
where $l_{b}\left(  \lambda\right)  =-\lambda n_{a}+k_{a}\in S_{n_{a}}$. This
decomposition will guide us in the following subsubsections to complete the
proof of the theorem.

This pencil is an $e\times u$ matrix in $%
\mathbb{C}
$, although it makes no difference if we think of its components in $%
\mathbb{R}
$. The Kronecker decomposition consists of rewriting the pencil as
\[
\mathfrak{N}_{~\beta}^{Bb}l_{b}\left(  \lambda\right)  =Y_{~B}^{A}\left(
k\right)  K_{~\alpha}^{B}\left(  \lambda,k\right)  W_{~\beta}^{\alpha}\left(
k\right)  .
\]
Where $Y_{~B}^{A}\left(  k\right)  \in%
\mathbb{C}
^{e\times e}$ and $W_{~\beta}^{\alpha}\left(  k\right)  \in%
\mathbb{C}
^{u\times u}$ are two invertible matrices that depend on $k_{i}$ and not on
$\lambda$; and $K_{~\alpha}^{B}\left(  \lambda,k\right)  \in%
\mathbb{C}
^{e\times u}$ is a block matrix (see \cite{gantmacher1992theory},
\cite{gantmakher1998theory} and \cite{Abalos:2018uwg} for details), which
depends on $\lambda$ and $k_{i}$. The matrices $Y_{~B}^{A}\left(  k\right)  $
and $W_{~\beta}^{\alpha}\left(  k\right)  $ can be thought of as
change-of-basis matrices, and $K_{~\alpha}^{B}\left(  \lambda,k\right)  $ is
the block matrix which we call the Kronecker structure.

We know from Lemmas 1 and 2 of \cite{Abalos:2018uwg} (whose hypotheses are
valid in this proof), that $K_{~\alpha}^{B}\left(  \lambda,k\right)  $ has all
their Jordan blocks of dimension 1. We also know from subsection 3.1 of
\cite{Abalos:2018uwg} that since $\mathfrak{N}_{~\beta}^{B0}$ has only trivial
right kernel, the remaining blocks appearing in this decomposition are the
$L_{m}^{T}$ and zero rows (called $L_{0}^{T}$ simplifying the notation). This
decomposition depends on $k_{i}$, i.e. the generalized eigenvectors, the
number of them, the block structure $L_{m}^{T}$ (how many and which ones) and
the number of zero rows can be modified for different $k_{i}$'s.

We say that the Kronecker structure of the $e\times u$ pencil $\mathfrak{N}%
_{~\beta}^{Bb}l_{b}\left(  \lambda\right)  $ has the following structure:
\begin{equation}
d\left(  k\right)  \times J_{1},~r\left(  k\right)  \times L_{1}^{T},~s\left(
k\right)  \times L_{0}^{T}. \label{Kro_K_1}%
\end{equation}
Where for each $k_{i}$, the quantities $d\left(  k\right)  $, $r\left(
k\right)  $, $s\left(  k\right)  $ are defined by (\ref{Eq_d_1}),
(\ref{Eq_r_1}) and (\ref{Eq_s_1}) respectively and they also satisfy the
equations (\ref{Eq_u_r_d_1}) and (\ref{Eq_c_r_s_2}). We justify each of the
terms of (\ref{Kro_K_1}) below.

The following analysis is valid for each $k_{i}$, so we assume $k_{i}$ is fixed.

$\bullet$ We begin by justifying the block structure $r\left(  k\right)
\times L_{1}^{T},~s\left(  k\right)  \times L_{0}^{T}$.

Let us study the left kernel of $L_{m}^{T}=\left[
\begin{array}
[c]{ccccc}%
\lambda & 0 & 0 & 0 & 0\\
1 & \lambda & 0 & 0 & 0\\
0 & 1 & ... & 0 & 0\\
0 & 0 & ... & \lambda & 0\\
0 & 0 & 0 & 1 & \lambda\\
0 & 0 & 0 & 0 & 1
\end{array}
\right]  \in%
\mathbb{R}
^{m+1\times m}$. It has dimension 1 and is expanded by the vector%
\[
X=\left[  -1,\lambda,...,\left(  -1\right)  ^{m}\lambda^{m-1},\left(
-1\right)  ^{m+1}\lambda^{m}\right]  \in%
\mathbb{R}
^{1\times m+1}.
\]
\ The coefficients of $X$ are polynomials in $\lambda$ whose major degree is
$m$. We can increase the degree of these coefficients, for example, by
considering the vector $\lambda X$, but we cannot reduce it without obtaining
rational functions. \ This allows us to introduce a method to detect the
$L_{m}^{T}$ blocks present in a given pencil. \ We first define the function
$gr$ which takes vectors with polynomial coefficients in $\lambda$ and returns
the greater polynomial degree between their coefficients (in our example,
$gr\left(  X\right)  =m$ and $gr\left(  \lambda X\right)  =m+1$). If we now
consider a pencil with different blocks $L_{m_{1}}^{T},...,L_{m_{v}}^{T}$ such
that $0\leq m_{1}\leq m_{2}\leq...\leq m_{v}$, there exists a left kernel
basis $\left\{  X_{i}\text{ with }i=1,...,v\right\}  $ of this pencil such
that the values of $gr\left(  X_{i}\right)  =m_{i}$ identify the types of
$L_{m}^{T}$ blocks present in the pencil. Any other basis $\left\{
Z_{i}\right\}  $, with $gr\left(  Z_{i}\right)  =z_{i}$ and ordered such that
$z_{1}\leq z_{2}\leq...\leq z_{p}$, has at least one $z_{i}$ such that
$m_{i}\leq z_{i}$ with $i$ between $1$ and $v$. \ This indicates that to find
the Kronecker structure of a pencil, we have to identify a left kernel basis
of the pencil that minimizes the $gr$ function for their elements. Notice
that, if we have a zero row $L_{0}^{T}$, its left kernel $F$ can be chosen
independent of $\lambda$, then $gr\left(  F\right)  =0$.

Applying this method to $\mathfrak{N}_{~\beta}^{Bb}l_{b}\left(  \lambda
\right)  $, using the basis $\left\{  m_{B}^{zj}k_{j}\text{, }c_{B}^{wa}%
l_{a}\left(  \lambda\right)  \right\}  $ with $z=1,...,s\left(  k\right)  $,
$w=1,...,r\left(  k\right)  $ \ from the previous subsubsection, which
satisfies that $gr\left(  m_{B}^{rj}k_{j}\right)  =0$ and $gr\left(
c_{B}^{wa}l_{a}\left(  \lambda\right)  \right)  =1$, we conclude that the
Kronecker structure of $\mathfrak{N}_{~\beta}^{Bb}l_{b}\left(  \lambda\right)
$ has the blocks $r\left(  k\right)  \times L_{1}^{T},~s\left(  k\right)
\times L_{0}^{T}$ and the rest of the structure are Jordan blocks.

$\bullet$ We have already explained that this pencil only has Jordan blocks of
dimension 1, it remains to explain that they are $d\left(  k\right)  $ blocks.

This result can be concluded by counting the number of $\lambda$'s appearing
in the columns of $K_{~\alpha}^{B}\left(  \lambda,k\right)  $. Since only one
$\lambda$ can be present per column and $K_{~\alpha}^{B}\left(  \lambda
,k\right)  $ has $u$ columns, there are $u$ $\lambda$'s in $K_{~\alpha}%
^{B}\left(  \lambda,k\right)  $. Notice that $L_{0}^{T}$ has no $\lambda$
since it is a row of zeros and that each $L_{1}^{T}$ has only one $\lambda$
per column, then in $r\left(  k\right)  \times L_{1}^{T},~s\left(  k\right)
\times L_{0}^{T}$\ there are $r\left(  k\right)  $ $\lambda$'s. Finally, the
number of Jordan blocks of dimension 1 is $u-r\left(  k\right)  =d\left(
k\right)  $ (see eq.(\ref{Eq_u_r_d_1})).

As a final comment, we note that the sum of the multiplicities of each of the
different $q\left(  k\right)  -$generalized eigenvalues has to be equal to
$d\left(  k\right)  $, that is,
\[
d_{\lambda_{1}}+d_{\lambda_{2}}+...+d_{\lambda_{q\left(  k\right)  }}=d\left(
k\right)  .
\]

\subsubsection{Basis which diagonalize $A_{~\beta}^{\alpha i}k_{i}$
\label{Eigenbasis_1}}

In this subsubsection, we use the obtained information of the Kronecker
decomposition of $\mathfrak{N}_{~\beta}^{Bb}l_{b}\left(  \lambda\right)  $ to
find the bases that diagonalize $A_{~\beta}^{\alpha i}k_{i}$.

The eigenvectors which diagonalize $A_{~\beta}^{\alpha i}k_{i}$ are divided
into two groups:

$\bullet$ The generalized eigenvectors $\delta\phi_{\lambda_{i}\left(
k\right)  }^{\beta}$, associated to the generalized eigenvalues $\lambda
_{i}\left(  k\right)  $, such that they satisfy%
\[
\mathfrak{N}_{~\beta}^{Bb}\left(  -\lambda_{i}\left(  k\right)  n_{a}%
+k_{a}\right)  \delta\phi_{\lambda_{i}\left(  k\right)  }^{\beta}=0,
\]
with $\lambda_{1}\left(  k\right)  \leq\lambda_{2}\left(  k\right)
\leq...\leq\lambda_{d\left(  k\right)  }\left(  k\right)  $. We know from the
previous subsubsection that for each $\lambda_{i}\left(  k\right)  $ there is
a generalized eigenvector $\delta\phi_{\lambda_{i}\left(  k\right)  }^{\beta}%
$. Notice that to simplify the notation, we have changed how we denote the
generalized eigenvalues with respect to the section \ref{Sec_Const_coeff_1}.

From the above equation, it follows that
\[
\left[
\begin{array}
[c]{c}%
h_{B}^{\alpha}\\
C_{B}^{\Delta0}%
\end{array}
\right]  \mathfrak{N}_{~\beta}^{Bb}\left(  -\lambda_{i}\left(  k\right)
n_{a}+k_{a}\right)  \delta\phi_{\lambda_{i}\left(  k\right)  }^{\beta}=\left[
\begin{array}
[c]{c}%
-\lambda_{i}\left(  k\right)  \delta_{\beta}^{\alpha}+A_{\beta}^{\alpha
i}k_{i}\\
C_{B}^{\Delta0}\mathfrak{N}_{~\beta}^{Bi}k_{i}%
\end{array}
\right]  \delta\phi_{\lambda_{i}\left(  k\right)  }^{\beta}=0,
\]
therefore, these generalized eigenvectors are eigenvectors of $A_{\beta
}^{\alpha i}k_{i}$ and belong to the right kernel of $C_{B}^{\Delta
0}\mathfrak{N}_{~\beta}^{Bi}k_{i}$.

Recalling that the dimension of the right kernel of \ $C_{B}^{\Delta
0}\mathfrak{N}_{~\beta}^{Bi}k_{i}$ is $d\left(  k\right)  $ (see eq.
(\ref{Eq_d_1})) and that the $\delta\phi_{\lambda_{i}\left(  k\right)
}^{\beta}$ are $d\left(  k\right)  $ linearly independent vectors, we conclude
that%
\begin{equation}
right\_\ker\left(  C_{B}^{\Delta0}\mathfrak{N}_{~\beta}^{Bi}k_{i}\right)
=span\left\langle \delta\phi_{\lambda_{i}\left(  k\right)  }^{\beta
}\right\rangle . \label{Eq_span_ker_der_C0_Ki_1}%
\end{equation}
We also recall that the set of generalized eigenvalues $\left\{  \lambda
_{i}\left(  k\right)  \right\}  $ are those we called the "physical" at the
beginning of the proof.

$\bullet$ The eigenvectors $\delta\phi_{\pi_{i}\left(  k\right)  }^{\beta}$,
associated to the eigenvalues $\pi_{i}\left(  k\right)  $, such that they
satisfy
\begin{align}
\left(  -\pi_{i}\left(  k\right)  \delta_{\beta}^{\alpha}+A_{\beta}^{\alpha
i}k_{i}\right)  \delta\phi_{\pi_{i}\left(  k\right)  }^{\beta}  &
=0,\label{eq_phi_pi_1}\\
C_{B}^{\Delta0}\mathfrak{N}_{~\beta}^{Bi}k_{i}\delta\phi_{\pi_{i}\left(
k\right)  }^{\beta}  &  \neq0.\nonumber
\end{align}

We are considering that the reduction $h_{~A}^{\beta}\left(  k\right)  $ is
chosen such that, for each $k_{i}$, the $\left\{  \pi_{i}\left(  k\right)
\right\}  $ are simple, different from each other and different from the
$\left\{  \lambda_{i}\left(  k\right)  \right\}  $. We order them in the
following way $\pi_{1}\left(  k\right)  <\pi_{2}\left(  k\right)
<...<\pi_{r\left(  k\right)  }\left(  k\right)  $ and notice that since they
are simple, each $\pi_{i}\left(  k\right)  $ has its associated eigenvector
$\delta\phi_{\pi_{i}\left(  k\right)  }^{\beta}$. Notice that $\left\{
\lambda_{i}\left(  k\right)  ,\text{ }\pi_{i}\left(  k\right)  \right\}  $ and
$\left\{  \delta\phi_{\lambda_{i}\left(  k\right)  }^{\beta},\text{ }%
\delta\phi_{\pi_{i}\left(  k\right)  }^{\beta}\right\}  $ are the sets of the
eigenvalues and eigenvectors of $A_{\beta}^{\alpha i}k_{i}$ and each of these
sets has $u=d\left(  k\right)  +r\left(  k\right)  $ elements. \ Finally, we
recall that the set of eigenvalues $\left\{  \pi_{i}\left(  k\right)
\right\}  $ was previously called the "constraints 1".

\subsubsection{Kronecker structure of the subsidiary system
(\ref{Eq_pencil_cons_1}) \label{Kro_sub}}

In this subsubsection, we show that the Kronecker structure of the pencil
(\ref{Eq_pencil_cons_1}) is
\begin{equation}
J_{1}\left(  \pi_{1}\left(  k\right)  \right)  ,...,J_{1}\left(  \pi_{r\left(
k\right)  }\left(  k\right)  \right)  ,s\left(  k\right)  \times L_{1}%
^{T},y\left(  k\right)  \times L_{0}^{T}, \label{eq_sub_kron_1}%
\end{equation}
where $y\left(  k\right)  :=\dim\left(  left\_\ker\left(  M_{\Delta}%
^{\tilde{\Delta}j}k_{j}\right)  \right)  $.

$\bullet$ We begin by showing that $J_{1}\left(  \pi_{1}\left(  k\right)
\right)  ,...,J_{1}\left(  \pi_{r\left(  k\right)  }\left(  k\right)  \right)
$ is part of the Kronecker structure.

For this purpose, we will show that the set of $\pi_{i}\left(  k\right)  $ are
the generalized eigenvalues of (\ref{Eq_pencil_cons_1}), with their
corresponding linearly independent generalized eigenvectors
\begin{equation}
\delta\psi_{\pi_{i}\left(  k\right)  }^{\Delta}:=C_{B}^{\Delta0}%
\mathfrak{N}_{~\beta}^{Bi}k_{i}\delta\phi_{\pi_{i}\left(  k\right)  }^{\beta};
\label{Eq_chi_pi_1}%
\end{equation}
i.e., for each $i=1,...,r\left(  k\right)  $, it holds that
\begin{equation}
\left[
\begin{array}
[c]{c}%
-\delta_{\Delta}^{\Gamma}\pi_{i}\left(  k\right)  +C_{A}^{\Gamma j}h_{~\Delta
}^{A}k_{j}\\
M_{\Delta}^{\tilde{\Delta}j}k_{j}%
\end{array}
\right]  \delta\psi_{\pi_{i}\left(  k\right)  }^{\Delta}=0.
\label{Eq_vin_chi_1}%
\end{equation}

First, recalling that $M_{\Delta}^{\tilde{\Delta}j}k_{j}$ expands the left
kernel of $C_{B}^{\Delta0}\mathfrak{N}_{~\beta}^{Bi}k_{i}$, we conclude that%
\[
M_{\Delta}^{\tilde{\Delta}j}k_{j}\delta\psi_{\pi_{i}\left(  k\right)
}^{\Delta}=M_{\Delta}^{\tilde{\Delta}j}k_{j}C_{B}^{\Delta0}\mathfrak{N}%
_{~\beta}^{Bi}k_{i}\delta\phi_{\pi_{i}\left(  k\right)  }^{\beta}=0.
\]

Second, evaluating the expression (\ref{Eq_simbs_1}) at $\lambda=\pi
_{i}\left(  k\right)  $, multiplying by $\delta\phi_{\pi_{i}\left(  k\right)
}^{\beta}$ and recalling the eq. (\ref{eq_phi_pi_1}),\ we obtain that%
\begin{align*}
0  &  =\left(  k_{j}C_{A}^{\Gamma0}\mathfrak{N}_{~\alpha}^{Aj}\right)  \left(
-\pi_{i}\left(  k\right)  \delta_{\alpha}^{\beta}+A_{\alpha}^{\beta i}%
k_{i}\right)  \delta\phi_{\pi_{i}\left(  k\right)  }^{\alpha},\\
&  =\left(  -\pi_{i}\left(  k\right)  \delta_{\Delta}^{\Gamma}+C_{A}^{\Gamma
j}h_{~\Delta}^{A}k_{j}\right)  \left(  C_{B}^{\Delta0}\mathfrak{N}_{~\alpha
}^{Bj}k_{j}\right)  \delta\phi_{\pi_{i}\left(  k\right)  }^{\alpha},\\
&  =\left(  -\pi_{i}\left(  k\right)  \delta_{\Delta}^{\Gamma}+C_{A}^{\Gamma
j}h_{~\Delta}^{A}k_{j}\right)  \delta\psi_{\pi_{i}\left(  k\right)  }^{\Delta
}.
\end{align*}
This shows that equations (\ref{Eq_vin_chi_1}) holds. Let us now demonstrate
that the $\delta\psi_{\pi_{i}\left(  k\right)  }^{\Delta}$ are linearly
independent. For this purpose, we assume that they are not, i.e. that there
exists $U^{\pi_{i}\left(  k\right)  }$ such that $0=U^{\pi_{i}\left(
k\right)  }\delta\psi_{\pi_{i}\left(  k\right)  }^{\Delta}$ (the sum runs into
the $i$ index) and we conclude that $U^{\pi_{i}\left(  k\right)  }=0$. We note
that%
\[
0=U^{\pi_{i}\left(  k\right)  }\delta\psi_{\pi_{i}\left(  k\right)  }^{\Delta
}=U^{\pi_{i}\left(  k\right)  }C_{B}^{\Delta0}\mathfrak{N}_{~\beta}^{Bi}%
k_{i}\delta\phi_{\pi_{i}\left(  k\right)  }^{\beta}=C_{B}^{\Delta
0}\mathfrak{N}_{~\beta}^{Bi}k_{i}\left(  U^{\pi_{i}\left(  k\right)  }%
\delta\phi_{\pi_{i}\left(  k\right)  }^{\beta}\right)  ,
\]
thus $U^{\pi_{i}\left(  k\right)  }\delta\phi_{\pi_{i}\left(  k\right)
}^{\beta}$ belongs to the right-hand kernel of $C_{B}^{\Delta0}\mathfrak{N}%
_{~\beta}^{Bi}k_{i}$ and, by (\ref{Eq_span_ker_der_C0_Ki_1}), should be a
linear combination of the $\left\{  \delta\phi_{\lambda_{i}\left(  k\right)
}^{\beta}\right\}  $. Since these $\left\{  \delta\phi_{\lambda_{i}\left(
k\right)  }^{\beta}\right\}  $ are linearly independent of the $\left\{
\delta\phi_{\pi_{i}\left(  k\right)  }^{\beta}\right\}  $ by construction, it
should be that $U^{\pi_{i}\left(  k\right)  }=0$.

$\bullet$ \ It remains to show that $s\left(  k\right)  \times L_{1}%
^{T},y\left(  k\right)  \times L_{0}^{T}$ is the other part of the Kronecker
structure. This can be concluded directly from lemma \ref{Lemma_2} (in
appendix \ref{Ap_lemmas}) if we prove the condition (\ref{eq_d1_d2_d_1}). This
condition follows from recalling that
\[
rigth\_\ker\left(  M_{\Delta}^{\tilde{\Delta}j}k_{j}\right)  =\left\langle
\delta\psi_{\pi_{i}\left(  k\right)  }^{\Delta}\right\rangle .
\]
Furthermore, we conclude that the pairs $\left\{  \pi_{i}\left(  k\right)
\right\}  $,$\left\{  \delta\psi_{\pi_{i}\left(  k\right)  }^{\Delta}\right\}
$ are all the generalized eigenvalues and eigenvectors of the pencil.

Notice that we called the "constraints 1" to the eigenvalues $\left\{  \pi
_{i}\left(  k\right)  \right\}  $, they are at the same time the generalized
eigenvalues of the system (\ref{Eq_pencil_cons_1}), i.e. they are the
"constraints 1" eigenvalues of $A_{\beta}^{\alpha i}k_{i}$ and the "physic"
eigenvalues for the pencil (\ref{Eq_pencil_cons_1}).

\subsubsection{Basis which diagonalize $B_{~\Delta}^{\Gamma j}k_{j}$
\ \label{Diag_sub}}

The previous subsubsection is the proof of a) and b) (stated at the beginning
of this proof). Therefore, there exists $N_{\tilde{\Delta}}^{\Gamma}$ such
that $B_{~\Delta}^{\Gamma j}k_{j}=C_{A}^{\Gamma j}h_{~\Delta}^{A}%
k_{j}+N_{\tilde{\Delta}}^{\Gamma}M_{\Delta}^{\tilde{\Delta}j}k_{j}$ is
uniformly diagonalizable. This concludes the proof of the theorem in the case
where $M_{\Delta}^{\tilde{\Delta}j}$ is non zero.

In this subsubsection, we will find the bases that diagonalize $B_{~\Delta
}^{\Gamma j}k_{j}$. We will use these bases with the bases that diagonalize
$A_{\beta}^{\alpha j}k_{j}$ to give a very simple expression of the equation
(\ref{Eq_rel_simbolos_1}). Finally, we will discuss the structure of the
latter equation in the cases where $A_{\beta}^{\alpha j}k_{j}$ and/or
$B_{~\Delta}^{\Gamma j}k_{j}$ are not diagonalizable.

Analogously to the case of $A_{\beta}^{\alpha i}k_{i}$, the eigenvectors which
diagonalize $B_{~\Delta}^{\Gamma j}k_{j}$ are also divided into two groups for
each $k_{i}$:

$\bullet$ The eigenvectors $\left\{  \delta\psi_{\pi_{i}\left(  k\right)
}^{\Delta}\right\}  $ with their corresponding eigenvalues $\pi_{i}\left(
k\right)  $, with $i=1,...,r\left(  k\right)  $, founded in the previous
subsubsection (see (\ref{Eq_chi_pi_1})) and such that they satisfy%
\[
\left(  -\delta_{\Delta}^{\Gamma}\pi_{i}\left(  k\right)  +C_{A}^{\Gamma
j}h_{~\Delta}^{A}k_{j}+N_{\tilde{\Delta}}^{\Gamma}M_{\Delta}^{\tilde{\Delta}%
j}k_{j}\right)  \delta\psi_{\pi_{i}\left(  k\right)  }^{\Delta}=0.
\]

$\bullet$ And the eigenvectors $\delta\psi_{\rho_{i}\left(  k\right)
}^{\Delta}$ associated to the eigenvalues $\rho_{i}\left(  k\right)  $, with
$i=1,..,s\left(  k\right)  $, such that%

\begin{align*}
\left(  -\delta_{\Delta}^{\Gamma}\rho_{i}\left(  k\right)  +C_{A}^{\Gamma
j}h_{~\Delta}^{A}k_{j}+N_{\tilde{\Delta}}^{\Gamma}M_{\Delta}^{\tilde{\Delta}%
j}k_{j}\right)  \delta\psi_{\rho_{i}\left(  k\right)  }^{\Delta}  &  =0\\
M_{\Delta}^{\tilde{\Delta}j}k_{j}\delta\psi_{\rho_{i}\left(  k\right)
}^{\Delta}  &  \neq0.
\end{align*}
Where the set $\left\{  \rho_{i}\left(  k\right)  \right\}  $ are simple,
different from each other and different from the $\left\{  \pi_{i}\left(
k\right)  \right\}  $. These $\left\{  \rho_{i}\left(  k\right)  \right\}  $
was called "constraints 2".

The set $\left\{  \delta\psi_{\pi_{i}\left(  k\right)  }^{\Delta}\text{,
}\delta\psi_{\rho_{i}\left(  k\right)  }^{\Delta}\right\}  $ uniformly
diagonalizes $B_{~\Delta}^{\Gamma j}k_{j}$ (since, for each $k_{i}$, all its
eigenvalues $\left\{  \pi_{i}\left(  k\right)  ,\rho_{i}\left(  k\right)
\right\}  $ are simple) showing that the system (\ref{Eq_sist_B_1}) is
strongly hyperbolic for this choice of $N_{\tilde{\Delta}}^{\Gamma}$.

Let us now rewrite the equation (\ref{Eq_rel_simbolos_1}) in the founded bases.

$\bullet$ For each $k_{i}$, we consider the co-basis $\left\{  \delta
\phi_{\beta}^{\pi_{i}\left(  k\right)  }\text{, }\delta\phi_{\beta}%
^{\lambda_{i}\left(  k\right)  }\right\}  $ and $\left\{  \delta\psi_{\Delta
}^{\pi_{i}\left(  k\right)  }\text{, }\delta\psi_{\Delta}^{\rho_{i}\left(
k\right)  }\right\}  $ of the bases $\left\{  \delta\phi_{\pi_{i}\left(
k\right)  }^{\beta}\text{, }\delta\phi_{\lambda_{i}\left(  k\right)  }^{\beta
}\right\}  $and $\left\{  \delta\psi_{\pi_{i}\left(  k\right)  }^{\Delta
}\text{, }\delta\psi_{\rho_{i}\left(  k\right)  }^{\Delta}\right\}  $
respectively. They satisfy
\begin{align*}
&
\begin{array}
[c]{ccc}%
\delta\phi_{\beta}^{\lambda_{i}\left(  k\right)  }\delta\phi_{\lambda
_{j}\left(  k\right)  }^{\beta}=\delta_{j}^{i} &  & \delta\phi_{\beta}%
^{\pi_{i}\left(  k\right)  }\delta\phi_{\lambda_{j}\left(  k\right)  }^{\beta
}=0
\end{array}
\\
&
\begin{array}
[c]{ccc}%
\delta\phi_{\beta}^{\lambda_{i}\left(  k\right)  }\delta\phi_{\pi_{j}\left(
k\right)  }^{\beta}=0 &  & \delta\phi_{\beta}^{\pi_{i}\left(  k\right)
}\delta\phi_{\pi_{j}\left(  k\right)  }^{\beta}=\delta_{j}^{i}%
\end{array}
\end{align*}%
\begin{align*}
&
\begin{array}
[c]{ccc}%
\delta\psi_{\Delta}^{\pi_{i}\left(  k\right)  }\delta\psi_{\pi_{j}\left(
k\right)  }^{\Delta}=\delta_{j}^{i} &  & \delta\psi_{\Delta}^{\rho_{i}\left(
k\right)  }\delta\psi_{\pi_{j}\left(  k\right)  }^{\Delta}=0
\end{array}
\\
&
\begin{array}
[c]{ccc}%
\delta\psi_{\Delta}^{\pi_{i}\left(  k\right)  }\delta\psi_{\rho_{j}\left(
k\right)  }^{\Delta}=0 &  & \delta\psi_{\Delta}^{\rho_{i}\left(  k\right)
}\delta\psi_{\rho_{j}\left(  k\right)  }^{\Delta}=\delta_{j}^{i}%
\end{array}
\end{align*}

In these bases, the equation (\ref{Eq_rel_simbolos_1}) reduces to
\begin{equation}
\left[
\begin{array}
[c]{cc}%
I & 0\\
0 & 0
\end{array}
\right]  \left[
\begin{array}
[c]{cc}%
\Pi & 0\\
0 & \Lambda
\end{array}
\right]  =\left[
\begin{array}
[c]{cc}%
\Pi & 0\\
0 & \Theta
\end{array}
\right]  \left[
\begin{array}
[c]{cc}%
I & 0\\
0 & 0
\end{array}
\right]  \label{Eq_bloques_1}%
\end{equation}
where \newline
$I=\left[
\begin{array}
[c]{ccc}%
1 & 0 & 0\\
0 & ... & 0\\
0 & 0 & 1
\end{array}
\right]  \in%
\mathbb{R}
^{r\left(  k\right)  \times r\left(  k\right)  }$, $\Pi=-\left[
\begin{array}
[c]{ccc}%
\lambda-\pi_{1}\left(  k\right)  & 0 & 0\\
0 & ... & 0\\
0 & 0 & \lambda-\pi_{r}\left(  k\right)
\end{array}
\right]  \in%
\mathbb{R}
^{r\left(  k\right)  \times r\left(  k\right)  }$, $\Lambda=-\left[
\begin{array}
[c]{ccc}%
\lambda-\lambda_{1}\left(  k\right)  & 0 & 0\\
0 & ... & 0\\
0 & 0 & \lambda-\lambda_{d}\left(  k\right)
\end{array}
\right]  \in%
\mathbb{R}
^{d\left(  k\right)  \times d\left(  k\right)  }$ \ and \newline $\ \Theta=-\left[
\begin{array}
[c]{ccc}%
\lambda-\rho_{1}\left(  k\right)  & 0 & 0\\
0 & ... & 0\\
0 & 0 & \lambda-\rho_{s}\left(  k\right)
\end{array}
\right]  \in%
\mathbb{R}
^{s\left(  k\right)  \times s\left(  k\right)  }$.

If we consider the not strongly hyperbolic case, where $\Lambda$ has at least
one Jordan block $J_{m}$ with $m\geq2$. Eq. (\ref{Eq_bloques_1}) implies that
the matrix $\left[
\begin{array}
[c]{cc}%
\Pi & 0\\
0 & \Theta
\end{array}
\right]  $ could still be diagonalizable and hence the subsidiary system would
be strongly hyperbolic.

On the other hand, in the case where $\Lambda$ shares some eigenvalue with
$\Pi$ forming a Jordan block, for example,%
\[
-\left[
\begin{array}
[c]{cc}%
\Pi & 0\\
0 & \Lambda
\end{array}
\right]  =\left[
\begin{array}
[c]{cccccc}%
\lambda-\pi_{1}\left(  k\right)  & 0 & 0 & 0 & 0 & 0\\
0 & ... & 0 & 0 & 0 & 0\\
0 & 0 & \lambda-\lambda_{1}\left(  k\right)  & 0 & 0 & 0\\
0 & 0 & 1 & \lambda-\lambda_{1}\left(  k\right)  & 0 & 0\\
0 & 0 & 0 & 0 & ... & 0\\
0 & 0 & 0 & 0 & 0 & \lambda-\lambda_{d}\left(  k\right)
\end{array}
\right]  ,
\]
by the eq. (\ref{Eq_bloques_1}), it could still happen that the subsidiary
system is diagonalizable.

We concluded that by proposing different options for \ $\left[
\begin{array}
[c]{cc}%
\Pi & 0\\
0 & \Lambda
\end{array}
\right]  $ and $\left[
\begin{array}
[c]{cc}%
\Pi & 0\\
0 & \Theta
\end{array}
\right]  $ such that they satisfy (\ref{Eq_bloques_1}) we obtain all possible
cases where the evolution equations and the subsidiary system are well-posed
or not.

\subsubsection{Case without $M_{\Delta}^{\tilde{\Delta}j}$}

Let us now study the case where the system (\ref{Eq_sis_1_pseudo_1}) does not
admit $M_{\Delta}^{\tilde{\Delta}j}$, i.e., by condition v), $C_{A}^{\Gamma
0}\mathfrak{N}_{~\alpha}^{Ai}k_{i}$ only has trivial left kernel for any
$k_{i}$. In this case, $s\left(  k\right)  =0$ and $c=r\left(  k\right)  $ by
equation (\ref{Eq_c_r_s_2}).

The proof of the theorem is the same as presented above, but now we conclude
that the set $\left\{  \delta\psi_{\pi_{i}\left(  k\right)  }^{\Delta}%
:=C_{B}^{\Delta0}\mathfrak{N}_{~\beta}^{Bi}k_{i}\delta\phi_{\pi_{i}\left(
k\right)  }^{\beta}\right\}  $ uniformly diagonalizes the matrix $B_{~\Delta
}^{\Gamma j}k_{j}:=C_{A}^{\Gamma j}h_{~\Delta}^{A}k_{j}$. Therefore, system
(\ref{Eq_sist_B_1}) is strongly hyperbolic. This concludes the proof of the theorem.

\subsection{Comment about condition v) in theorem SH of the SS \label{teo_3}}

The condition v) of the theorem can be suppressed without losing the strong
hyperbolicity of the subsidiary system in the case of constant coefficients.
For this purpose, we have to change the pencils%
\[
\left[
\begin{array}
[c]{c}%
-\lambda\delta_{\Delta}^{\Gamma}+\left(  C_{A}^{\Gamma j}h_{~\Delta}%
^{A}\right)  k_{j}\\
M_{\Delta}^{\tilde{\Delta}j}k_{j}%
\end{array}
\right]  \rightarrow\left[
\begin{array}
[c]{cc}%
-C_{A}^{\Gamma0}\mathfrak{N}_{~\alpha}^{Aj}k_{j} & -\lambda\delta_{\Delta
}^{\Gamma}+C_{A}^{\Gamma j}h_{~\Delta}^{A}k_{j}\\
0 & M_{\Delta}^{\tilde{\Delta}j}k_{j}\\
0 & X_{\Delta}^{\check{s}}\left(  k\right)
\end{array}
\right]  ,
\]
as explained in subsubsection \ref{Subsi_Fourier} and carry on the same proof
as before but with this new pencil. This change is introduced since
$M_{\Delta}^{\tilde{\Delta}j}k_{j}$ no longer spans the $left\_\ker\left(
C_{B}^{\Delta0}\mathfrak{N}_{~\beta}^{Bi}k_{i}\right)  $, so it is necessary
to add the vectors $X_{\Delta}^{\check{s}}\left(  k\right)  $ such that, for
each $k_{i}$, it holds that
\[
span\left\langle M_{\Delta}^{\tilde{\Delta}j}k_{j},X_{\Delta}^{\check{s}%
}\left(  k\right)  \right\rangle =left\_\ker\left(  C_{B}^{\Delta
0}\mathfrak{N}_{~\beta}^{Bi}k_{i}\right)  .
\]

The cost of this change is reflected in the principal symbol of the subsidiary
system, whose new form is%
\begin{equation}
B_{~\Delta}^{\Gamma j}k_{j}=C_{A}^{\Gamma j}h_{~\Delta}^{A}k_{j}%
+N_{1\tilde{\Delta}}^{\Gamma}M_{\Delta}^{\tilde{\Delta}j}k_{j}+N_{2\check{s}%
}^{\Gamma}X_{\Delta}^{\check{s}}\left(  k\right)  . \label{eq_B_new_1}%
\end{equation}

This change the statement of the theorem.

\begin{theorem}
When all the hypotheses of theorem \ref{Theorem_coef_const_2} are satisfied
except condition v), it is possible to find $N_{1\tilde{\Delta}}^{\Gamma}$ and
$N_{2\check{s}}^{\Gamma}$ such that the principal symbol of the subsidiary
system eq. (\ref{eq_B_new_1}) is uniformly diagonalizable and therefore the
subsidiary system is strongly hyperbolic.
\end{theorem}

Since the $X_{\Delta}^{\check{s}}\left(  k\right)  $ do not come from Geroch
fields, they could be non-zero only for some particular $k_{i}$ or they could
have a non-linear dependence on $k_{i}$. This latter case would imply that the
Fourier anti-transform of the term $N_{2\check{s}}^{\Gamma}X_{\Delta}%
^{\check{s}}\left(  k\right)  $ include second or higher derivatives for any
$N_{2\check{s}}^{\Gamma}$ (except for $N_{2\check{s}}^{\Gamma}=0$). This
changes the final hyperbolic answer of the subsidiary system into a pure
pseudo-differential answer, which can not be directly extrapolated to the
quasi-linear case. Some known systems admit vectors $X_{\Delta}^{\check{s}%
}\left(  k\right)  $, but in general, it is possible to set $N_{2\check{s}%
}^{\Gamma}=0$ and find some $N_{1\tilde{\Delta}}^{\Gamma}$ such that
(\ref{eq_B_new_1}) is uniformly diagonalizable.

\section{Examples \label{Examples}}

In this section, we reproduce some known results about the constraint
propagations of two specific theories: Maxwell electrodynamics and the wave
equation. We use these theories to illustrate the results presented in the
previous sections.

We consider the systems on a space-time $M$ of $\dim M=3+1$ with a background
Lorentzian metric $g_{ab}$ \ (with signature $-,+,+,+$) and with their
equations in first-order derivatives (i.e. eqs. (\ref{eq_max_or_1}%
-\ref{eq_max_or_2}) and (\ref{Eq_W_E_1})).

We show that in the Maxwell case there are no $M_{A}^{\Delta a}$ fields and
the $C_{A}^{\Gamma a}$ fields are associated with the standard constraints
$\psi_{1}:=D_{a}E^{a}-\tilde{J}^{0}$ and $\psi_{2}=D_{a}B^{a}$. We present its
subsidiary equations and comment on its characteristic analysis. We also note
that it is commonly used in the literature that when the equations are coupled
to a source $J^{a}$, this has to satisfy an on-shell integrability condition
$\nabla_{a}J^{a}=0$ to preserve the constraints. However, we will show that
this condition is relaxed in the off-shell case, by choosing the divergence
proportional to the equations of the system (i.e. \ref{eq_max_id_dJ}) and
maintaining the preservation of the constraints.

On the other hand, for the wave equation, we find both $M_{A}^{\tilde{\Delta
}a}$ and $C_{A}^{\Gamma a}$ fields. They appear as a result of reducing the
system from second to a first-order derivative. As we explained in theorem
\ref{teorema_ec_ev_vin_1}, the non-uniqueness of the subsidiary system is
associated with the presence of $M_{A}^{\tilde{\Delta}a}$. Therefore, we
verify this non-uniqueness and comment on its characteristic analysis.

In both cases, we introduce (as in subsection \ref{n+1_decomposition_sec}) a
foliation of $M=\underset{t\in%
\mathbb{R}
}{\cup}\Sigma_{t}$ associated to the function $t:M\rightarrow%
\mathbb{R}
$, with the spatial (with respect to the metric $g_{ab}$) hypersurfaces
$\Sigma_{t}$. In addition, following appendix \ref{App_coordenadas}, we
consider the definitions%
\begin{align*}
n_{a}  &  :=\nabla_{a}t,\\
\tilde{n}_{b}  &  :=-Nn_{b}\text{\ \ \ with \ }N:=\frac{1}{\sqrt{-\nabla
t.\nabla t}},\\
p^{a}  &  =\left(  \partial_{t}\right)  ^{a}-\beta^{a}\text{ \ \ with }\left(
\partial_{t}\right)  ^{a}n_{a}=1\text{ and }\beta^{a}n_{a}=0\\
\tilde{m}^{a}  &  :=\tilde{n}^{a}=\frac{1}{N}p^{a},\text{ }\\
\tilde{n}^{a}\tilde{n}_{a}  &  =-1,
\end{align*}
and the expression (\ref{Eq_met_1}) for the metric $g_{ab}$. We also consider
the projection $\tilde{\eta}_{b}^{a}$ to the hypersurfaces $\Sigma_{t}$ as%
\[
\tilde{\eta}_{b}^{a}:=\delta_{b}^{a}-p^{a}n_{b}=\delta_{b}^{a}+\tilde{n}%
^{a}\tilde{n}_{b}.
\]

\subsection{Maxwell electrodynamics\label{S_Maxwell_eq}}

We define the fields $Q_{1,2}^{d}$ as
\begin{align*}
Q_{1}^{d}  &  :=\nabla_{a}F^{ad}-J^{d},\\
Q_{2}^{d}  &  :=\nabla_{a}\ast F^{ad},
\end{align*}
where $F^{ad}$ is the electromagnetic (antisymmetric) tensor, $\ast
F^{ad}:=\frac{1}{2}\varepsilon_{cq}^{~\ \ \ ad}F^{cq}$ and $J^{d}$ is the
source of the system. Here, $J^{d}=J^{d}\left(  F^{ad},\ast F^{ad}%
,x^{a}\right)  $ may depend on $F^{ad}$, $\ast F^{ad}$ and $x^{a}\in M$ but it
can not depend on derivatives of $F^{ad}$ or $\ast F^{ad}$. In addition, we
assume that it\ satisfies the off-shell identity
\begin{equation}
\nabla_{d}J^{d}=L_{1d}Q_{1}^{d}+L_{2d}Q_{2}^{d}, \label{eq_max_id_dJ}%
\end{equation}
As $J^{d}$, the fields $L_{1,2d}\left(  F^{ad},\ast F^{ad},x^{a}\right)  $ do
not depend on derivatives of $F^{ad}$ or $\ast F^{ad}$.

We notice two extra off-shell identities%
\[
\nabla_{d}\nabla_{a}F^{ad}=0=\nabla_{d}\nabla_{a}\ast F^{ad},
\]
(both easy to verify). These expressions, in addition with (\ref{eq_max_id_dJ}%
),\ give the off-shell identities
\begin{align}
\nabla_{d}Q_{1}^{d}  &  =-L_{1d}Q_{1}^{d}-L_{2d}Q_{2}^{d}%
,\label{eq_max_id_dQ1}\\
\nabla_{d}Q_{2}^{d}  &  =0. \label{eq_max_id_dQ2}%
\end{align}
Multiplying by $-N$ and $N$, these equations can be rewritten as
\begin{equation}
\nabla_{d}\left(  \left[
\begin{array}
[c]{c}%
-NQ_{1}^{d}\\
NQ_{2}^{d}%
\end{array}
\right]  \right)  +\left[
\begin{array}
[c]{cc}%
\nabla_{d}\left(  N\right)  +NL_{1d} & NL_{2d}\\
0 & -\nabla_{d}\left(  N\right)
\end{array}
\right]  \left[
\begin{array}
[c]{c}%
Q_{1}^{d}\\
Q_{2}^{d}%
\end{array}
\right]  =0 \label{eq_max_id_mat_1}%
\end{equation}
These latter equations are the integrability conditions (\ref{eq_int_LE_1})
from which the evolution equations of the constraints are obtained as we
explain below.

The Maxwell equations are defined by $Q_{1,2}^{d}$ as%

\begin{align}
Q_{1}^{d}  &  =\nabla_{a}F^{ad}-J^{d}=0,\label{eq_max_or_1}\\
Q_{2}^{d}  &  =\nabla_{a}\ast F^{ad}=0. \label{eq_max_or_2}%
\end{align}
We will use the electric $E^{c}$ and magnetic $B_{e}$ fields as the variables
of the system. For this purpose, we begin by rewriting $F^{cq}$ in terms of
$E^{c}$ and $B_{e}$
\[
F^{cq}=\tilde{m}^{c}E^{q}-\tilde{m}^{q}E^{c}+\varepsilon^{cqde}\tilde{n}%
_{d}B_{e}.
\]
These fields are defined by
\[
E^{c}:=F^{cq}\tilde{n}_{q}\text{, \ \ }B_{d}:=\ast F_{ld}\tilde{m}^{l}%
\]
and they belong to the tangent of $\Sigma_{t}$, since
\[
\tilde{n}_{d}E^{d}=0=\tilde{m}^{e}B_{e}.
\]
On the other hand, the dual $\ast F_{lm}$ can be written as
\[
\ast F_{lm}=-\tilde{n}_{l}B_{m}+\tilde{n}_{m}B_{l}+\varepsilon_{cqlm}\tilde
{m}^{c}E^{q}.
\]

Following the same steps as in appendix \ref{App_coordenadas} and the
definitions (\ref{Eq_apendix_T_3}-\ref{Eq_apendix_T_8}), we rewrite
$Q_{1,2}^{d}$ as%
\begin{align}
Q_{1}^{d}  &  =\tilde{\eta}_{l}^{d}\frac{1}{N}\tilde{e}_{1}^{l}-\tilde{m}%
^{d}\psi_{1},\label{eq_Q1_E_B_1}\\
Q_{2}^{d}  &  =-\tilde{\eta}_{l}^{d}\frac{1}{N}\tilde{e}_{2}^{l}+\tilde{m}%
^{d}\psi_{2}, \label{eq_Q2_E_B_1}%
\end{align}
where
\begin{align}
\tilde{J}^{l}  &  :=\tilde{\eta}_{c}^{l}J^{c},\text{ \ \ \ \ }\tilde{J}%
^{0}=\tilde{n}_{c}J^{c},\label{eq_J_1}\\
\tilde{e}_{1}^{l}  &  :=%
\mathcal{L}%
_{p}E^{l}+\varepsilon^{clde}\tilde{n}_{d}D_{c}\left(  NB_{e}\right)
-NKE^{l}-N\tilde{J}^{l},\label{eq_e1_1}\\
\tilde{e}_{2}^{l}  &  :=%
\mathcal{L}%
_{p}B^{l}-\varepsilon^{clde}\tilde{n}_{d}D_{c}\left(  NE_{e}\right)
-NKB^{l},\label{eq_e1_2}\\
\psi_{1}  &  =D_{a}E^{a}-\tilde{J}^{0},\label{eq_phi1_1}\\
\psi_{2}  &  =D_{a}B^{a}. \label{eq_phi2_1}%
\end{align}
Here, the expressions (\ref{eq_J_1}-\ref{eq_phi2_1}) are quantities over
$\Sigma_{t}$ by definition, which means that $\tilde{\eta}_{l}^{d}\tilde
{e}_{1,2}^{l}=\tilde{e}_{1,2}^{d}$ and $\tilde{\eta}_{l}^{d}\tilde{J}%
^{l}=\tilde{J}^{d}$.

Notice that $\tilde{e}_{1,2}^{d}$ and $\psi_{1,2}$ are the evolution and the
constraints equations\ of the system respectively. $\psi_{1,2}$ act as
constraints since they have not derivatives in the $\partial_{t}$ direction.
We will see below that these constraints are preserved in the evolutions
$\tilde{e}_{1,2}^{l}=0$ when the initial data $\left.  \phi\right\vert
_{\Sigma_{0}}$ is chosen such that $\left.  \psi_{1,2}\right\vert _{\Sigma
_{0}}=0$.

Using the expressions (\ref{eq_Q1_E_B_1}, \ref{eq_Q2_E_B_1}, \ref{eq_e1_1}%
-\ref{eq_phi2_1}), we conclude that the principal symbol of the system is%
\begin{align}
\left[
\begin{array}
[c]{c}%
Q_{1}^{s}\\
Q_{2}^{s}%
\end{array}
\right]   &  \approx\mathfrak{N}_{~\beta}^{Bq}\nabla_{q}\phi^{\beta
},\nonumber\\
&  =\left[
\begin{array}
[c]{cc}%
2\delta_{b}^{[s}\tilde{m}^{q]} & \varepsilon_{~\ \ \ \ b}^{qsd}\tilde{n}_{d}\\
\varepsilon_{~\ \ \ \ b}^{qsd}\tilde{n}_{d} & -2\delta_{b}^{[s}\tilde{m}^{q]}%
\end{array}
\right]  \nabla_{q}\left[
\begin{array}
[c]{c}%
E^{b}\\
B^{b}%
\end{array}
\right]  . \label{eq_max_K_1}%
\end{align}
We are considering
\[%
\mathcal{L}%
_{p}E^{r}\approx\tilde{\eta}_{b}^{r}p^{q}\nabla_{q}E^{b},
\]
neglecting lower order terms.

Associated with this system, we have the Geroch fields (see eq. (\ref{eq_CK_1}%
))%
\begin{equation}
C_{B}^{\Delta z}=\left[
\begin{array}
[c]{cc}%
-N\delta_{s}^{z} & 0\\
0 & N\delta_{s}^{z}%
\end{array}
\right]  . \label{eq_max_Cz_1}%
\end{equation}
Contracting it with $n_{z}$, we obtain
\[
n_{z}C_{B}^{\Delta z}=C_{B}^{\Delta0}=\left[
\begin{array}
[c]{cc}%
\tilde{n}_{s} & 0\\
0 & -\tilde{n}_{s}%
\end{array}
\right]  .
\]
This last expression gives the constraints $\psi_{1,2}$ when it is contracted
with $\left[
\begin{array}
[c]{c}%
Q_{1}^{s}\\
Q_{2}^{s}%
\end{array}
\right]  $.

On the other hand, we choose the following reduction
\[
h_{B}^{\alpha}=\left[
\begin{array}
[c]{cc}%
N\tilde{\eta}_{s}^{l} & 0\\
0 & -N\tilde{\eta}_{s}^{l}%
\end{array}
\right]  .
\]
This reduction leads to the standard symmetric hyperbolic Maxwell evolution
equations $e_{1,2}^{l}$ when it is contracted with $\left[
\begin{array}
[c]{c}%
Q_{1}^{s}\\
Q_{2}^{s}%
\end{array}
\right]  $. Combining the last two expressions, we conclude that
\begin{equation}
\left[
\begin{array}
[c]{c}%
h_{B}^{\alpha}\\
C_{B}^{\Delta0}%
\end{array}
\right]  =\left[
\begin{array}
[c]{cc}%
N\tilde{\eta}_{s}^{l} & 0\\
0 & -N\tilde{\eta}_{s}^{l}\\
\tilde{n}_{s} & 0\\
0 & -\tilde{n}_{s}%
\end{array}
\right]  . \label{eq_max_h_C0_1}%
\end{equation}
Then by eqs. (\ref{eq_Q1_E_B_1}) and (\ref{eq_Q2_E_B_1}), we obtain
\begin{equation}
\left[
\begin{array}
[c]{c}%
h_{B}^{\alpha}\\
C_{B}^{\Delta0}%
\end{array}
\right]  E^{B}=\left[
\begin{array}
[c]{cc}%
N\tilde{\eta}_{s}^{l} & 0\\
0 & -N\tilde{\eta}_{s}^{l}\\
\tilde{n}_{s} & 0\\
0 & -\tilde{n}_{s}%
\end{array}
\right]  \left[
\begin{array}
[c]{c}%
Q_{1}^{s}\\
Q_{2}^{s}%
\end{array}
\right]  =\left[
\begin{array}
[c]{c}%
\tilde{e}_{1}^{l}\\
\tilde{e}_{2}^{l}\\
\psi_{1}\\
\psi_{2}%
\end{array}
\right]  . \label{eq_max_e_cons_1}%
\end{equation}

The inverse of $\left[
\begin{array}
[c]{c}%
h_{B}^{\alpha}\\
C_{B}^{\Delta0}%
\end{array}
\right]  $ has the following form
\begin{equation}
\left[
\begin{array}
[c]{cc}%
\mathfrak{N}_{~\alpha}^{A0} & h_{~\Delta}^{A}%
\end{array}
\right]  =\left[
\begin{array}
[c]{cccc}%
\frac{1}{N}\tilde{\eta}_{l}^{s} & 0 & -\tilde{m}^{s} & 0\\
0 & -\frac{1}{N}\tilde{\eta}_{l}^{s} & 0 & \tilde{m}^{s}%
\end{array}
\right]  , \label{eq_max_inv_1}%
\end{equation}
and, of course, satisfies that
\begin{align}
\left[
\begin{array}
[c]{cc}%
\mathfrak{N}_{~\alpha}^{A0} & h_{~\Delta}^{A}%
\end{array}
\right]  \left[
\begin{array}
[c]{c}%
h_{B}^{\alpha}\\
C_{B}^{\Delta0}%
\end{array}
\right]   &  =\delta_{B}^{A},\nonumber\\
\left[
\begin{array}
[c]{cccc}%
\frac{1}{N}\tilde{\eta}_{l}^{s} & 0 & -\tilde{m}^{s} & 0\\
0 & -\frac{1}{N}\tilde{\eta}_{l}^{s} & 0 & \tilde{m}^{s}%
\end{array}
\right]  \left[
\begin{array}
[c]{cc}%
N\tilde{\eta}_{s}^{l} & 0\\
0 & -N\tilde{\eta}_{s}^{l}\\
\tilde{n}_{s} & 0\\
0 & -\tilde{n}_{s}%
\end{array}
\right]   &  =\left[
\begin{array}
[c]{cc}%
\delta_{s}^{l} & 0\\
0 & \delta_{s}^{l}%
\end{array}
\right]  . \label{eq_p_p_1_1}%
\end{align}

Using this expression and the equations (\ref{eq_p_p_1_1}) and
(\ref{eq_max_e_cons_1}), we conclude that
\begin{align*}
\nabla_{d}\left(  C_{B}^{\Delta d}E^{B}\right)   &  =\nabla_{d}\left(  \left(
C_{A}^{\Delta d}\left[
\begin{array}
[c]{cc}%
\mathfrak{N}_{~\alpha}^{A0} & h_{~\Gamma}^{A}%
\end{array}
\right]  \right)  \left(  \left[
\begin{array}
[c]{c}%
h_{B}^{\alpha}\\
C_{B}^{\Gamma0}%
\end{array}
\right]  E^{B}\right)  \right) \\
&  \nabla_{d}\left(  \left[
\begin{array}
[c]{cc}%
-N\delta_{s}^{d} & 0\\
0 & N\delta_{s}^{d}%
\end{array}
\right]  \left[
\begin{array}
[c]{c}%
Q_{1}^{s}\\
Q_{2}^{s}%
\end{array}
\right]  \right) \\
&  =\nabla_{d}\left(  \left[
\begin{array}
[c]{cccc}%
-\tilde{\eta}_{l}^{s} & 0 & N\tilde{m}^{s} & 0\\
0 & -\tilde{\eta}_{l}^{s} & 0 & N\tilde{m}^{s}%
\end{array}
\right]  \left[
\begin{array}
[c]{c}%
\tilde{e}_{1}^{l}\\
\tilde{e}_{2}^{l}\\
\psi_{1}\\
\psi_{2}%
\end{array}
\right]  \right)  ,
\end{align*}
which allows to rewrite the identities (\ref{eq_max_id_mat_1}) as
\begin{align*}
0  &  =%
\mathcal{L}%
_{p}\left(  \left[
\begin{array}
[c]{c}%
\psi_{1}\\
\psi_{2}%
\end{array}
\right]  \right)  +\left[
\begin{array}
[c]{cc}%
-N-N\tilde{m}^{d}L_{1d} & N\tilde{m}^{d}L_{2d}\\
0 & -NK
\end{array}
\right]  \left[
\begin{array}
[c]{c}%
\psi_{1}\\
\psi_{2}%
\end{array}
\right] \\
&  +\left[
\begin{array}
[c]{c}%
-N\nabla_{d}\left(  \frac{1}{N}e_{1}^{d}\right)  +L_{1d}\tilde{e}_{1}%
^{d}-L_{2d}\tilde{e}_{2}^{d}\\
-N\nabla_{d}\left(  \frac{1}{N}e_{2}^{d}\right)
\end{array}
\right]  .
\end{align*}
On-shell (i.e., when $e_{1,2}^{d}=0$), these equations lead to the subsidiary
system
\begin{equation}
0=%
\mathcal{L}%
_{p}\left(  \left[
\begin{array}
[c]{c}%
\psi_{1}\\
\psi_{2}%
\end{array}
\right]  \right)  +\left[
\begin{array}
[c]{cc}%
-N-N\tilde{m}^{d}L_{1d} & N\tilde{m}^{d}L_{2d}\\
0 & -NK
\end{array}
\right]  \left[
\begin{array}
[c]{c}%
\psi_{1}\\
\psi_{2}%
\end{array}
\right]  . \label{Eq_sub_max_1}%
\end{equation}

Clearly (since there are no spatial derivatives of $\psi_{1,2}$) this system
is strongly hyperbolic, so it has a unique solution for a given initial data.
Since the initial data is chosen such that $\left.  \psi_{1,2}\right\vert
_{\Sigma_{0}}=0$ and $\psi_{1,2}=0$ is a solution of the subsidiary system, we
conclude (by the uniqueness of the solutions) that the constraints are
preserved during the evolutions $e_{1,2}^{d}=0$.

We now consider the relationship between the principal symbols of the
evolution equations and the principal symbols of the constraints equations.

Using equations (\ref{eq_max_K_1}), (\ref{eq_max_h_C0_1}), (\ref{eq_max_Cz_1})
and (\ref{eq_max_inv_1}), we obtain the matrices $\left[
\begin{array}
[c]{c}%
h_{B}^{\alpha}\\
C_{B}^{\Gamma0}%
\end{array}
\right]  \mathfrak{N}_{~\beta}^{Bq}$ and $l_{z}C_{A}^{\Delta z}\left[
\begin{array}
[c]{cc}%
\mathfrak{N}_{~\alpha}^{A0} & h_{~\Gamma}^{A}%
\end{array}
\right]  $, they are
\begin{equation}
\left[
\begin{array}
[c]{c}%
h_{B}^{\alpha}\\
C_{B}^{\Gamma0}%
\end{array}
\right]  \mathfrak{N}_{~\beta}^{Bq}l_{q}\phi^{\beta}=\left[
\begin{array}
[c]{cc}%
\tilde{\eta}_{b}^{l}p^{q} & N\varepsilon_{~\ \ \ \ b}^{qld}\tilde{n}_{d}\\
-N\varepsilon_{~\ \ \ \ b}^{qld}\tilde{n}_{d} & \tilde{\eta}_{b}^{l}p^{q}\\
\tilde{\eta}_{b}^{q} & 0\\
0 & \tilde{\eta}_{b}^{q}%
\end{array}
\right]  l_{q}\left[
\begin{array}
[c]{c}%
E^{b}\\
B^{b}%
\end{array}
\right]  , \label{Eq_max_1}%
\end{equation}
and%
\[
l_{z}C_{A}^{\Delta z}\left[
\begin{array}
[c]{cc}%
\mathfrak{N}_{~\alpha}^{A0} & h_{~\Gamma}^{A}%
\end{array}
\right]  =l_{z}\left[
\begin{array}
[c]{cccc}%
-\tilde{\eta}_{l}^{z} & 0 & N\tilde{m}^{z} & 0\\
0 & -\tilde{\eta}_{l}^{z} & 0 & N\tilde{m}^{z}%
\end{array}
\right]  .
\]
Recalling that
\[
l_{z}C_{A}^{\Delta z}\left[
\begin{array}
[c]{cc}%
\mathfrak{N}_{~\alpha}^{A0} & h_{~\Gamma}^{A}%
\end{array}
\right]  \left[
\begin{array}
[c]{c}%
h_{B}^{\alpha}\\
C_{B}^{\Gamma0}%
\end{array}
\right]  \mathfrak{N}_{~\beta}^{Bq}l_{q}=0,
\]
we concluded
\begin{align}
&  l_{z}\left[
\begin{array}
[c]{cc}%
-\tilde{\eta}_{l}^{z} & 0\\
0 & -\tilde{\eta}_{l}^{z}%
\end{array}
\right]  \left[
\begin{array}
[c]{cc}%
\tilde{\eta}_{b}^{l}p^{q} & N\varepsilon_{~\ \ \ \ b}^{qld}\tilde{n}_{d}\\
-N\varepsilon_{~\ \ \ \ b}^{qld}\tilde{n}_{d} & \tilde{\eta}_{b}^{l}p^{q}%
\end{array}
\right]  l_{q}\nonumber\\
&  =l_{z}\left[
\begin{array}
[c]{cc}%
p^{z} & 0\\
0 & p^{z}%
\end{array}
\right]  \left[
\begin{array}
[c]{cc}%
-\tilde{\eta}_{b}^{q} & 0\\
0 & -\tilde{\eta}_{b}^{q}%
\end{array}
\right]  l_{q}. \label{eq_mas_C_A_P_C_1}%
\end{align}
This expression is exactly the equation (\ref{Eq_simbs_1}), where%
\[
h_{B}^{\alpha}\mathfrak{N}_{~\beta}^{Bq}l_{q}=\left[
\begin{array}
[c]{cc}%
\tilde{\eta}_{b}^{l}p^{q} & N\varepsilon_{~\ \ \ \ b}^{qld}\tilde{n}_{d}\\
-N\varepsilon_{~\ \ \ \ b}^{qld}\tilde{n}_{d} & \tilde{\eta}_{b}^{l}p^{q}%
\end{array}
\right]  l_{q}%
\]
is the principal symbol of $\left[
\begin{array}
[c]{c}%
e_{1}^{l}\\
e_{2}^{l}%
\end{array}
\right]  $,
\[
l_{z}C_{A}^{\Delta z}h_{~\Gamma}^{A}=l_{z}\left[
\begin{array}
[c]{cc}%
p^{z} & 0\\
0 & p^{z}%
\end{array}
\right]
\]
is the principal symbol of the subsidiary system and%
\[
C_{B}^{\Gamma0}\mathfrak{N}_{~\beta}^{Bq}l_{q}=\left[
\begin{array}
[c]{cc}%
\tilde{\eta}_{b}^{q} & 0\\
0 & \tilde{\eta}_{b}^{q}%
\end{array}
\right]  l_{q}.
\]

Let us now study the characteristic structures of the system and the
subsidiary system. We will study the Kronecker decomposition of the pencils
$\left[
\begin{array}
[c]{c}%
h_{~A}^{\alpha}\mathfrak{N}_{~\alpha}^{Aq}\\
C_{A}^{\Gamma0}\mathfrak{N}_{~\alpha}^{Aq}%
\end{array}
\right]  l_{q}$ and $C_{A}^{\Delta d}h_{~\Gamma}^{A}l_{d}$ with $l_{q}%
=-\lambda n_{a}+k_{a}$ and $k_{a}\left(  \partial_{t}\right)  ^{a}=0$. In
their pullback version to $\Sigma_{t}$, these pencils are
\begin{align}
&  \left[
\begin{array}
[c]{c}%
h_{~A}^{\alpha}\mathfrak{N}_{~\alpha}^{Aq}\\
C_{A}^{\Gamma0}\mathfrak{N}_{~\alpha}^{Aq}%
\end{array}
\right]  l_{q}\phi^{\alpha}\nonumber\\
&  =\left[
\begin{array}
[c]{cc}%
\tilde{\eta}_{b}^{l}\left(  -\lambda-\left(  \beta.k\right)  \right)  &
N\varepsilon_{~\ \ \ \ b}^{dl}k_{d}\\
-N\varepsilon_{~\ \ \ \ b}^{dl}k_{d} & \left(  -\lambda-\left(  \beta
.k\right)  \right)  \tilde{\eta}_{b}^{l}\\
k_{b} & 0\\
0 & k_{b}%
\end{array}
\right]  \left[
\begin{array}
[c]{c}%
E^{b}\\
B^{b}%
\end{array}
\right]  , \label{Eq_Max_p_1}%
\end{align}%
\begin{equation}
C_{A}^{\Delta q}h_{~\Gamma}^{A}l_{q}\psi^{\Gamma}=\left[
\begin{array}
[c]{cc}%
\left(  -\lambda-\left(  \beta.k\right)  \right)  & 0\\
0 & \left(  -\lambda-\left(  \beta.k\right)  \right)
\end{array}
\right]  \left[
\begin{array}
[c]{c}%
\psi_{1}\\
\psi_{2}%
\end{array}
\right]  . \label{Eq_Max_p_2}%
\end{equation}
Where $\varepsilon^{qwy}:=\tilde{n}_{d}\varepsilon^{dqwy}$ is the Levi-Civita
tensor over $\Sigma_{t}$; and all the lowercase indices $l,b,d$ (and any other
lowercase indices that appear until the end of this section) run from $1$ to
$3$. We also note the size of the matrices $\left[
\begin{array}
[c]{c}%
h_{~A}^{\alpha}\mathfrak{N}_{~\alpha}^{Aq}\\
C_{A}^{\Gamma0}\mathfrak{N}_{~\alpha}^{Aq}%
\end{array}
\right]  l_{q}$ $\in%
\mathbb{R}
^{8\times6}$, $C_{A}^{\Delta q}h_{~\Gamma}^{A}l_{q}\in%
\mathbb{R}
^{2\times2}$, $h_{~A}^{\alpha}\mathfrak{N}_{~\alpha}^{Aq}l_{q}\in%
\mathbb{R}
^{6\times6}$ and $C_{A}^{\Gamma0}\mathfrak{N}_{~\alpha}^{Aq}l_{q}\in%
\mathbb{R}
^{2\times6}$\ in this pullbacked version.

Consider the matrix $\left(  C_{B}^{\Delta0}\mathfrak{N}_{~\alpha}%
^{Bq}\right)  l_{q}$ whose expression,
\[
C_{B}^{\Gamma0}\mathfrak{N}_{~\beta}^{Bq}l_{q}\phi^{\beta}=\left[
\begin{array}
[c]{cc}%
k_{b} & 0\\
0 & k_{b}%
\end{array}
\right]  \left[
\begin{array}
[c]{c}%
E^{b}\\
B^{b}%
\end{array}
\right]  .
\]
is obtained from (\ref{Eq_Max_p_1}). For each $k_{q}$, we have%
\begin{align}
\dim\left(  left\_\ker\left(  C_{B}^{\Gamma0}\mathfrak{N}_{~\beta}^{Bq}%
l_{q}\right)  \right)   &  =0,\label{eq_max_lf_1}\\
rank\left(  C_{B}^{\Gamma0}\mathfrak{N}_{~\beta}^{Bq}l_{q}\right)   &
=2,\label{eq_max_r_1}\\
\dim\left(  right\_\ker\left(  C_{B}^{\Delta0}\mathfrak{N}_{~\alpha}^{Bq}%
k_{q}\right)  \right)   &  =4.
\end{align}
This means that there are not Geroch fields $M_{\Gamma}^{\tilde{\Delta}z}$
such that $l_{z}M_{\Gamma}^{\tilde{\Delta}z}C_{B}^{\Gamma0}\mathfrak{N}%
_{~\beta}^{Bq}l_{q}=0$ and therefore the condition v) of the theorem
\ref{Theorem_coef_const_2} is satisfied.

This result allows us to give the Kronecker structure associated with the
Maxwell equation, i.e.
\[
2\times J_{1}\left(  N\sqrt{k.k}-\beta.k\right)  ,2\times J_{1}\left(
-N\sqrt{k.k}-\beta.k\right)  ,2\times L_{1}^{T}.
\]
is the Kronecker structure of the pencil (\ref{Eq_Max_p_1}).

The $2\times L_{1}^{T}$ blocks are justified by (\ref{eq_max_lf_1}) and
(\ref{eq_max_r_1}) as explained in subsubsection \ref{Kronecker_decomposition}%
. The Jordan part is justified by giving explicitly the generalized
eigenvectors
\[
\left(  \delta\phi_{\lambda_{1}}^{1}\right)  ^{\beta}=\left[
\begin{array}
[c]{c}%
\frac{1}{\sqrt{\left(  v.v\right)  }}v^{b}\\
\frac{1}{\sqrt{\left(  w.w\right)  }}w^{b}%
\end{array}
\right]  ,\text{ \ }\left(  \delta\phi_{\lambda_{1}}^{2}\right)  ^{\beta
}=\left[
\begin{array}
[c]{c}%
\frac{1}{\sqrt{\left(  w.w\right)  }}w^{b}\\
-\frac{1}{\sqrt{\left(  v.v\right)  }}v^{b}%
\end{array}
\right]
\]%
\[
\left(  \delta\phi_{\lambda_{2}}^{1}\right)  ^{\beta}=\left[
\begin{array}
[c]{c}%
\frac{1}{\sqrt{\left(  v.v\right)  }}v^{b}\\
-\frac{1}{\sqrt{\left(  w.w\right)  }}w^{b}%
\end{array}
\right]  ,\text{ \ }\left(  \delta\phi_{\lambda_{2}}^{2}\right)  ^{\beta
}=\left[
\begin{array}
[c]{c}%
\frac{1}{\sqrt{\left(  w.w\right)  }}w^{b}\\
\frac{1}{\sqrt{\left(  v.v\right)  }}v^{b}%
\end{array}
\right]
\]
where $v^{l}$ and $w^{l}$ are linearly independent vectors, defined by
\begin{align*}
2v^{[l}w^{d]}  &  =\varepsilon^{qlb}k_{q}\\
v.k  &  =v.w=w.k=0.
\end{align*}
They are associated to the generalized eigenvalues $\lambda_{1}=N\sqrt
{k.k}-\beta.k$ and $\lambda_{2}=-N\sqrt{k.k}-\beta.k$. These eigenvalues were
called "the physical" eigenvalues in subsection
\ref{Proof_Theorem_coef_const_2}.

Now, let us study the characteristic structure of $h_{~A}^{\alpha}%
\mathfrak{N}_{~\alpha}^{Aq}l_{q}$, whose explicit expression is
\[
h_{~A}^{\alpha}\mathfrak{N}_{~\alpha}^{Aq}l_{q}\phi^{\alpha}=\left[
\begin{array}
[c]{cc}%
\tilde{\eta}_{b}^{l}\left(  -\lambda-\left(  \beta.k\right)  \right)  &
N\varepsilon_{~\ \ \ \ b}^{dl}k_{d}\\
-N\varepsilon_{~\ \ \ \ b}^{dl}k_{d} & \left(  -\lambda-\left(  \beta
.k\right)  \right)  \tilde{\eta}_{b}^{l}%
\end{array}
\right]  \left[
\begin{array}
[c]{c}%
E^{b}\\
B^{b}%
\end{array}
\right]  .
\]
Since $\left(  \lambda_{1,2},\left(  \delta\phi_{\lambda_{1,2}}\right)
^{\beta}\right)  $ are the generalized eigenvalues and eigenvectors of
$\left[
\begin{array}
[c]{c}%
h_{~A}^{\alpha}\mathfrak{N}_{~\alpha}^{Aq}\\
C_{A}^{\Gamma0}\mathfrak{N}_{~\alpha}^{Aq}%
\end{array}
\right]  l_{q}$,\ \ they are also the eigenvalues and eigenvectors of
$h_{~A}^{\alpha}\mathfrak{N}_{~\alpha}^{Aq}l_{q}$. Moreover, $h_{~A}^{\alpha
}\mathfrak{N}_{~\alpha}^{Aq}l_{q}$ has $2$ additional eigenvalues $\pi
_{1,2}\left(  k\right)  =-\beta.k$, these were called "the constraints 1"
eigenvalues in subsection \ref{Proof_Theorem_coef_const_2}. \ These
eigenvalues have the associated eigenvectors
\[
\left(  \delta\phi_{\pi_{1}}\right)  ^{\beta}=\left[
\begin{array}
[c]{c}%
k^{b}\\
0
\end{array}
\right]  \text{ \ \ }\left(  \delta\phi_{\pi_{2}}\right)  ^{\beta}=\left[
\begin{array}
[c]{c}%
0\\
k^{b}%
\end{array}
\right]  .
\]

To conclude this subsection, let us now study the Kronecker decomposition of
the pencil $C_{A}^{\Delta q}h_{~\Gamma}^{A}l_{q}$ \ (see eq. (\ref{Eq_Max_p_2}%
)). This is a square pencil and it is associated to the constraints as can be
seen from (\ref{Eq_sub_max_1}). Its Kronecker decomposition is given by
\[
2\times J_{1}\left(  -\beta.k\right)  .
\]

As it was shown in subsubsection \ref{Kro_sub}, the generalized eigenvectors
of this pencil are obtained by projecting $\left\{  \left(  \delta\phi
_{\pi_{1}}\right)  ^{\beta},\left(  \delta\phi_{\pi_{2}}\right)  ^{\beta
}\right\}  $ with $C_{B}^{\Gamma0}\mathfrak{N}_{~\beta}^{Bq}k_{q}$, that is,
\begin{align*}
\left[  \delta\psi_{\pi_{1}}^{\Gamma},\delta\psi_{\pi_{2}}^{\Gamma}\right]
&  :=C_{B}^{\Gamma0}\mathfrak{N}_{~\beta}^{Bq}l_{q}\left[  \left(  \delta
\phi_{1}\right)  ^{\beta},\left(  \delta\phi_{\pi_{2}}\right)  ^{\beta
}\right]  ,\\
&  =\left[
\begin{array}
[c]{cc}%
k_{b} & 0\\
0 & k_{b}%
\end{array}
\right]  \left[
\begin{array}
[c]{cc}%
k^{b} & 0\\
0 & k^{b}%
\end{array}
\right]  ,\\
&  =\left[
\begin{array}
[c]{cc}%
\left(  k.k\right)  & 0\\
0 & \left(  k.k\right)
\end{array}
\right]  .
\end{align*}
Where the eigenvectors are the columns of this matrix and they are associated
to "the constraints 1" eigenvalues $\pi_{1,2,3}\left(  k\right)  =-\beta.k$.
This result follows from equation (\ref{Eq_simbs_1}) and can be easily
checked, since $C_{A}^{\Delta q}h_{~\Gamma}^{A}l_{q}=0$ when $\lambda
=-\beta.k$.

\subsection{Wave equation \label{S_wave_equation}}

Consider the wave equation%
\begin{equation}
g^{ab}\nabla_{a}\nabla_{b}\phi=0. \label{eq_wave_1}%
\end{equation}
We lead this equation to first order in derivatives. The wave equation in
first-order is
\begin{equation}
E=0,\text{ \ \ }E_{b}=0,\text{ \ \ }E_{ab}=0. \label{Eq_W_E_1}%
\end{equation}
where we have defined
\[
u_{b}:=\nabla_{b}\phi,
\]
and
\begin{align*}
E  &  :=g^{ab}\nabla_{a}u_{b},\\
E_{b}  &  :=\nabla_{b}\phi-u_{b},\\
E_{ab}  &  :=\nabla_{\lbrack a}u_{b]}.
\end{align*}

Notice that $E_{ab}$ is obtained from $E_{b}$ by taking an antisymmetric
derivative, i.e. $\nabla_{\lbrack a}E_{b]}=-E_{ab}$. Additionally, there is
another identity for $E_{ab}$, this is $\nabla_{\lbrack c}E_{ab]}=0$
\footnote{To show this result, you need to use the first Bianchi identity.}.
Thus, the off-shell identities of the system are%

\begin{align}
0  &  =\nabla_{f}\left(  \delta_{a}^{[f}\delta_{b}^{g]}E_{g}\right)
+E_{ab},\label{Eq_W_id_1}\\
0  &  =\nabla_{c}\left(  \delta_{f}^{[c}\delta_{g}^{a}\delta_{h}^{b]}%
E_{ab}\right)  . \label{Eq_W_id_2}%
\end{align}
The following four projections of these equations
\[
0=\left[
\begin{array}
[c]{c}%
2N\tilde{m}^{a}\tilde{\eta}_{r}^{b}\left(  \nabla_{f}\left(  \delta_{a}%
^{[f}\delta_{b}^{g]}E_{g}\right)  +E_{ab}\right) \\
3N\tilde{m}^{f}\tilde{\eta}_{g}^{d}\tilde{\eta}_{h}^{e}\nabla_{c}\left(
\delta_{f}^{[c}\delta_{d}^{a}\delta_{e}^{b]}E_{ab}\right) \\
\tilde{\eta}_{s}^{a}\tilde{\eta}_{r}^{b}\nabla_{f}\left(  \delta_{a}%
^{[f}\delta_{b}^{g]}E_{g}\right)  +E_{ab}\\
\tilde{n}_{d}\varepsilon^{dfgh}\nabla_{c}\left(  \delta_{f}^{[c}\delta_{g}%
^{a}\delta_{h}^{b]}E_{ab}\right)
\end{array}
\right]  ,
\]
can be rewritten as follows
\begin{align}
0  &  =\nabla_{z}\left(  \left[
\begin{array}
[c]{ccc}%
0 & 2N\tilde{m}^{[z}\tilde{\eta}_{r}^{g]} & 0\\
0 & 0 & 3N\tilde{m}^{[z}\tilde{\eta}_{g}^{a}\tilde{\eta}_{h}^{b]}\\
0 & \tilde{\eta}_{[s}^{z}\tilde{\eta}_{r]}^{g} & 0\\
0 & 0 & \tilde{n}_{d}\varepsilon^{dzab}%
\end{array}
\right]  \left[
\begin{array}
[c]{c}%
E\\
E_{g}\\
E_{ab}%
\end{array}
\right]  \right) \nonumber\\
&  +\left[
\begin{array}
[c]{ccc}%
0 & -\nabla_{z}\left(  2N\tilde{m}^{[z}\tilde{\eta}_{r}^{g]}\right)  &
2N\tilde{m}^{a}\tilde{\eta}_{r}^{b}\\
0 & 0 & -\nabla_{z}\left(  3N\tilde{m}^{[z}\tilde{\eta}_{g}^{a}\tilde{\eta
}_{h}^{b]}\right) \\
0 & -\nabla_{z}\left(  \tilde{\eta}_{s}^{[z}\tilde{\eta}_{r}^{g]}\right)  &
\tilde{\eta}_{s}^{[a}\tilde{\eta}_{r}^{b]}\\
0 & 0 & -\nabla_{z}\left(  \tilde{n}_{d}\varepsilon^{dzab}\right)
\end{array}
\right]  \left[
\begin{array}
[c]{c}%
E\\
E_{g}\\
E_{ab}%
\end{array}
\right]  . \label{eq_wav_id_dCE_1}%
\end{align}

As we will show, these expressions are exactly the evolution equations of the
constraints and the constraints of the constraints of the system. The latter
is associated to the Geroch fields $M_{\Gamma}^{\tilde{\Delta}z}$.

We now introduce new variables
\begin{align*}
\tilde{u}^{0}  &  :=\tilde{n}_{b}u^{b},\\
\tilde{u}_{d}  &  :=\tilde{\eta}_{db}u^{b},
\end{align*}
where $\tilde{u}_{d}$ results from projecting $u^{b}$ onto $\Sigma_{t}$. Since
we are interested in describing the system with the variables $\left(
\tilde{u}^{0},\text{ }\phi,\text{ }\tilde{u}_{w}\right)  $, we rewrite the
equations (\ref{Eq_W_E_1}) as follows
\begin{align}
E  &  =-\frac{1}{N}\tilde{e}_{1},\label{eq_wav_E_1}\\
E_{c}  &  =-\frac{1}{N}\tilde{n}_{c}\tilde{e}_{2}+\psi_{1c},\label{eq_wav_E_2}%
\\
E_{ac}  &  =\frac{1}{N}\tilde{e}_{3[a}\tilde{n}_{c]}+\psi_{2ac},
\label{eq_wav_E_3}%
\end{align}
where
\begin{align}
\tilde{e}_{1}  &  :=%
\mathcal{L}%
_{p}\tilde{u}^{0}-ND_{d}\tilde{u}^{d}-N\left(  \tilde{u}^{w}S_{w}\right)
-N\tilde{u}^{0}K,\label{Eq_W_sis_1}\\
\tilde{e}_{2}  &  :=%
\mathcal{L}%
_{p}\phi-N\tilde{u}^{0},\label{Eq_W_sis_2}\\
\tilde{e}_{3a}  &  :=%
\mathcal{L}%
_{p}\tilde{u}_{a}-ND_{a}\tilde{u}^{0}-N\left(  \tilde{u}^{r}K_{ar}+\tilde
{u}^{0}S_{a}\right)  -N\tilde{u}^{f}K_{fa},\label{Eq_W_sis_3}\\
\psi_{1c}  &  :=D_{f}\phi-\tilde{u}_{f},\label{Eq_W_sis_4}\\
\psi_{2ac}  &  :=D_{[a}\tilde{u}_{c]}. \label{Eq_W_sis_5}%
\end{align}

Notice that (\ref{Eq_W_sis_1}-\ref{Eq_W_sis_5}) are projected onto $\Sigma
_{t}$, where $\tilde{e}_{1},$ $\tilde{e}_{2},$ $\tilde{e}_{3a}$ are the
evolution equations for our variables and $\psi_{1c},$ $\psi_{2ac}$ are the
constraints of the system.

Using the expressions (\ref{eq_wav_E_1}), (\ref{eq_wav_E_2}) and
(\ref{eq_wav_E_3}), we obtain the principal symbol of the system
\begin{align*}
\left[
\begin{array}
[c]{c}%
E\\
E_{c}\\
E_{ac}%
\end{array}
\right]   &  \approx\mathfrak{N}_{~\beta}^{Bq}\nabla_{q}\phi^{\beta}\\
&  =\left[
\begin{array}
[c]{ccc}%
-\tilde{m}^{q} & 0 & \tilde{\eta}^{qw}\\
0 & \delta_{c}^{q} & 0\\
-\tilde{n}_{[c}\tilde{\eta}_{a]}^{q} & 0 & -\delta_{\lbrack c}^{q}\tilde{\eta
}_{a]}^{w}%
\end{array}
\right]  \nabla_{q}\left[
\begin{array}
[c]{c}%
\tilde{u}^{0}\\
\phi\\
\tilde{u}_{w}%
\end{array}
\right]
\end{align*}
and the Geroch fields $C_{A}^{\Gamma z}$ and $M_{A}^{\tilde{\Delta}z}$
\begin{equation}
\left[
\begin{array}
[c]{c}%
C_{A}^{\Gamma z}\\
M_{A}^{\tilde{\Delta}z}%
\end{array}
\right]  E^{A}=\left[
\begin{array}
[c]{ccc}%
0 & 2N\tilde{m}^{[z}\tilde{\eta}_{r}^{g]} & 0\\
0 & 0 & 3N\tilde{m}^{[z}\tilde{\eta}_{g}^{a}\tilde{\eta}_{h}^{b]}\\
0 & \tilde{\eta}_{[s}^{z}\tilde{\eta}_{r]}^{g} & 0\\
0 & 0 & \tilde{n}_{d}\varepsilon^{dzab}%
\end{array}
\right]  \left[
\begin{array}
[c]{c}%
E\\
E_{g}\\
E_{ab}%
\end{array}
\right]  . \label{Eq_W_C_M_1}%
\end{equation}
The first two lines are associated with $C_{A}^{\Gamma z}$ and the next two
with $M_{A}^{\tilde{\Delta}z}$. Notice that
\[
\left[
\begin{array}
[c]{c}%
C_{A}^{\Gamma z}\\
M_{A}^{\tilde{\Delta}z}%
\end{array}
\right]  n_{z}=\left[
\begin{array}
[c]{ccc}%
0 & \tilde{\eta}_{w}^{c} & 0\\
0 & 0 & \tilde{\eta}_{w}^{[a}\tilde{\eta}_{y}^{c]}\\
0 & 0 & 0\\
0 & 0 & 0
\end{array}
\right]
\]
where $M_{A}^{\tilde{\Delta}0}:=M_{A}^{\tilde{\Delta}z}n_{z}=0$ and
\[
C_{A}^{\Gamma0}:=C_{A}^{\Gamma z}n_{z}=\left[
\begin{array}
[c]{ccc}%
0 & \tilde{\eta}_{w}^{c} & 0\\
0 & 0 & \tilde{\eta}_{w}^{[a}\tilde{\eta}_{y}^{c]}%
\end{array}
\right]  .
\]
The constraints\ $\psi_{1c}$ and $\psi_{2ac}$ are obtained by contracting this
last expression with $E^{A}$.

On the other hand, we choose the following reduction $h_{~A}^{\alpha}$, such
that%
\[
h_{~A}^{\alpha}E^{A}=\left[
\begin{array}
[c]{ccc}%
-N & 0 & 0\\
0 & N\tilde{m}^{c} & 0\\
0 & 0 & -2N\tilde{\eta}_{q}^{[a}\tilde{m}^{c]}%
\end{array}
\right]  \left[
\begin{array}
[c]{c}%
E\\
E_{c}\\
E_{ac}%
\end{array}
\right]  .
\]
which gives the (symmetric hyperbolic) evolution equations $\tilde{e}_{1,2,3}$.

Combining the last two results, we have that
\begin{equation}
\left[
\begin{array}
[c]{c}%
h_{~A}^{\alpha}\\
C_{A}^{\Gamma0}%
\end{array}
\right]  E^{A}=\left[
\begin{array}
[c]{ccc}%
-N & 0 & 0\\
0 & N\tilde{m}^{c} & 0\\
0 & 0 & -2N\tilde{\eta}_{q}^{[a}\tilde{m}^{c]}\\
0 & \tilde{\eta}_{w}^{c} & 0\\
0 & 0 & \tilde{\eta}_{w}^{[a}\tilde{\eta}_{y}^{c]}%
\end{array}
\right]  \left[
\begin{array}
[c]{c}%
E\\
E_{c}\\
E_{ac}%
\end{array}
\right]  =\left[
\begin{array}
[c]{c}%
\tilde{e}_{1}\\
\tilde{e}_{2}\\
\tilde{e}_{3q}\\
\psi_{1w}\\
\psi_{2wy}%
\end{array}
\right]  . \label{eq_wav_e_c_1}%
\end{equation}
The principal symbol $\left[
\begin{array}
[c]{c}%
h_{~A}^{\alpha}\\
C_{A}^{\Gamma0}%
\end{array}
\right]  \mathfrak{N}_{~\alpha}^{Aq}l_{q}$ of the latter expression is
\begin{equation}
\left[
\begin{array}
[c]{c}%
h_{~A}^{\alpha}\\
C_{A}^{\Gamma0}%
\end{array}
\right]  \mathfrak{N}_{~\alpha}^{Aq}\nabla_{q}\left[
\begin{array}
[c]{c}%
\tilde{u}^{0}\\
\phi\\
\tilde{u}_{w}%
\end{array}
\right]  =\left[
\begin{array}
[c]{ccc}%
N\tilde{m}^{q} & 0 & -N\tilde{\eta}^{qw}\\
0 & N\tilde{m}^{q} & 0\\
-N\tilde{\eta}_{s}^{q} & 0 & N\tilde{\eta}_{s}^{w}\tilde{m}^{q}\\
0 & \tilde{\eta}_{s}^{q} & 0\\
0 & 0 & \tilde{\eta}_{s}^{[q}\tilde{\eta}_{y}^{w]}%
\end{array}
\right]  \nabla_{q}\left[
\begin{array}
[c]{c}%
\tilde{u}^{0}\\
\phi\\
\tilde{u}_{w}%
\end{array}
\right]  . \label{Eq_W_pencil_1}%
\end{equation}

On the other hand, the inverse of $\left[
\begin{array}
[c]{c}%
h_{B}^{\alpha}\\
C_{B}^{\Delta0}%
\end{array}
\right]  $ has the following form%
\[
\left[
\begin{array}
[c]{cc}%
\mathfrak{N}_{~\alpha}^{B0} & h_{~\Gamma}^{B}%
\end{array}
\right]  =\left[
\begin{array}
[c]{ccccc}%
-\frac{1}{N} & 0 & 0 & 0 & 0\\
0 & -\frac{1}{N}\tilde{n}_{c} & 0 & \tilde{\eta}_{c}^{w} & 0\\
0 & 0 & \frac{1}{N}\tilde{\eta}_{[a}^{q}\tilde{n}_{c]} & 0 & \tilde{\eta}%
_{[a}^{w}\tilde{\eta}_{c]}^{y}%
\end{array}
\right]  ,
\]
and satisfy that
\[
\left[
\begin{array}
[c]{cc}%
\mathfrak{N}_{~\alpha}^{A0} & h_{~\Delta}^{A}%
\end{array}
\right]  \left[
\begin{array}
[c]{c}%
h_{B}^{\alpha}\\
C_{B}^{\Delta0}%
\end{array}
\right]  =\delta_{B}^{A},
\]
i.e.%
\begin{align}
&  \left[
\begin{array}
[c]{ccccc}%
-\frac{1}{N} & 0 & 0 & 0 & 0\\
0 & -\frac{1}{N}\tilde{n}_{r} & 0 & \tilde{\eta}_{r}^{w} & 0\\
0 & 0 & \frac{1}{N}\tilde{\eta}_{[r}^{q}\tilde{n}_{s]} & 0 & \tilde{\eta}%
_{[r}^{w}\tilde{\eta}_{s]}^{y}%
\end{array}
\right]  \left[
\begin{array}
[c]{ccc}%
-N & 0 & 0\\
0 & N\tilde{m}^{c} & 0\\
0 & 0 & -2N\tilde{\eta}_{q}^{[a}\tilde{m}^{c]}\\
0 & \tilde{\eta}_{w}^{c} & 0\\
0 & 0 & \tilde{\eta}_{w}^{[a}\tilde{\eta}_{y}^{c]}%
\end{array}
\right] \nonumber\\
&  =\left[
\begin{array}
[c]{ccc}%
1 & 0 & 0\\
0 & \delta_{r}^{c} & 0\\
0 & 0 & \delta_{r}^{[a}\delta_{s}^{c]}%
\end{array}
\right]  \label{eq_wav_p_p_1_1}%
\end{align}

Using the above expressions, we obtain
\begin{align*}
&  \left[
\begin{array}
[c]{c}%
C_{A}^{\Delta d}\\
M_{A}^{\tilde{\Delta}z}%
\end{array}
\right]  E^{A}\\
&  =\left(  \left[
\begin{array}
[c]{c}%
C_{A}^{\Delta d}\\
M_{A}^{\tilde{\Delta}z}%
\end{array}
\right]  \left[
\begin{array}
[c]{cc}%
\mathfrak{N}_{~\alpha}^{A0} & h_{~\Gamma}^{A}%
\end{array}
\right]  \right)  \left(  \left[
\begin{array}
[c]{c}%
h_{B}^{\alpha}\\
C_{B}^{\Gamma0}%
\end{array}
\right]  E^{B}\right)  ,\\
&  =\left[
\begin{array}
[c]{ccccc}%
0 & -\tilde{\eta}_{r}^{z} & 0 & N\tilde{m}^{z}\tilde{\eta}_{r}^{w} & 0\\
0 & 0 & \tilde{\eta}_{g}^{[q}\tilde{\eta}_{h}^{z]} & 0 & N\tilde{m}^{z}%
\tilde{\eta}_{g}^{[w}\tilde{\eta}_{h}^{y]}\\
0 & 0 & 0 & \tilde{\eta}_{[s}^{z}\tilde{\eta}_{r]}^{w} & 0\\
0 & 0 & 0 & 0 & \tilde{n}_{d}\varepsilon^{dzab}\tilde{\eta}_{a}^{w}\tilde
{\eta}_{b}^{y}%
\end{array}
\right]  \left[
\begin{array}
[c]{c}%
\tilde{e}_{1}\\
\tilde{e}_{2}\\
\tilde{e}_{3q}\\
\psi_{1w}\\
\psi_{2wy}%
\end{array}
\right]  .
\end{align*}
Furthermore, $\left[
\begin{array}
[c]{c}%
C_{A}^{\Delta d}\\
M_{A}^{\tilde{\Delta}z}%
\end{array}
\right]  E^{A}$ can be written as (\ref{Eq_W_C_M_1}) and its divergence as
(\ref{eq_wav_id_dCE_1}). This leads to the following identity
\begin{align}
&  0=\nabla_{z}\left(  \left[
\begin{array}
[c]{ccccc}%
0 & -\tilde{\eta}_{r}^{z} & 0 & N\tilde{m}^{z}\tilde{\eta}_{r}^{w} & 0\\
0 & 0 & \tilde{\eta}_{g}^{[q}\tilde{\eta}_{h}^{z]} & 0 & N\tilde{m}^{z}%
\tilde{\eta}_{g}^{[w}\tilde{\eta}_{h}^{y]}\\
0 & 0 & 0 & \tilde{\eta}_{[s}^{z}\tilde{\eta}_{r]}^{w} & 0\\
0 & 0 & 0 & 0 & \tilde{n}_{d}\varepsilon^{dzab}\tilde{\eta}_{a}^{w}\tilde
{\eta}_{b}^{y}%
\end{array}
\right]  \left[
\begin{array}
[c]{c}%
\tilde{e}_{1}\\
\tilde{e}_{2}\\
\tilde{e}_{3q}\\
\psi_{1w}\\
\psi_{2wy}%
\end{array}
\right]  \right) \nonumber\\
&  +\left[
\begin{array}
[c]{ccc}%
0 & -\nabla_{z}\left(  2N\tilde{m}^{[z}\tilde{\eta}_{r}^{g]}\right)  &
2N\tilde{m}^{a}\tilde{\eta}_{r}^{b}\\
0 & 0 & -\nabla_{z}\left(  3N\tilde{m}^{[z}\tilde{\eta}_{g}^{a}\tilde{\eta
}_{h}^{b]}\right) \\
0 & -\nabla_{z}\left(  \tilde{\eta}_{s}^{[z}\tilde{\eta}_{r}^{g]}\right)  &
\tilde{\eta}_{s}^{[a}\tilde{\eta}_{r}^{b]}\\
0 & 0 & -\nabla_{z}\left(  \tilde{n}_{d}\varepsilon^{dzab}\right)
\end{array}
\right]  \left[
\begin{array}
[c]{c}%
-\frac{1}{N}\tilde{e}_{1}\\
-\frac{1}{N}\tilde{n}_{g}\tilde{e}_{2}+\psi_{1g}\\
\frac{1}{N}\tilde{e}_{3[a}\tilde{n}_{b]}+\psi_{2ab}%
\end{array}
\right]  , \label{Eq_W_D_C_M_3}%
\end{align}
where it has been using the eqs. (\ref{eq_wav_E_1}), (\ref{eq_wav_E_2}) and
(\ref{eq_wav_E_3}) \ Finally, this expression can be rewritten as%

\begin{align*}
0  &  =\left[
\begin{array}
[c]{cc}%
\tilde{\eta}_{c}^{q} & 0\\
0 & \tilde{\eta}_{s}^{g}\tilde{\eta}_{r}^{h}\\
0 & 0\\
0 & 0
\end{array}
\right]
\mathcal{L}%
_{p}\left[
\begin{array}
[c]{c}%
\psi_{1q}\\
\psi_{2gh}%
\end{array}
\right]  +\left[
\begin{array}
[c]{cc}%
0 & 0\\
0 & 0\\
\tilde{\eta}_{[s}^{f}\tilde{\eta}_{r]}^{q} & 0\\
0 & \tilde{n}_{d}\varepsilon^{dfgh}%
\end{array}
\right]  D_{f}\left[
\begin{array}
[c]{c}%
\psi_{1q}\\
\psi_{2gh}%
\end{array}
\right] \\
&  +\left[
\begin{array}
[c]{cc}%
0 & 0\\
0 & 0\\
0 & \tilde{\eta}_{[s}^{g}\tilde{\eta}_{r]}^{h}\\
0 & 0
\end{array}
\right]  \left[
\begin{array}
[c]{c}%
\psi_{1q}\\
\psi_{2gh}%
\end{array}
\right]  +\left[
\begin{array}
[c]{c}%
-2N\tilde{m}^{a}\tilde{\eta}_{c}^{b}\nabla_{\lbrack a}\left(  \tilde{n}%
_{b]}\frac{1}{N}\tilde{e}_{2}\right)  +\tilde{e}_{3c}\\
+3N\tilde{m}^{f}\tilde{\eta}_{g}^{a}\tilde{\eta}_{h}^{b}\nabla_{\lbrack
f}\left(  \tilde{e}_{3a}\tilde{n}_{b]}\frac{1}{N}\right) \\
-\tilde{\eta}_{s}^{a}\tilde{\eta}_{r}^{c}\nabla_{\lbrack a}\left(  \tilde
{n}_{c]}\frac{1}{N}\tilde{e}_{2}\right) \\
+\tilde{n}_{d}\varepsilon^{dfgh}\nabla_{f}\left(  \tilde{e}_{3g}\tilde{n}%
_{h}\frac{1}{N}\right)
\end{array}
\right]
\end{align*}
In the on-shell case ($e_{1,2,3}=0$), these equations are the subsidiary
system (1st and 2nd line) and the constraints of the constraints (3rd and 4th
line) which are trivially satisfied in the evolution. Namely,
\begin{align}
0  &  =\left[
\begin{array}
[c]{cc}%
\tilde{\eta}_{c}^{q} & 0\\
0 & \tilde{\eta}_{s}^{g}\tilde{\eta}_{r}^{h}\\
0 & 0\\
0 & 0
\end{array}
\right]
\mathcal{L}%
_{p}\left[
\begin{array}
[c]{c}%
\psi_{1q}\\
\psi_{2gh}%
\end{array}
\right]  +\left[
\begin{array}
[c]{cc}%
0 & 0\\
0 & 0\\
\tilde{\eta}_{[s}^{f}\tilde{\eta}_{r]}^{q} & 0\\
0 & \tilde{n}_{d}\varepsilon^{dfgh}%
\end{array}
\right]  D_{f}\left[
\begin{array}
[c]{c}%
\psi_{1q}\\
\psi_{2gh}%
\end{array}
\right] \label{Eq_sub_onda_1}\\
&  +\left[
\begin{array}
[c]{cc}%
0 & 0\\
0 & 0\\
0 & \tilde{\eta}_{[s}^{g}\tilde{\eta}_{r]}^{h}\\
0 & 0
\end{array}
\right]  \left[
\begin{array}
[c]{c}%
\psi_{1q}\\
\psi_{2gh}%
\end{array}
\right]  .\nonumber
\end{align}
where $%
\mathcal{L}%
_{p}=\partial_{t}-%
\mathcal{L}%
_{\beta}$, since we are considering the coordinates $\left(  t,x^{i}\right)
$. \ We note that the evolution equations obtained are quite simple. However,
by adding to these evolution equations terms proportional to the constraints
of the constraints, one can modify this system and obtain a family of
subsidiary systems. This family is obtained by contracting the expression
(\ref{Eq_sub_onda_1}) with the following reduction
\[
N_{\tilde{\Delta}}^{\Gamma}=\left[
\begin{array}
[c]{cccc}%
\tilde{\eta}_{w}^{c} & 0 & N_{1w}^{sr} & N_{2w}\\
0 & \tilde{\eta}_{n}^{s}\tilde{\eta}_{m}^{r} & N_{3nm}^{sr} & N_{4nm}%
\end{array}
\right]  ,
\]
where \ $N_{1w}^{sr},$ $N_{2w},$ $N_{3nm}^{sr}$ and $N_{4nm}$ can be freely
chosen. Of course, different choices of $N_{\tilde{\Delta}}^{\Gamma}$ give
rise to ill/well-posed evolution equations.

Let us now study the characteristic structures of the system and the
subsidiary system. In other words, we will study the Kronecker decomposition
of the pencils $\left[
\begin{array}
[c]{c}%
h_{~A}^{\alpha}\mathfrak{N}_{~\alpha}^{Aq}\\
C_{A}^{\Gamma0}\mathfrak{N}_{~\alpha}^{Aq}%
\end{array}
\right]  l_{q}$ and $\left[
\begin{array}
[c]{c}%
C_{A}^{\Delta d}h_{~\Gamma}^{A}\\
M_{A}^{\tilde{\Delta}z}h_{~\Gamma}^{A}%
\end{array}
\right]  l_{d}$ with $l_{q}=-\lambda n_{a}+k_{a}$ and $k_{a}\left(
\partial_{t}\right)  ^{a}=0$. In their pullback version to $\Sigma_{t}$, these
pencils are
\begin{align}
&  \left[
\begin{array}
[c]{c}%
h_{~A}^{\alpha}\mathfrak{N}_{~\alpha}^{Aq}\\
C_{A}^{\Gamma0}\mathfrak{N}_{~\alpha}^{Aq}%
\end{array}
\right]  l_{q}\phi^{\alpha}\nonumber\\
&  =\left[
\begin{array}
[c]{ccc}%
\left(  -\lambda-\left(  \beta.k\right)  \right)  & 0 & -Nk^{w}\\
0 & \left(  -\lambda-\left(  \beta.k\right)  \right)  & 0\\
-Nk_{s} & 0 & \left(  -\lambda-\left(  \beta.k\right)  \right)  \tilde{\eta
}_{s}^{w}\\
0 & k_{s} & 0\\
0 & 0 & k_{[s}\tilde{\eta}_{y]}^{w}%
\end{array}
\right]  \left[
\begin{array}
[c]{c}%
\tilde{u}^{0}\\
\phi\\
\tilde{u}_{w}%
\end{array}
\right]  , \label{Eq_W_pencil_1_b}%
\end{align}%
\begin{equation}
\left[
\begin{array}
[c]{c}%
C_{A}^{\Delta q}h_{~\Gamma}^{A}\\
M_{A}^{\tilde{\Delta}q}h_{~\Gamma}^{A}%
\end{array}
\right]  l_{q}\psi^{\Gamma}=\left[
\begin{array}
[c]{cc}%
\tilde{\eta}_{r}^{w}\left(  -\lambda-\left(  \beta.k\right)  \right)  & 0\\
0 & \tilde{\eta}_{g}^{[w}\tilde{\eta}_{h}^{y]}\left(  -\lambda-\left(
\beta.k\right)  \right) \\
k_{[s}\tilde{\eta}_{r]}^{w} & 0\\
0 & k_{q}\varepsilon^{qwy}%
\end{array}
\right]  \left[
\begin{array}
[c]{c}%
\psi_{1w}\\
\psi_{2wy}%
\end{array}
\right]  . \label{Eq_W_pencil_sub_1}%
\end{equation}
Where $\varepsilon^{qwy}:=\tilde{n}_{d}\varepsilon^{dqwy}$ is the Levi-Civita
tensor over $\Sigma_{t}$; and all the lowercase indices $w,s,y,g,q$, and any
other lowercase indices that appear until the end of this section, run from
$1$ to $3$. We also note that $\left[
\begin{array}
[c]{c}%
h_{~A}^{\alpha}\mathfrak{N}_{~\alpha}^{Aq}\\
C_{A}^{\Gamma0}\mathfrak{N}_{~\alpha}^{Aq}%
\end{array}
\right]  l_{q}$ $\in%
\mathbb{R}
^{11\times5}$, $\left[
\begin{array}
[c]{c}%
C_{A}^{\Delta d}h_{~\Gamma}^{A}\\
M_{A}^{\tilde{\Delta}z}h_{~\Gamma}^{A}%
\end{array}
\right]  l_{d}\in%
\mathbb{R}
^{10\times6}$, $h_{~A}^{\alpha}\mathfrak{N}_{~\alpha}^{Aq}l_{q}\in%
\mathbb{R}
^{5\times5}$, $C_{A}^{\Gamma0}\mathfrak{N}_{~\alpha}^{Aq}l_{q}\in%
\mathbb{R}
^{6\times5}$, $C_{A}^{\Delta d}h_{~\Gamma}^{A}l_{d}\in%
\mathbb{R}
^{6\times6}$\ \ and $M_{A}^{\tilde{\Delta}z}h_{~\Gamma}^{A}l_{d}\in%
\mathbb{R}
^{4\times6}$ in their pullbacked version.

Let us now study the matrix $\left(  C_{B}^{\Delta0}\mathfrak{N}_{~\alpha
}^{Bq}\right)  l_{q}$, whose expression
\[
\left(  C_{B}^{\Delta0}\mathfrak{N}_{~\alpha}^{Bq}\right)  l_{q}\phi^{\alpha
}=\left[
\begin{array}
[c]{ccc}%
0 & \tilde{\eta}_{s}^{q} & 0\\
0 & 0 & \tilde{\eta}_{s}^{[q}\tilde{\eta}_{y}^{w]}%
\end{array}
\right]  k_{q}\left[
\begin{array}
[c]{c}%
\tilde{u}^{0}\\
\phi\\
\tilde{u}_{w}%
\end{array}
\right]  ,
\]
is obtained from (\ref{Eq_W_pencil_1_b}).

For each $k_{q}$,
\begin{align*}
k_{z}\left(  M_{A}^{\tilde{\Delta}z}h_{~\Delta}^{A}\right)  \left(
C_{B}^{\Delta0}\mathfrak{N}_{~\alpha}^{Bq}\right)  k_{q}  &  =0,\\
k_{z}\left[
\begin{array}
[c]{cc}%
\tilde{\eta}_{[f}^{z}\tilde{\eta}_{r]}^{s} & 0\\
0 & \tilde{n}_{d}\varepsilon^{dzsw}%
\end{array}
\right]  \left[
\begin{array}
[c]{ccc}%
0 & \tilde{\eta}_{s}^{q} & 0\\
0 & 0 & \tilde{\eta}_{s}^{[q}\tilde{\eta}_{y}^{w]}%
\end{array}
\right]  k_{q}  &  =0,
\end{align*}
i.e. $k_{z}\left(  M_{A}^{\tilde{\Delta}z}h_{~\Delta}^{A}\right)  $ expands
the left kernel of $\left(  C_{B}^{\Delta0}\mathfrak{N}_{~\alpha}^{Bq}\right)
k_{q}$ and therefore the condition v) of the theorem
\ref{Theorem_coef_const_2} is satisfied. Furthermore, this shows that
\begin{align}
\dim\left(  left\_\ker\left(  C_{B}^{\Delta0}\mathfrak{N}_{~\alpha}^{Bq}%
k_{q}\right)  \right)   &  =3,\label{Eq_W_lk_1}\\
rank\left(  C_{B}^{\Delta0}\mathfrak{N}_{~\alpha}^{Bq}k_{q}\right)   &
=3,\label{Eq_W_ra_1}\\
\dim\left(  right\_\ker\left(  C_{B}^{\Delta0}\mathfrak{N}_{~\alpha}^{Bq}%
k_{q}\right)  \right)   &  =2. \label{Eq_W_rk_1}%
\end{align}

This result allows us to give the Kronecker structure associated to the Wave
equation, i.e., the Kronecker structure of the pencil (\ref{Eq_W_pencil_1_b}).
It is%
\[
J_{1}\left(  N\sqrt{k.k}-\beta.k\right)  ,J_{1}\left(  -N\sqrt{k.k}%
-\beta.k\right)  ,3\times L_{1}^{T},3\times L_{0}^{T}.
\]
The $3\times L_{1}^{T},3\times L_{0}^{T}$ blocks are justified by
(\ref{Eq_W_lk_1}) and (\ref{Eq_W_ra_1}) as explained in subsubsection
\ref{Kronecker_decomposition}. The Jordan part is justified by giving
explicitly the generalized eigenvectors
\[
\left(  \delta\phi_{\lambda_{1}}\right)  ^{\beta}=\left[
\begin{array}
[c]{c}%
1\\
0\\
-\frac{k_{w}}{\sqrt{\left(  k.k\right)  }}%
\end{array}
\right]  ,\text{ \ \ \ }\left(  \delta\phi_{\lambda_{2}}^{2}\right)  ^{\beta
}=\left[
\begin{array}
[c]{c}%
1\\
0\\
\frac{k_{w}}{\sqrt{\left(  k.k\right)  }}%
\end{array}
\right]  .
\]
They are associated to the generalized eigenvalues $\lambda_{1}=N\sqrt
{k.k}-\beta.k$ and $\lambda_{2}=-N\sqrt{k.k}-\beta.k$. These eigenvalues were
called "the physical" eigenvalues in subsection
\ref{Proof_Theorem_coef_const_2}.

Now, let us study the characteristic structure of $h_{~A}^{\alpha}%
\mathfrak{N}_{~\alpha}^{Aq}l_{q}$, whose explicit expression is
\[
h_{~A}^{\alpha}\mathfrak{N}_{~\alpha}^{Aq}l_{q}\phi^{\alpha}=\left[
\begin{array}
[c]{ccc}%
\left(  -\lambda-\left(  \beta.k\right)  \right)  & 0 & -Nk^{w}\\
0 & \left(  -\lambda-\left(  \beta.k\right)  \right)  & 0\\
-Nk_{s} & 0 & \left(  -\lambda-\left(  \beta.k\right)  \right)  \tilde{\eta
}_{s}^{w}%
\end{array}
\right]  \left[
\begin{array}
[c]{c}%
\tilde{u}^{0}\\
\phi\\
\tilde{u}_{w}%
\end{array}
\right]  .
\]
Since $\left(  \lambda_{1,2},\left(  \delta\phi_{\lambda_{1,2}}\right)
^{\beta}\right)  $ are the generalized eigenvalues and eigenvectors of
$\left[
\begin{array}
[c]{c}%
h_{~A}^{\alpha}\mathfrak{N}_{~\alpha}^{Aq}\\
C_{A}^{\Gamma0}\mathfrak{N}_{~\alpha}^{Aq}%
\end{array}
\right]  l_{q}$,\ \ they are also the eigenvalues and eigenvectors of
$h_{~A}^{\alpha}\mathfrak{N}_{~\alpha}^{Aq}l_{q}$. Moreover, $h_{~A}^{\alpha
}\mathfrak{N}_{~\alpha}^{Aq}l_{q}$ has $3$ more eigenvalues $\pi
_{1,2,3}\left(  k\right)  =-\beta.k$, these were called "the constraints 1"
eigenvalues in subsection \ref{Proof_Theorem_coef_const_2}. \ These
eigenvalues have the associated eigenvectors
\begin{align*}
\left(  \delta\phi_{\pi_{1}}\right)  ^{\beta}  &  =\left[
\begin{array}
[c]{c}%
0\\
1\\
0
\end{array}
\right]  \text{ \ \ }\left(  \delta\phi_{\pi_{2}}\right)  ^{\beta}=\left[
\begin{array}
[c]{c}%
0\\
0\\
v_{1w}%
\end{array}
\right] \\
\left(  \delta\phi_{\pi_{2}}\right)  ^{\beta}  &  =\left[
\begin{array}
[c]{c}%
0\\
0\\
v_{2w}%
\end{array}
\right]  \text{\ }%
\end{align*}
where $v_{1},v_{2}$ are linearly independent vectors, and they are defined by
the following conditions
\[
\left(  v_{1}.v_{2}\right)  =\left(  v_{1,2}.k\right)  =\left(  v_{1,2}%
.\tilde{m}\right)  =0.
\]

Let us now study the Kronecker decomposition of the pencil $\left[
\begin{array}
[c]{c}%
C_{A}^{\Delta q}h_{~\Gamma}^{A}\\
M_{A}^{\tilde{\Delta}q}h_{~\Gamma}^{A}%
\end{array}
\right]  l_{q}$ \ (see eq. (\ref{Eq_W_pencil_sub_1})). This pencil is
associated to the constraints as can be seen from (\ref{Eq_W_D_C_M_3}) and
(\ref{Eq_sub_onda_1}); and its Kronecker decomposition is given by
\[
3\times J_{1}\left(  -\beta.k\right)  ,3\times L_{1}^{T},1\times L_{0}^{T},
\]
as we will show below.

We begin by studying the Jordan blocks. As it was shown in subsubsection
\ref{Kro_sub}, the generalized eigenvectors of this pencil are obtained by
projecting $\left\{  \left(  \delta\phi_{\pi_{1}}\right)  ^{\beta},\left(
\delta\phi_{\pi_{2}}\right)  ^{\beta},\left(  \delta\phi_{\pi_{3}}\right)
^{\beta}\right\}  $ with $C_{B}^{\Gamma0}\mathfrak{N}_{~\beta}^{Bq}k_{q}$,
that is,
\begin{align*}
\left[  \delta\psi_{\pi_{1}}^{\Gamma},\delta\psi_{\pi_{2}}^{\Gamma},\delta
\psi_{\pi_{3}}^{\Gamma}\right]   &  :=C_{B}^{\Gamma0}\mathfrak{N}_{~\beta
}^{Bq}k_{q}\left[  \left(  \delta\phi_{1}\right)  ^{\beta},\left(  \delta
\phi_{\pi_{2}}\right)  ^{\beta},\left(  \delta\phi_{\pi_{3}}\right)  ^{\beta
}\right]  ,\\
&  =\left[
\begin{array}
[c]{ccc}%
0 & k_{s} & 0\\
0 & 0 & k_{[s}\tilde{\eta}_{y]}^{w}%
\end{array}
\right]  \left[
\begin{array}
[c]{ccc}%
0 & 0 & 0\\
1 & 0 & 0\\
0 & v_{1w} & v_{2w}%
\end{array}
\right]  ,\\
&  =\left[
\begin{array}
[c]{ccc}%
k_{s} & 0 & 0\\
0 & k_{[s}v_{1y]} & k_{[s}v_{2y]}%
\end{array}
\right]  .
\end{align*}
Where the eigenvectors are the columns of this matrix and they are associated
to the eigenvalues "the constraints 1" $\pi_{1,2,3}\left(  k\right)
=-\beta.k$.

To complete the Kronecker decomposition, we need to find the
\[
\dim\left(  left\_\ker\left(  M_{A}^{\tilde{\Delta}z}h_{~\Delta}^{A}%
k_{z}\right)  \right)  ,
\]
where $M_{A}^{\tilde{\Delta}z}h_{~\Delta}^{A}k_{z}$ is given by%

\[
M_{A}^{\tilde{\Delta}z}h_{~\Delta}^{A}k_{z}\psi^{\Delta}=\left[
\begin{array}
[c]{cc}%
\tilde{\eta}_{[s}^{z}\tilde{\eta}_{r]}^{w} & 0\\
0 & \varepsilon^{zwy}%
\end{array}
\right]  k_{z}\left[
\begin{array}
[c]{c}%
\psi_{1w}\\
\psi_{2wy}%
\end{array}
\right]  .
\]
Notice that
\[
span\left\langle \left[
\begin{array}
[c]{cc}%
\tilde{n}_{f}k_{f}\varepsilon^{fgsr} & 0
\end{array}
\right]  \right\rangle =left\_\ker\left(  M_{A}^{\tilde{\Delta}z}h_{~\Delta
}^{A}k_{z}\right)
\]
therefore
\[
\dim\left(  left\_\ker\left(  M_{A}^{\tilde{\Delta}z}h_{~\Delta}^{A}%
k_{z}\right)  \right)  =1.
\]
Using this result and the expressions (\ref{Eq_W_lk_1}) and
(\ref{eq_sub_kron_1}), we conclude that the rest of the Kronecker
decomposition is $3\times L_{1}^{T},1\times L_{0}^{T}$.

Finally, we give the basis that diagonalizes matrix $C_{A}^{\Gamma
z}h_{~\Delta}^{A}l_{z}$. This matrix
\[
C_{A}^{\Gamma z}h_{~\Delta}^{A}l_{z}\psi^{\Delta}=\left[
\begin{array}
[c]{cc}%
\tilde{\eta}_{r}^{w}\left(  -\lambda-\left(  \beta.k\right)  \right)  & 0\\
0 & \tilde{\eta}_{g}^{[w}\tilde{\eta}_{h}^{y]}\left(  -\lambda-\left(
\beta.k\right)  \right)
\end{array}
\right]  \left[
\begin{array}
[c]{c}%
\psi_{1w}\\
\psi_{2wy}%
\end{array}
\right]
\]
can be read from (\ref{Eq_W_pencil_sub_1}). By construction, we know that
$\left\{  \delta\psi_{\pi_{1}}^{\Gamma},\delta\psi_{\pi_{2}}^{\Gamma}%
,\delta\psi_{\pi_{3}}^{\Gamma}\right\}  $ are eigenvectors associated to the
eigenvalues $\pi_{1,2,3}\left(  k\right)  =-\beta.k$. We still need to find
$3$ more eigenvectors associated to the "Constraints 2" eigenvalues
$\rho_{1,2,3}\left(  k\right)  $. These eigenvectors $\left\{  \delta
\psi_{\rho_{1}}^{\Delta},\delta\psi_{\rho_{2}}^{\Delta},\delta\psi_{\rho_{3}%
}^{\Delta}\right\}  $ are any $3$ vectors, linearly independent from $\left\{
\delta\psi_{\pi_{1}}^{\Gamma},\delta\psi_{\pi_{2}}^{\Gamma},\delta\psi
_{\pi_{3}}^{\Gamma}\right\}  $.

\section{Conclusions and discussion \label{Conclusions}}

In this paper, we have considered generic systems of first-order partial
differential equations that include differential constraints. We have shown
sufficient conditions for these systems to have first-order partial
differential subsidiary system (SS) with a strongly hyperbolic evolution. This
guarantees the constraint preservation.

We have shown that if the constraints of the system are defined by the Geroch
fields $C_{A}^{a\Gamma}$ and the system admits the integrability conditions
(\ref{eq_int_LE_1}) and (\ref{eq_int_LE_2}), then the SS exists and it is a
set of first-order partial differential equations. Furthermore, we have shown
that when the system only admits Geroch fields $C_{A}^{a\Gamma}$ and
$M_{A}^{\tilde{\Delta}z}$, it implies that the Kronecker structures of the
principal symbol, of the system and of the subsidiary system, does not include
$L_{m}^{T}$ blocks with $m\geq2$. Since most-known physical systems have these
kind of Geroch fields, we conclude that they only have $L_{1}^{T}-$blocks and
vanishing rows in their Kronecker decomposition. On the other hand, this
connection between Geroch fields and the Kronecker structure indicates a
possible extension to first-order partial differential equations,\ of the
classifications presented in the cases of ordinary differential equations
\cite{gantmacher1992theory} and algebraic differential equations
\cite{kunkel2006differential}. In these latter classifications, the Kronecker
structure of the principal symbol is used to find the integrability conditions
(as we have done here) and the solutions of the system. Currently, we are
working on extending the results presented here to systems that admit other
kinds of Geroch fields.

The study of strong hyperbolicity of the SS is performed in the case of
constant coefficients and presented as a continuation of the work
\cite{Abalos:2018uwg}. We have given a complete analysis of the characteristic
structure of the system, showing how can be chosen the propagation velocities
of the physical fields and of the constraints. As in \cite{Abalos:2018uwg},
the analysis is algebraic and pseudo-differential, so the possible non
analyticity in the evolution equations as well as in the subsidiary systems
may introduce causality issues. However, the steps of the proofs presented
here can be readapted to each physical system (including the quasi-linear
ones) avoiding the non analytical pseudo-differential reductions. This
procedure will/may come at the cost of less freedom in the choice of the
propagation velocities of the constraints.

With the new tools developed here, it seems natural to continue with the study
of boundary conditions that guarantee the constraint preservation. We have
made great progress in this direction and we are currently writing an article
where we generalize the ideas of \cite{Calabrese:2001kj},
\cite{Calabrese:2002xy}, \cite{Tarfulea:2013qjq}, \cite{sarbach2012continuum}
to the variable and quasi-linear coefficient cases.

\section*{Acknowledgements}
\addcontentsline{toc}{section}{Acknowledgement}

I would like to thank Federico Carrasco for many interesting comments and
suggestions. I am also grateful to Carlos Olmos, David Hilditch and Oscar
Reula for the discussions and ideas exchanged throughout this work. This
research work was partially supported by CONICET.

\begin{appendices}
\section{Lemmas\label{Ap_lemmas}}

In this appendix, we introduce some lemmas about the Kronecker decomposition
of a pencil. These results are used in the proof of Theorem SH of the SS.

Consider the following matrix pencil
\begin{equation}
\lambda I_{\beta}^{A}+K_{\beta}^{A}=\lambda\left[
\begin{array}
[c]{c}%
-\delta_{\beta}^{\alpha}\\
0
\end{array}
\right]  +\left[
\begin{array}
[c]{c}%
A_{\beta}^{\alpha}\\
C_{\beta}^{\Gamma}%
\end{array}
\right]  , \label{pencil_1}%
\end{equation}
such that $\delta_{\beta}^{\alpha},A_{\beta}^{\alpha}\in%
\mathbb{R}
^{u\times u}$, $C_{\beta}^{\Gamma}\in%
\mathbb{R}
^{c\times u}$, $\delta_{\beta}^{\alpha}$ is the identity matrix, $I_{\beta
}^{A}=\left[
\begin{array}
[c]{c}%
-\delta_{\beta}^{\alpha}\\
0
\end{array}
\right]  $ and $K_{\beta}^{A}=\left[
\begin{array}
[c]{c}%
A_{\beta}^{\alpha}\\
C_{\beta}^{\Gamma}%
\end{array}
\right]  $.

Since $I_{\beta}^{A}$ has only trivial right kernel, the Kronecker
decomposition of this pencil may only include Jordan blocks (with $\lambda
_{i}$ as the generalized eigenvalues)
\[
J_{m}\left(  \lambda_{i}\right)  =\left[
\begin{array}
[c]{cccc}%
\left(  \lambda-\lambda_{i}\right)  & 1 & 0 & 0\\
0 & \left(  \lambda-\lambda_{i}\right)  & 1 & 0\\
0 & 0 & ... & 1\\
0 & 0 & 0 & \left(  \lambda-\lambda_{i}\right)
\end{array}
\right]  \in%
\mathbb{C}
^{m\times m},
\]
$L_{m}^{T}$ blocks%
\begin{equation}
L_{m}^{T}=\left[
\begin{array}
[c]{ccccc}%
\lambda & 0 & 0 & 0 & 0\\
1 & \lambda & 0 & 0 & 0\\
0 & 1 & ... & 0 & 0\\
0 & 0 & ... & \lambda & 0\\
0 & 0 & 0 & 1 & \lambda\\
0 & 0 & 0 & 0 & 1
\end{array}
\right]  \in%
\mathbb{C}
^{\left(  m+1\right)  \times m}, \label{bloq_LmT_1}%
\end{equation}
and zero rows (called here $L_{0}^{T}$).

\begin{lemma}
\label{Lemma_1}a) Let
\begin{equation}
J,m_{0}\times L_{0}^{T},m_{1}\times L_{1}^{T},m_{2}\times L_{2}^{T}%
,...,m_{n}\times L_{n}^{T} \label{estruc_1}%
\end{equation}
be the Kronecker structure of $\lambda I_{\beta}^{A}+K_{\beta}^{A}$, where
$J\in%
\mathbb{C}
^{a\times a}$ includes only the Jordan blocks, then%
\begin{align}
m_{0}+m_{1}+m_{2}+...+m_{n}  &  =c,\label{eq_mi_c_1}\\
a+m_{1}+2m_{2}+...+nm_{n}  &  =u. \label{eq_mi_c_2}%
\end{align}

b) For any $\lambda$ other than the generalized eigenvalues, it holds that
\begin{equation}
\dim\left(  left\_\ker\left(  \lambda I_{\beta}^{A}+K_{\beta}^{A}\right)
\right)  =c. \label{eq_dim_left_ker_1}%
\end{equation}

c) Let (\ref{estruc_1}) with
\[
J=d_{1}\times J_{1}\left(  \lambda_{1}\right)  ,d_{2}\times J_{1}\left(
\lambda_{2}\right)  ,...,d_{q}\times J_{1}\left(  \lambda_{q}\right)  ,
\]
and $\lambda_{1}<...<\lambda_{q}$ be the Kronecker structure of $\lambda
I_{\beta}^{A}+K_{\beta}^{A}$, then
\begin{align}
\dim\left(  left\_\ker\left(  \lambda I_{\beta}^{A}+K_{\beta}^{A}\right)
\right)   &  =c\text{ \ \ \ }\forall\lambda\neq\lambda_{i}\text{ with
}i=1,...,q,\label{eq_lk_pen_1}\\
\dim\left(  left\_\ker\left(  \left.  \lambda I_{\beta}^{A}+K_{\beta}%
^{A}\right\vert _{\lambda=\lambda_{i}}\right)  \right)   &  =c+d_{i}.
\label{eq_lk_pen_2}%
\end{align}

\end{lemma}

\begin{proof}
a) The pencil $\lambda I_{\beta}^{A}+K_{\beta}^{A}$ is a $(u+c)\times u$
matrix so its Kronecker matrix is of the same size. Therefore, the number of
columns of $J\in%
\mathbb{C}
^{a\times a}$ plus that of the $L_{m}^{T}\in%
\mathbb{C}
^{m+1\times m}$ blocks in (\ref{estruc_1}) has to be equal to $u$, i.e.,
\[
a+m_{1}+2m_{2}+...+nm_{n}=u.
\]
This concludes the proof of eq. (\ref{eq_mi_c_2}).

The same analysis on the rows shows that
\[
a+m_{0}+2m_{1}+3m_{2}+...+\left(  n+1\right)  m_{n}=u+c.
\]
Subtracting these last two expressions, we obtain the eq. (\ref{eq_mi_c_1})
and conclude the proof of (a). Notice that (\ref{eq_mi_c_1}) is independent of
the size of $J$.

b) Studying the left kernel of $\lambda I_{\beta}^{A}+K_{\beta}^{A}$ is
equivalent to studying the left kernel of its Kronecker structure, so we
assume that $\lambda I_{\beta}^{A}+K_{\beta}^{A}$ is already in its Kronecker form.

We begin by noting that the Jordan blocks only have non-trivial left kernel
when $\lambda=\lambda_{i}$. Since we are not considering these values of
$\lambda$, we conclude that to show (\ref{eq_dim_left_ker_1}), it is
sufficient to study the left kernel of the blocks $L_{m}^{T}$ and the zero
rows $L_{0}^{T}$.

The left kernel of each block $L_{m}^{T}$ (eq. (\ref{bloq_LmT_1})) has
dimension 1 and is expanded by the vector
\[
X=\left[  -1,\lambda,...,\left(  -1\right)  ^{m}\lambda^{m-1},\left(
-1\right)  ^{m+1}\lambda^{m}\right]  \in%
\mathbb{R}
^{1\times m+1}.
\]
Therefore, if we assume without loss of generality, that the Kronecker
structure of the pencil is given by (\ref{estruc_1}) we conclude that the left
kernel of the pencil has dimension
\[
m_{0}+m_{1}+m_{2}+...+m_{n}.
\]
Using the equation (\ref{eq_mi_c_1}) of a) we conclude the proof of b).

c) The proof of equation (\ref{eq_lk_pen_1}) is the same as the proof of b),
since b) is independent of the form of $J$.

The proof of eq. (\ref{eq_lk_pen_2}) is followed by noting that all Jordan
blocks in $J$ are of $1\times1$ and that each generalized eigenvalue
$\lambda_{i}$ has degeneracy $d_{i}$; therefore, when $\lambda=\lambda_{i}$
there are $d_{i}$ extra vectors added to the left kernel of the pencil
justifying the expression (\ref{eq_lk_pen_2}).
\end{proof}

\bigskip

The following lemma reveals the connection between $C_{\beta}^{\Gamma}$ and
the Kronecker structure of the pencil (\ref{pencil_1}).

We call
\begin{align*}
d  &  :=\dim\left(  right\_\ker\left(  C_{\beta}^{\Gamma}\right)  \right)  ,\\
r  &  :=rank\left(  C_{\beta}^{\Gamma}\right)  ,\\
s  &  :=\dim\left(  left\_\ker\left(  C_{\beta}^{\Gamma}\right)  \right)  ,
\end{align*}
and recall that by the rank--nullity theorem%
\begin{align}
u  &  =r+d,\label{eq_u_r_d_1}\\
c  &  =r+s. \label{eq_c_r_s_1}%
\end{align}

\begin{lemma}
\label{Lemma_2}Let the pencil (\ref{pencil_1}) such that part of its Kronecker
structure is%
\begin{equation}
d_{1}\times J_{1}\left(  \lambda_{1}\right)  ,d_{2}\times J_{1}\left(
\lambda_{2}\right)  ,...,d_{q}\times J_{1}\left(  \lambda_{q}\right)  ,
\label{eq_d1_J1_1}%
\end{equation}
with
\begin{equation}
d_{1}+d_{2}+...+d_{q}=d. \label{eq_d1_d2_d_1}%
\end{equation}
Then the complete Kronecker structure of this pencil is
\begin{equation}
d_{1}\times J_{1}\left(  \lambda_{1}\right)  ,d_{2}\times J_{1}\left(
\lambda_{2}\right)  ,...,d_{q}\times J_{1}\left(  \lambda_{q}\right)  ,s\times
L_{0}^{T},r\times L_{1}^{T}. \label{eq_penc_kro_1}%
\end{equation}

\end{lemma}

\begin{proof}
We begin by showing that if the conditions (\ref{eq_d1_J1_1}) and
(\ref{eq_d1_d2_d_1}) are satisfied, the system does not admit any other Jordan
blocks than those present in (\ref{eq_d1_J1_1}). Consider the set of
generalized eigenvectors $\delta\phi_{j}^{\beta\lambda_{i}}$ associated to the
structure (\ref{eq_d1_J1_1}), they compound the set $\left\{  \delta\phi
_{j}^{\beta\lambda_{i}}\right\}  $ of $d$ linearly independent vectors with
$j=1,...,d_{i}$ and $i=1,...,q$ \ such that
\begin{equation}
\left(  \lambda_{i}\left[
\begin{array}
[c]{c}%
-\delta_{\beta}^{\alpha}\\
0
\end{array}
\right]  +\left[
\begin{array}
[c]{c}%
A_{\beta}^{\alpha}\\
C_{\beta}^{\Gamma}%
\end{array}
\right]  \right)  \delta\phi_{j}^{\beta\lambda_{i}}=0. \label{eq_aut_gen_1}%
\end{equation}

This means that
\[
C_{\beta}^{\Gamma}\delta\phi_{j}^{\beta\lambda_{i}}=0\text{ \ \ }\forall
j=1,..d\text{ and }\forall i=1,...,q.
\]
Notice that any other possible generalized eigenvector satisfying
(\ref{eq_aut_gen_1}) should also belong to the right kernel of $C_{\beta
}^{\Gamma}$. But since $d=\dim\left(  right\_\ker\left(  C_{\beta}^{\Gamma
}\right)  \right)  $, the entire right kernel of $C_{\beta}^{\Gamma}$ must be
expanded by the $d$ vectors $\left\{  \delta\phi_{j}^{\beta\lambda_{i}%
}\right\}  $. In other words, the pencil cannot admit any extra Jordan block,
otherwise, we will reach a contradiction.

It only remains to show that the complete pencil structure is given by
(\ref{eq_penc_kro_1}). We call $J$ the part of the Jordan blocks
\[
J=d_{1}\times J_{1}\left(  \lambda_{1}\right)  ,d_{2}\times J_{1}\left(
\lambda_{2}\right)  ,...,d_{q}\times J_{1}\left(  \lambda_{q}\right)
\]
and we notice that $J$ is a $d\times d$ matrix, so using the equation
(\ref{eq_mi_c_2}) of the previous lemma, we obtain
\[
d+m_{1}+2m_{2}+...+nm_{n}=u.
\]

Since by (\ref{eq_u_r_d_1}), $u=r+d$, we conclude that
\begin{equation}
m_{1}+2m_{2}+...+nm_{n}=r. \label{eq_sum_r_1}%
\end{equation}

On the other hand, since $\dim\left(  left\_\ker\left(  C_{\beta}^{\Gamma
}\right)  \right)  =s$, the pencil $\lambda I_{\beta}^{A}+K_{\beta}^{A}$ has
$s$ linearly independent vectors $\left[
\begin{array}
[c]{cc}%
0 & v_{\Gamma}^{i}%
\end{array}
\right]  $, with $i=1,...,s$, belonging to the left kernel of the pencil.
These vectors do not depend on $\lambda$ since $C_{\beta}^{\Gamma}$ does not.
Moreover, due to the explicit form of $\lambda I_{\beta}^{A}+K_{\beta}^{A}$,
every other vector of the left kernel, independent of $\lambda$ should be a
linear combination of the previous ones. Therefore, the set of vectors
$\left[
\begin{array}
[c]{cc}%
0 & v_{\Gamma}^{i}%
\end{array}
\right]  $\ are associated to $s$ null rows in the Kronecker decomposition of
$\lambda I_{\beta}^{A}+K_{\beta}^{A}$, i.e.,
\[
m_{0}=s.
\]

Replacing this expression in (\ref{eq_sum_r_1}) and (\ref{eq_c_r_s_1}) we
obtain
\begin{align*}
c  &  =r+s\\
&  =m_{1}+2m_{2}+...+nm_{n}+m_{0}.
\end{align*}

Using now the eq. (\ref{eq_mi_c_1}), we conclude
\[
m_{0}+m_{1}+m_{2}+...+m_{n}=c=m_{1}+2m_{2}+...+nm_{n}+m_{0}%
\]
and, therefore,
\[
0=m_{2}+2m_{3}+...+\left(  n-1\right)  m_{n}.
\]
Since $m_{i}\geq0$ the unique solution for this equation is
\[
m_{i}=0\text{ para }i\geq2.
\]

Therefore,%
\[
c=r+s=m_{1}+m_{0}=m_{1}+s,
\]
from which we conclude that $m_{1}=r$, ending the proof.
\end{proof}

\section{Lapse and shift\label{App_coordenadas}}

In this appendix, we consider a space-time $M$, with the foliation $\Sigma
_{t}$ described in subsection \ref{n+1_decomposition_sec}, and we study the
equation
\begin{equation}
\nabla_{a}q^{b}=0, \label{Eq_ap_sis_1}%
\end{equation}
where $\nabla_{a}$ is any covariant derivative without torsion. We present
this simple system as an example of the process of rewriting the equations in
their $n+1$ version. Any more complex system can be handled in the same way by
repeating the steps presented here. We perform this $n+1$ decomposition by
using a different projector $\tilde{\eta}_{b}^{a}$ to the one $\eta_{b}^{a}$
used in section \ref{Seccion_Setting_1} and assuming that the system does not
have a background metric. This new projector is parametrized by the lapse
function $N$ and the shift vector $\beta^{a}.$ Of course, we recover the
standard results in the cases with a\ background metric \ 

We remark that theorems \ref{teorema_ec_ev_vin_1} and
\ \ref{Theorem_coef_const_2} can be re-adapted to the type of projection used
here. This is showing in the examples of section \ref{Examples}.

We begin by considering the definitions introduced in subsection
\ref{n+1_decomposition_sec}, and defining the shift vector $\beta^{a}$ such
that
\[
\beta^{a}n_{a}=0.
\]
It allows us to define the projector
\begin{equation}
\tilde{\eta}_{b}^{a}:=\delta_{b}^{a}-p^{a}n_{b}, \label{Eq_proy_b_1}%
\end{equation}
with
\[
p^{a}=t^{a}-\beta^{a}.
\]
Notice that this projector reduces to the $\eta_{b}^{a}$ projector when
$\beta^{a}=0$.

Using the coordinates $\left(  t,x^{i}\right)  $ adapted to the foliation
introduced in subsection \ref{n+1_decomposition_sec} we obtain
\[
\tilde{\eta}_{b}^{a}=\left[
\begin{array}
[c]{cccc}%
0 & 0 & 0 & 0\\
\beta^{1} & 1 & 0 & 0\\
\beta^{2} & 0 & 1 & 0\\
\beta^{3} & 0 & 0 & 1
\end{array}
\right]  ,
\]
where the indices $a$ and $b$ correspond to the columns and the rows of the
matrix respectively.

On the other hand, since
\[
p^{a}n_{a}=t^{a}n_{a}-\beta^{a}n_{a}=t^{a}n_{a}=1,
\]
it is easy to check that $\tilde{\eta}_{b}^{a}$ satisfies similar projector
properties as $\eta_{c}^{b}$, that is,
\[%
\begin{array}
[c]{ccccc}%
\tilde{\eta}_{b}^{a}\tilde{\eta}_{c}^{b}=\tilde{\eta}_{c}^{a}, &  &
\tilde{\eta}_{b}^{a}p^{b}=0, &  & \tilde{\eta}_{b}^{a}n_{a}=0.
\end{array}
\]

We introduce now the lapse function $N$ in the following way, consider the
projector
\begin{align}
\tilde{\eta}_{b}^{a}  &  :=\delta_{b}^{a}+\frac{1}{N}p^{a}\left(
-Nn_{b}\right)  ,\nonumber\\
&  =\delta_{b}^{a}+\tilde{m}^{a}\tilde{n}_{b}, \label{Eq_proy_b_3}%
\end{align}
where we have defined%
\begin{align*}
\tilde{m}^{a}  &  :=\frac{1}{N}p^{a},\\
\tilde{n}_{b}  &  :=-Nn_{b}\text{,}%
\end{align*}
such that
\begin{equation}
\tilde{m}^{a}\tilde{n}_{a}=-1. \label{Eq_nor_m_n_1}%
\end{equation}

With these definitions, we can project the eq. (\ref{Eq_ap_sis_1}) and its
variables as follows.

\begin{lemma}%
\begin{align}
\nabla_{a}q^{b}  &  =-\tilde{m}^{b}n_{a}\left(
\mathcal{L}%
_{p}\tilde{q}^{0}-N\left(  \left(  \tilde{q}^{w}S_{w}\right)  +\tilde{q}%
^{0}\left(  \tilde{m}^{d}Z_{d}\right)  \right)  \right)  \label{Eq_apendix_T}%
\\
&  +\tilde{\eta}_{w}^{b}n_{a}\left(
\mathcal{L}%
_{p}\tilde{q}^{w}-N\left(  \tilde{q}^{r}K_{~r}^{w}+\tilde{q}^{0}S^{w}\right)
\right) \\
&  +\tilde{\eta}_{w}^{b}\tilde{\eta}_{a}^{d}\left(  D_{d}\tilde{q}^{w}%
+\tilde{q}^{0}K_{~d}^{w}\right) \\
&  -\tilde{m}^{b}\tilde{\eta}_{a}^{d}\left(  D_{d}\tilde{q}^{0}+\tilde{q}%
^{w}K_{wd}-\tilde{q}^{0}Z_{d}\right)
\end{align}
with $\tilde{q}^{w},S^{r},$ $S_{r}$, $K_{ba}$ and $K_{~a}^{b}$ tangents to
$\Sigma_{t}$ and such that
\begin{align}
\tilde{q}^{w}  &  :=\tilde{\eta}_{r}^{w}q^{r},\label{Eq_apendix_T_1}\\
\tilde{q}^{0}  &  :=\tilde{n}_{r}q^{r}, \label{Eq_apendix_T_2}%
\end{align}%
\begin{equation}
Z_{d}:=\tilde{n}_{w}\nabla_{d}\tilde{m}^{w}, \label{Eq_apendix_T_3}%
\end{equation}%
\begin{align}
S_{r}  &  :=\tilde{\eta}_{r}^{w}\tilde{m}^{d}\nabla_{d}\tilde{n}%
_{w},\label{Eq_apendix_T_4}\\
S^{r}  &  :=\tilde{\eta}_{w}^{r}\tilde{m}^{d}\nabla_{d}\tilde{m}^{w},
\label{Eq_apendix_T_5}%
\end{align}%
\begin{align}
-K_{~a}^{b}  &  :=\tilde{\eta}_{w}^{b}\tilde{\eta}_{a}^{d}\nabla_{d}\tilde
{m}^{w},\label{Eq_apendix_T_6}\\
-K_{ba}  &  :=\tilde{\eta}_{a}^{d}\tilde{\eta}_{b}^{r}\nabla_{d}\tilde{n}_{r}
\label{Eq_apendix_T_7}%
\end{align}
The covariant derivative $D_{d}$ is defined over $\Sigma_{t}$ in the standard
form
\begin{equation}
D_{d}\tilde{q}^{w}:=\tilde{\eta}_{r}^{w}\tilde{\eta}_{d}^{s}\nabla_{s}%
\tilde{q}^{w}, \label{Eq_apendix_T_8}%
\end{equation}
and the Lie derivative $%
\mathcal{L}%
_{p}$ is%
\begin{align}%
\mathcal{L}%
_{p}  &  =%
\mathcal{L}%
_{\partial_{t}}-%
\mathcal{L}%
_{\beta},\label{Eq_apendix_T_9}\\
&  =\partial_{t}-%
\mathcal{L}%
_{\beta},
\end{align}
where the last equation holds since we are considering the coordinates
$\left(  t,x^{i}\right)  .$

Moreover, it holds
\begin{equation}
S_{r}=\tilde{\eta}_{r}^{w}\left(  D_{w}\left(  \ln N\right)  -Z_{w}\right)
\label{Eq_apendix_T_10}%
\end{equation}
and
\begin{align}%
\mathcal{L}%
_{\tilde{m}}\tilde{\eta}_{r}^{d}  &  =\tilde{m}^{d}D_{r}\left(  \ln N\right)
,\label{Eq_apendix_T_11}\\%
\mathcal{L}%
_{p}\tilde{\eta}_{r}^{d}  &  =0 \label{Eq_apendix_T_12}%
\end{align}

\end{lemma}

\begin{proof}
We start by isolating $\delta_{a}^{d}$ from equations (\ref{Eq_proy_b_1}) and
(\ref{Eq_proy_b_3}) and rewriting (\ref{Eq_ap_sis_1}) as follows
\begin{align*}
\nabla_{a}q^{b}  &  =\delta_{a}^{d}\delta_{q}^{b}\nabla_{d}\left(  \delta
_{c}^{b}q^{c}\right)  ,\\
&  =\left(  \tilde{\eta}_{a}^{d}+p^{d}n_{a}\right)  \left(  \tilde{\eta}%
_{q}^{b}+p^{b}n_{q}\right)  \nabla_{d}\left(  \left(  \tilde{\eta}_{c}%
^{b}-\tilde{m}^{b}\tilde{n}_{c}\right)  q^{c}\right)  .
\end{align*}
Defining
\begin{align*}
\tilde{q}^{0}  &  :=\tilde{n}_{c}q^{c},\\
\tilde{q}^{b}  &  :=\tilde{\eta}_{c}^{b}q^{c}%
\end{align*}
and introducing these definitions in the last expression we obtain%
\begin{align*}
\nabla_{a}q^{b}  &  =-\tilde{m}^{b}n_{a}\left(  p^{d}\tilde{n}_{w}\left(
\nabla_{d}\tilde{q}^{w}-\tilde{q}^{0}\nabla_{d}\tilde{m}^{w}-\tilde{m}%
^{w}\nabla_{d}\tilde{q}^{0}\right)  \right) \\
&  +\tilde{\eta}_{w}^{b}n_{a}\left(  p^{d}\left(  \nabla_{d}\tilde{q}%
^{w}-\tilde{q}^{0}\nabla_{d}\tilde{m}^{w}-\tilde{m}^{w}\nabla_{d}\tilde{q}%
^{0}\right)  \right) \\
&  +\tilde{\eta}_{w}^{b}\tilde{\eta}_{a}^{d}\left(  \nabla_{d}\tilde{q}%
^{w}-\tilde{q}^{0}\nabla_{d}\tilde{m}^{w}-\tilde{m}^{w}\nabla_{d}\tilde{q}%
^{0}\right) \\
&  -\tilde{m}^{b}\tilde{\eta}_{a}^{d}\left(  \tilde{n}_{w}\left(  \nabla
_{d}\tilde{q}^{w}-\tilde{q}^{0}\nabla_{d}\tilde{m}^{w}-\tilde{m}^{w}\nabla
_{d}\tilde{q}^{0}\right)  \right)  .
\end{align*}
Using now that
\begin{align*}
\tilde{\eta}_{w}^{b}\tilde{m}^{w}  &  =0,\\
\tilde{n}_{w}\tilde{q}^{w}  &  =0,\\
\tilde{n}_{w}\nabla_{d}\tilde{q}^{w}  &  =-\tilde{q}^{w}\nabla_{d}\tilde
{n}_{w},\\
\tilde{n}_{w}\tilde{m}^{w}  &  =-1,\\
\tilde{n}_{w}\nabla_{d}\tilde{m}^{w}  &  =-\tilde{m}^{w}\nabla_{d}\tilde
{n}_{w},
\end{align*}
we conclude%
\begin{align*}
\nabla_{a}q^{b}  &  =-\tilde{m}^{b}n_{a}\left(  -p^{d}\left(  \left(
\tilde{q}^{w}-\tilde{q}^{0}\tilde{m}^{w}\right)  \nabla_{d}\tilde{n}%
_{w}\right)  +p^{d}\nabla_{d}\tilde{q}^{0}\right) \\
&  +\tilde{\eta}_{w}^{b}n_{a}\left(  p^{d}\left(  \nabla_{d}\tilde{q}%
^{w}-\tilde{q}^{0}\nabla_{d}\tilde{m}^{w}\right)  \right) \\
&  +\tilde{\eta}_{w}^{b}\tilde{\eta}_{a}^{d}\left(  \nabla_{d}\tilde{q}%
^{w}-\tilde{q}^{0}\nabla_{d}\tilde{m}^{w}\right) \\
&  -\tilde{m}^{b}\tilde{\eta}_{a}^{d}\left(  -\left(  \left(  \tilde{q}%
^{w}-\tilde{q}^{0}\tilde{m}^{w}\right)  \nabla_{d}\tilde{n}_{w}\right)
+\nabla_{d}\tilde{q}^{0}\right)  .
\end{align*}
Using also that
\begin{align*}
D_{a}\tilde{q}^{b}  &  :=\tilde{\eta}_{w}^{b}\tilde{\eta}_{a}^{d}\nabla
_{d}\tilde{q}^{w}=\tilde{\eta}_{w}^{b}\tilde{\eta}_{a}^{d}D_{d}\tilde{q}%
^{w},\\
D_{a}\tilde{q}^{0}  &  =\tilde{\eta}_{a}^{d}\nabla_{d}\tilde{q}^{0}%
=\tilde{\eta}_{a}^{d}D_{d}\tilde{q}^{0},\\%
\mathcal{L}%
_{p}\tilde{q}^{0}  &  =p^{d}\nabla_{d}\tilde{q}^{0},\\%
\mathcal{L}%
_{p}\tilde{q}^{w}  &  =p^{a}\nabla_{a}\tilde{q}^{w}-\tilde{q}^{a}\nabla
_{a}p^{w}\rightarrow%
\mathcal{L}%
_{p}\tilde{q}^{w}+\tilde{q}^{a}\nabla_{a}p^{w}=p^{a}\nabla_{a}\tilde{q}^{w},
\end{align*}
we conclude%
\begin{align*}
\nabla_{a}q^{b}  &  =-\tilde{m}^{b}n_{a}\left(  -p^{d}\left(  \tilde{q}%
^{w}-\tilde{q}^{0}\tilde{m}^{w}\right)  \nabla_{d}\tilde{n}_{w}+%
\mathcal{L}%
_{p}\tilde{q}^{0}\right) \\
&  +\tilde{\eta}_{w}^{b}n_{a}\left(
\mathcal{L}%
_{p}\tilde{q}^{w}+\tilde{q}^{r}\nabla_{r}p^{w}-\tilde{q}^{0}p^{d}\nabla
_{d}\tilde{m}^{w}\right) \\
&  +\tilde{\eta}_{w}^{b}\tilde{\eta}_{a}^{d}\left(  D_{d}\tilde{q}^{w}%
-\tilde{q}^{0}\nabla_{d}\tilde{m}^{w}\right) \\
&  -\tilde{m}^{b}\tilde{\eta}_{a}^{d}\left(  -\left(  \tilde{q}^{w}-\tilde
{q}^{0}\tilde{m}^{w}\right)  \nabla_{d}\tilde{n}_{w}+D_{d}\tilde{q}%
^{0}\right)  .
\end{align*}
Recalling that
\[
\tilde{m}^{w}=\frac{1}{N}p^{w},
\]
we rewrite $\nabla_{r}p^{w}$ as follows
\[
\nabla_{r}p^{w}=\tilde{m}^{w}\nabla_{r}N+N\nabla_{r}\tilde{m}^{w}.
\]
Replacing this expression in the previous development we arrive at the
following result%
\begin{align*}
\nabla_{a}q^{b}  &  =-\tilde{m}^{b}n_{a}\left(  -\left(  \tilde{q}^{w}%
-\tilde{q}^{0}\tilde{m}^{w}\right)  p^{d}\nabla_{d}\tilde{n}_{w}+%
\mathcal{L}%
_{p}\tilde{q}^{0}\right) \\
&  +\tilde{\eta}_{w}^{b}n_{a}\left(
\mathcal{L}%
_{p}\tilde{q}^{w}+\tilde{q}^{r}N\nabla_{r}\tilde{m}^{w}-\tilde{q}^{0}%
p^{d}\nabla_{d}\tilde{m}^{w}\right) \\
&  +\tilde{\eta}_{w}^{b}\tilde{\eta}_{a}^{d}\left(  D_{d}\tilde{q}^{w}%
-\tilde{q}^{0}\nabla_{d}\tilde{m}^{w}\right) \\
&  -\tilde{m}^{b}\tilde{\eta}_{a}^{d}\left(  -\left(  \tilde{q}^{w}-\tilde
{q}^{0}\tilde{m}^{w}\right)  \nabla_{d}\tilde{n}_{w}+D_{d}\tilde{q}%
^{0}\right)  .
\end{align*}
Finally, introducing the definitions
\begin{align*}
-K_{~w}^{b}  &  :=\tilde{\eta}_{w}^{b}\tilde{\eta}_{a}^{d}\nabla_{d}\tilde
{m}^{w},\\
-K_{ew}  &  :=\tilde{\eta}_{w}^{d}\tilde{\eta}_{e}^{r}\nabla_{d}\tilde{n}%
_{r},\\
Z_{d}  &  :=\tilde{n}_{w}\nabla_{d}\tilde{m}^{w}=-\tilde{m}^{w}\nabla
_{d}\tilde{n}_{w},\\
S_{r}  &  :=\tilde{\eta}_{r}^{w}\tilde{m}^{d}\nabla_{d}\tilde{n}_{w},\\
S^{r}  &  :=\tilde{\eta}_{w}^{r}\tilde{m}^{d}\nabla_{d}\tilde{m}^{w},
\end{align*}
we obtain%
\begin{align*}
\nabla_{a}q^{b}  &  =-\tilde{m}^{b}n_{a}\left(
\mathcal{L}%
_{p}\tilde{q}^{0}-N\left(  \left(  \tilde{q}^{w}S_{w}\right)  +\tilde{q}%
^{0}\left(  \tilde{m}^{d}Z_{d}\right)  \right)  \right) \\
&  +\tilde{\eta}_{w}^{b}n_{a}\left(
\mathcal{L}%
_{p}\tilde{q}^{w}-\tilde{q}^{r}NK_{~r}^{w}-N\tilde{q}^{0}S^{w}\right) \\
&  +\tilde{\eta}_{w}^{b}\tilde{\eta}_{a}^{d}\left(  D_{d}\tilde{q}^{w}%
+\tilde{q}^{0}K_{~d}^{w}\right) \\
&  -\tilde{m}^{b}\tilde{\eta}_{a}^{d}\left(  D_{d}\tilde{q}^{0}+\tilde{q}%
^{w}K_{wd}-\tilde{q}^{0}Z_{d}\right)
\end{align*}

In addition, we show that
\begin{align*}
\tilde{m}^{d}\nabla_{d}\tilde{n}_{b}  &  =-\tilde{m}^{d}\nabla_{d}\left(
N\nabla_{b}t\right)  =-\tilde{m}^{d}\nabla_{d}N\nabla_{b}t-\tilde{m}%
^{d}N\nabla_{d}\left(  \nabla_{b}t\right)  ,\\
&  =\frac{1}{N}\tilde{n}_{b}\tilde{m}^{d}\nabla_{d}N-\tilde{m}^{d}N\nabla
_{b}\left(  \nabla_{d}t\right)  ,\\
&  =\frac{1}{N}\tilde{n}_{b}\tilde{m}^{d}\nabla_{d}N+\tilde{m}^{d}N\nabla
_{b}\left(  \frac{1}{N}\tilde{n}_{d}\right)  ,\\
&  =\frac{1}{N}\tilde{n}_{b}\tilde{m}^{d}\nabla_{d}N-N\nabla_{b}\left(
\frac{1}{N}\right)  +\tilde{m}^{d}\nabla_{b}\tilde{n}_{d},\\
&  =\frac{1}{N}\tilde{n}_{b}\tilde{m}^{d}\nabla_{d}N+\frac{1}{N}\nabla
_{b}N+\tilde{m}^{d}\nabla_{b}\tilde{n}_{d},\\
&  =\frac{1}{N}\tilde{\eta}_{b}^{d}\nabla_{d}N+\tilde{m}^{d}\nabla_{b}%
\tilde{n}_{d},\\
&  =D_{b}\left(  \ln N\right)  -Z_{b},
\end{align*}
and
\begin{align*}%
\mathcal{L}%
_{\tilde{m}}\tilde{\eta}_{r}^{d}  &  =\tilde{m}^{q}\nabla_{q}\tilde{\eta}%
_{r}^{d}-\tilde{\eta}_{r}^{q}\nabla_{q}\tilde{m}^{d}+\tilde{\eta}_{q}%
^{d}\nabla_{r}\tilde{m}^{q},\\
&  =\tilde{m}^{q}\nabla_{q}\left(  \tilde{m}^{d}\tilde{n}_{r}\right)  -\left(
\delta_{r}^{q}+\tilde{m}^{q}\tilde{n}_{r}\right)  \nabla_{q}\tilde{m}%
^{d}+\left(  \delta_{q}^{d}+\tilde{m}^{d}\tilde{n}_{q}\right)  \nabla
_{r}\tilde{m}^{q},\\
&  =\tilde{n}_{r}\tilde{m}^{q}\nabla_{q}\tilde{m}^{d}+\tilde{m}^{d}\tilde
{m}^{q}\nabla_{q}\tilde{n}_{r}-\tilde{n}_{r}\tilde{m}^{q}\nabla_{q}\tilde
{m}^{d}+\tilde{m}^{d}\tilde{n}_{q}\nabla_{r}\tilde{m}^{q},\\
&  =\tilde{m}^{d}\tilde{m}^{q}\nabla_{q}\tilde{n}_{r}+\tilde{m}^{d}\tilde
{n}_{q}\nabla_{r}\tilde{m}^{q},\\
&  =\tilde{m}^{d}\left(  \tilde{m}^{q}\nabla_{q}\tilde{n}_{r}+Z_{r}\right)
,\\
&  =\tilde{m}^{d}D_{r}\left(  \ln N\right)  .
\end{align*}
Notice that $%
\mathcal{L}%
_{p}\tilde{\eta}_{r}^{d}$ can be obtained from the latter expression by taking
$N=1$, so that
\[%
\mathcal{L}%
_{p}\tilde{\eta}_{r}^{d}=0
\]

\end{proof}

\bigskip

Let us now consider the case when we have a Lorentzian metric $g_{ab}$, such
that in coordinates $\left(  t,x^{i}\right)  $ it has the following form%
\begin{equation}
ds^{2}=\left(  -N^{2}+\beta_{k}\beta^{k}\right)  dt^{2}+2\beta_{i}%
dtdx^{i}+\gamma_{ij}dx^{i}dx^{j}. \label{Eq_met_1}%
\end{equation}
Here $N$ and $\beta^{k}$ have been defined as before. In its matrix mode
$g_{ab}$ and its inverse $g^{bc}$ are given by%
\[
g_{ab}=\left[
\begin{array}
[c]{cc}%
-N^{2}+\beta_{k}\beta^{k} & \beta_{j}\\
\beta_{j} & \gamma_{ij}%
\end{array}
\right]  ,
\]%
\[
g^{bc}=\left[
\begin{array}
[c]{cc}%
-\frac{1}{N^{2}} & \frac{\beta^{j}}{N^{2}}\\
\frac{\beta^{j}}{N^{2}} & \gamma^{ij}-\frac{\beta^{i}\beta^{j}}{N^{2}}%
\end{array}
\right]  .
\]
With this ansatz, it is easy to check that%
\[
N=\frac{1}{\sqrt{-\nabla t.\nabla t}},
\]
and if we use the metric $g_{ab}$ (and its inverse) to raise and lower indices
we obtain that
\[
\tilde{m}^{b}=g^{bc}\tilde{n}_{c}.
\]
Notice that the latter expression is equivalent to%
\[
\tilde{n}^{b}=\tilde{m}^{b}%
\]
and furthermore $\tilde{n}_{c}$ is temporal due to the eq. (\ref{Eq_nor_m_n_1}%
), i.e.,%
\[
\tilde{n}^{b}\tilde{n}_{b}=\tilde{m}^{b}\tilde{n}_{b}=-1.
\]
Moreover, we see that the projector $\tilde{\eta}_{b}^{a}$ (eq.
(\ref{Eq_proy_b_1})), takes the form%
\begin{equation}
\tilde{\eta}_{b}^{a}:=\delta_{b}^{a}+\tilde{n}^{a}\tilde{n}_{b}.
\label{eq_proy_b}%
\end{equation}
When we lower the $a$ component with $g_{ac}$, the latter expression is a
Riemannian metric on $\Sigma_{t}$, it is
\begin{equation}
\tilde{\eta}_{ab}=g_{ab}+\tilde{n}_{a}\tilde{n}_{b}. \label{eq_3_metrica}%
\end{equation}

Notice also that when we do not consider a metric we have 2 unrelated types of
"extrinsic curvatures" $K_{~a}^{b}$ and $K_{ba}$. However, when we have a
metric like (\ref{Eq_met_1}) and we consider a Levi-Civita connection
$\nabla_{d}$ then
\begin{align*}
K_{ba}  &  =g_{bc}K_{~a}^{c},\\
S_{r}  &  =g_{rc}S^{c},
\end{align*}
furthermore
\[
Z_{d}:=\tilde{n}_{w}\nabla_{d}\tilde{m}^{w}=\tilde{n}_{w}\nabla_{d}\tilde
{n}^{w}=\frac{1}{2}\nabla_{d}\left(  \tilde{n}.\tilde{n}\right)  =0.
\]

\end{appendices}
\bigskip

%
%


\bibliographystyle{plain}
\bibliography{Constraint_preservation_Abalos}




\end{document}